\documentclass[a4paper,11pt]{amsart}
\usepackage{mathptmx}
\usepackage{amsmath}
\usepackage{amscd}
\usepackage{amssymb}
\usepackage{amsthm}
\usepackage{xspace}
\usepackage[all,tips]{xy}
\usepackage[dvips]{graphicx}
\usepackage{verbatim}
\usepackage{syntonly}
\usepackage{hyperref}
\usepackage{float}
\usepackage{amsmath, amsthm, graphics, amssymb,fullpage,color, epsfig,url, wrapfig, ulem}
\usepackage{indentfirst}
\usepackage{esint}
\usepackage{tikz}
\usetikzlibrary{calc}
\usetikzlibrary{arrows, arrows.meta}

\tikzset{
solid node/.style={circle,draw,inner sep = 1.5,fill=black},
hollow node/.style={circle,draw,inner sep=1.5}
}

\providecommand{\MR}{\relax\ifhmode\unskip\space\fi MR }

\providecommand{\href}[2]{#2}

\theoremstyle{plain}
\newtheorem{thm}{Theorem}[section]
\newtheorem{lem}[thm]{Lemma}
\newtheorem{prop}[thm]{Proposition}
\newtheorem{defn}[thm]{Definition}

\newtheorem{cor}[thm]{Corollary}

\newtheorem*{examples}{Examples}

\theoremstyle{definition}
\newtheorem{rem}{Remark}[section]

\newcommand{\disp}{\displaystyle}

\DeclareMathOperator{\dist}{dist}
\DeclareMathOperator{\diam}{diam}

\DeclareMathOperator{\Lip}{Lip}

\DeclareMathOperator{\sgn}{sgn}
\DeclareMathOperator{\Cur}{Cur}
\DeclareMathOperator{\NCur}{NCur}
\DeclareMathOperator{\proj}{proj}
\DeclareMathOperator{\Fav}{Fav}
\DeclareMathOperator{\FavC}{\Fav_{\mathcal{C}}}

\newcommand{\eps}{\varepsilon}
\newcommand{\vp}{\varphi}

\newcommand{\al}{\alpha}
\newcommand{\be}{\beta}
\newcommand{\ga}{\gamma}
\newcommand{\de}{\delta}
\newcommand{\De}{\Delta}

\newcommand{\te}{\theta}
\newcommand{\la}{\lambda}

\newcommand{\om}{\omega}
\newcommand{\Om}{\Omega}

\newcommand{\nid}{\noindent}

\newcommand{\iny}{\infty}
\newcommand{\del}{ \partial}
\newcommand{\su}{\subset}

\newcommand{\norm}[1]{\left\| #1\right\|}

\newcommand{\abs}[1]{\left\vert#1\right\vert}
\newcommand{\set}[1]{\left\{#1\right\}}
\newcommand{\brac}[1]{\left[#1\right]}
\newcommand{\pr}[1]{\left( #1 \right) }
\newcommand{\pb}[1]{\left( #1 \right] }

\newcommand{\N}{\ensuremath{\mathbb{N}}}

\newcommand{\R}{\ensuremath{\mathbb{R}}}
\newcommand{\Z}{\ensuremath{\mathbb{Z}}}

\numberwithin{equation}{section}
\begin{document}

\title{A Quantification of a Besicovitch Nonlinear Projection Theorem \\ via Multiscale Analysis }
\author[Davey, Taylor]{Blair Davey \and Krystal Taylor}
\address{Blair Davey, Department of Mathematical Sciences, Montana State University}
\email{blairdavey@montana.edu}
\thanks{Davey is supported in part by the Simons Foundation Grant 430198.}
\address{Krystal Taylor, Department of Mathematics, The Ohio State University}
\email{taylor.2952@osu.edu}
\thanks{Taylor is supported in part by the Simons Foundation Grant 523555.}

\date{}


\begin{abstract}
The Besicovitch projection theorem states that if a subset $E$ of the plane has finite length in the sense of Hausdorff measure and is purely unrectifiable (so its intersection with any Lipschitz graph has zero length), then almost every orthogonal projection of $E$ to a line will have zero measure.
In other words, the Favard length of a purely unrectifiable $1$-set vanishes.
In this article, we show that when linear projections are replaced by certain nonlinear projections called \textit{curve projections}, this result remains true.
In fact, we go further and use multiscale analysis to prove a quantitative version of this Besicovitch nonlinear projection theorem.
Roughly speaking, we show that if a subset of the plane has finite length in the sense of Hausdorff and is nearly purely unrectifiable, then its \textit{Favard curve length} is very small.
Our techniques build on those of Tao, who in \cite{Tao09} proves a quantification of the original Besicovitch projection theorem.
\end{abstract}

\maketitle

\hypersetup{linktocpage}
  \setcounter{tocdepth}{1}
\tableofcontents

\section{Introduction and Main Results}

The \textit{Favard length} of a set $E\subset \mathbb{R}^2$ is defined as a rescaled average length of its orthogonal projections. 
That is,
$$\Fav(E) = \int_{\mathbb{S}^1} \abs{\proj_{\om}(E)} d\om,$$
where $\proj_{\om}(E)$ denotes the linear projection of a set $E$ onto the angle $\om\in \mathbb{S}^1 := [0, 2\pi)$.
Specifically,  for a point $\pr{x, y} \in \R^2$, $\proj_\om(x,y) = x\cos{\om} + y \sin{\om}$.
Here and throughout, we use $|\cdot|$ to denote the ($1$-dimensional) Lebesgue measure.  
The Favard length arises in a number of central questions in geometric measure theory, and its study is closely tied to that of rectifiability and  analytic capacity (see, for instance, \cite{EidVol}). 

The Besicovitch projection theorem provides a direct link between the rectifiability of a set and its Favard length.
We refer the reader to \cite[Theorem 6.13]{Fal85} and \cite[Theorem 18.1]{Mat95} for the proof of this theorem, as well as higher-dimensional analogues.  
We use the notation $\mathcal{H}^1(E)$ to denote the $1$-dimensional Hausdorff measure of a set $E$.  
A set $E \su \R^2$ is called \textit{purely $1$-unrectifiable} if for every Lipschitz function $f : \R \to \R^2$, it holds that $\mathcal{H}^1\pr{E \cap f(\R)} = 0$.  

\begin{thm}[Besicovitch Projection Theorem]
\label{Besi_thm}
Let $E \su \R^2$ be such that $\mathcal{H}^1(E) \in (0, \infty)$.  
Then $E$ is purely $1$-unrectifiable if and only if $\Fav(E) = 0$.  
\end{thm}

In this paper, we use multiscale analysis to obtain upper bounds on the \textit{Favard curve length} for nearly unrectifiable sets.  
The quantitative version of rectifiability introduced by Tao in \cite{Tao09} is used to describe what we mean by a set being ``nearly unrectifiable", while the Favard curve length is defined using the nonlinear projection maps introduced in \cite{ST17}, \cite{STint}, and \cite{CDT21}. 
More precisely, in Theorem~\ref{QCBT}, we prove that an upper bound on the rectifiability constant given in Definition~\ref{rect_const_defn} translates to an upper bound on the Favard curve length described by Definition~\ref{FavC}. 
As applications of this theorem, we 
\begin{enumerate}
\item recover the qualitative Besicovitch projection theorem in this nonlinear setting; and 
\item obtain a bound on the rate of decay for the Favard curve length of the $n^{\text{th}}$ generation in the construction of the four-corner Cantor set.  
\end{enumerate}
The upper bound described by (2) is by no means optimal, but we include this result as an example of the utility of the main theorem.
In fact, a much faster rate of decay, as well as a lower bound, is obtained in \cite{CDT21}.
For a qualitative Besicovitch projection theorem for nonlinear families of mappings satisfying a transversality condition, see \cite{HovJ2Led}.

\subsection{Motivation}
\label{Buffon}
We consider a probabilistic interpretation of Favard length known as the Buffon needle problem.  
Let $E \su \brac{0, 1}^2$.
The \textit{Buffon needle problem} asks the probability that a needle, or a line, that is dropped at random onto the plane intersects the set $E$ given that it intersects $\brac{0,1}^2$.
We define this probability as
$$\mathbf{P} := P\pr{\ell \cap E \ne \emptyset : \ell \textrm{ is any line in $\R^2$ for which } \ell \cap \brac{0,1}^2 \ne \emptyset}.$$
If we parametrize all such lines by letting $\ell_{\be, \om}$ denote the line passing through $\pr{0, \be}$ with direction orthogonal to $\om$, then
\begin{align*}
\mathbf{P} 
&\simeq \abs{\set{(\be, \om) \in \R \times \mathbb{S}^1 : E \cap \ell_{\be, \om} \ne \emptyset}}.
\end{align*}
Upon fixing $\om \in \mathbb{S}^1$, we see that
\begin{align*}
\set{\be \in \R : E \cap \ell_{\be, \om} \ne \emptyset} \simeq \proj_{\om}\pr{E},
\end{align*}
where $\proj_{\om}\pr{S}$ denotes the linear projection of a set $S$ onto the angle $\om$.
By Fubini's theorem, we see that
\begin{align}
\mathbf{P}
&\simeq \int_{\mathbb{S}^1} \abs{\set{\be \in \R : E \cap \ell_{\be, \om} \ne \emptyset}} d\om
\simeq \int_{\mathbb{S}^1} \abs{\proj_{\om}(E)} d\om
=: \Fav(E).
\end{align}
Therefore, the Favard length is connected to the classical Buffon needle problem.

To motivate the introduction of the \textit{Favard curve length}, we now we ask what happens when lines are replaced by more general curves.
Let $\mathcal{C}$ denote a curve in $\R^2$.
We want to calculate the probability that $\mathcal{C}$ intersects $E$ when it is dropped randomly onto the plane and intersects $\brac{0,1}^2$.
Define
\begin{align*}
\mathbf{P}_\mathcal{C}
&:= P\pr{\mathcal{C}_e \cap E \ne \emptyset : \mathcal{C}_e \textrm{ is a translation of $\mathcal{C}$ for which } \mathcal{C}_e \cap \brac{0,1}^2 \ne \emptyset}.
\end{align*}
Then
$$\mathbf{P}_\mathcal{C} \simeq \abs{\set{(\al, \be) \in \R^2 : E \cap \pr{(\al, \be) + \mathcal{C}} \ne \emptyset}}.$$
Observe that $E \cap \pr{(\al, \be) + \mathcal{C}} \ne \emptyset$ iff $(\al, \be) \in E - \mathcal{C}$.

Associated to the curve $\mathcal{C}$ in $\R^2$ is a family of nonlinear projections that we call \textit{curve projections}, $\Phi_\al : \R^2 \to \mathcal{P}(\R)$, where $\mathcal{P}(\R)$ denotes the power set of $\R$.
For each $\al \in \R$ and $p \in \R^2$, $\Phi_\al(p)$ is the set of $y$-coordinates of the intersection of  $p - \mathcal{C}$ with the line $x = \al$.
That is, 
\begin{equation}
\label{PaProj}
\Phi_\al(p) = \set{\be \in \R : (\al, \be) \in \pr{p - \mathcal{C}} \cap \set {x = \al}}.
\end{equation}
The inverse map $\Phi_\al^{-1} : \R \to \mathcal{P}(\R^2)$ is given by 
\begin{equation}
\label{PaProjIn}
\Phi_\al^{-1}(\be) = (\al, \be) + \mathcal{C}.
\end{equation}

With this new notation, it follows that
\begin{align*}
\mathbf{P}_\mathcal{C}
&\simeq \abs{\set{(\al, \be) \in \R^2 : E \cap \Phi_\al^{-1}(\be) \ne \emptyset}}.
\end{align*}
And for each fixed $\al \in \R$, we have
\begin{align*}
\set{\be \in \R :  E \cap \Phi_\al^{-1}(\be) \ne \emptyset}
&= \Phi_\al(E).
\end{align*}
As above, an application of Fubini's theorem shows that
\begin{equation}\label{probFav}
\mathbf{P}_\mathcal{C}
\simeq \int_\R \abs{\set{\be : E \cap \Phi_\al^{-1}(\be) \ne \emptyset}} d\al
= \int_\R \abs{\Phi_\al(E)} d\al
=: \FavC(E).
\end{equation}
Therefore, in this nonlinear case, the Favard curve length is proportional to the probability associated to the so-called \textit{Buffon curve problem}.

The expression in \eqref{probFav} is also equivalent to the measure of the Minkowki difference set $E -\mathcal{C}$.  
That is,  
$$\FavC(E) = \int_\R \abs{\Phi_\al(E)} d\al \simeq \abs{E -\mathcal{C}}.$$
This observation is explained in detail in \cite{CDT21}; see also \cite{ST17} and \cite{STint}, where such sum sets are studied.

\subsection{Projections and Favard length}\label{Projections and Favard length}

For a curve $\mathcal{C}$ in $\R^2$, we define a family of curve projections $\Phi_\al : \R^2 \to \mathcal{P}(\R)$ by \eqref{PaProj} with inverse given by \eqref{PaProjIn}.
We now formalize the definition of the Favard curve length.

\begin{defn}[Favard curve length]
\label{FavC}
Let $E \su \R^2$ and let $\mathcal{C}$ be some curve in $\R^2$.
We define the {\bf Favard curve length} as 
\begin{equation*}
\FavC(E)
= \abs{\set{(\al, \be) \in \R^2 : \Phi_\al^{-1}(\be) \cap E \ne \emptyset}}
= \int_{\R} \abs{\Phi_\al(E)} d\al.
\end{equation*}
If $\mathcal{E} \su \R^3$, then with $E_\al = \set{e \in \R^2 : \pr{e, \al} \in \mathcal{E}}$, the {\bf Favard curve length} is given  by
\begin{equation*}
\FavC(\mathcal{E}) 
= \abs{\set{(\al, \be) \in \R^2 : \pr{\Phi_\al^{-1}(\be) \times \set{\al}} \cap \mathcal{E} \ne \emptyset}}
= \int_{\R} \abs{\Phi_\al(E_\al)} d\al.
\end{equation*}
\end{defn}

Although we defined $\Phi_\al$ to be the set of $y$-values of the intersection of $p - \mathcal{C}$ with a vertical line defined by $x = \al$, the equivalence between the quantities in Definition~\ref{FavC} still holds for any other choice of orthonormal basis.
For example, we could define $\Psi_\be$ to be the set of $x$-values of the intersection of $p - \mathcal{C}$ with a horizontal line $y = \be$, and then we would compute the Favard curve length by integrating over $\be$.

\subsection{Hausdorff measure and rectifiability}
\label{defn section}

Some additional notions that we require include the Hausdorff measure, rectifiability, as well as the rectifiability constant of a set.
We build the Hausdorff measure via the restricted Hausdorff content, which will be useful when we work with the multiscale analysis.

\begin{defn}[Restricted Hausdorff content; Hausdorff measure]
Let $E \su \R^2$ and let $0 \le r_- < r_+$.
The {\bf restricted Hausdorff content} $\mathcal{H}^1_{r_-, r_+}(E)$ is defined as
\begin{equation*}
\mathcal{H}^1_{r_-, r_+}(E) = \inf \sum_{B \in \mathcal{B}} \diam(B),
\end{equation*}
where the infimum ranges over all at most countable collections $\mathcal{B}$ of open balls $B$ with radius $r(B) \in \brac{r_-, r_+}$ that cover $E$.
The $1$-dimensional {\bf Hausdorff measure} is then defined as
$$\mathcal{H}^1(E) = \lim_{r_+ \to 0}\mathcal{H}^1_{0, r_+}(E).$$
\end{defn}

\begin{defn}[Rectifiability; Unrectifiability]
A set $E \su \R^2$ is said to be {\bf $1$-rectifiable} if there exists a countable collection $\set{f_i}$ of Lipschitz curves $f_i : \R \to \R^2$ such that
$$\mathcal{H}^1\pr{E \setminus \cup f_i(\R)} = 0.$$
Conversely, $E$ is called {\bf purely $1$-unrectifiable} if for every Lipschitz function $f : \R \to \R^2$, it holds that $\mathcal{H}^1\pr{E \cap f(\R)} = 0$.
\end{defn}

To present the statement of our theorem, we first need to quantify  the notion of rectifiability.
Thus, we want to find a way to measure how much of a given set $E$ is covered, in an appropriate sense, by Lipschitz curves. 
To do this, we follow Tao's definition \cite[Definition 1.10]{Tao09}.

\begin{defn}[Rectifiability constant]
\label{rect_const_defn}
Let $E \su \R^2$ be a set, and let $\eps, r, M > 0$.
The {\bf rectifiability constant} $R_E(\eps, r, M)$ of $E$ is defined by
\begin{equation*}
R_E(\eps, r, M) = \sup \frac{\abs{\set{x \in J: x \, \om_1 + \pr{F(x) + y} \om_2 \in E \text{ for some } \abs{y} \le \eps}}}{\abs{J}},
\end{equation*}
where the supremum ranges over all orthonormal pairs $\om_1$, $\om_2$ in $\mathbb{S}^1$, all Lipschitz functions $F : \R \to \R$ with Lipschitz constant bounded above by $M$, and all intervals $J \su \R$ for which $\abs{J} \ge r$.
\end{defn}

\begin{rem}
In contrast to above, here we use $\mathbb{S}^1$ to denote a set of unit $2$-vectors.
The use of $\mathbb{S}^1$, whether it denotes angles or vectors in a given direction, should be clear from the context.
\end{rem}

We see that for any $E \su \R^2$ and any choice of parameters, $R_E(\eps, r, M) \in \brac{0, 1}$.
To gain some intuition for this definition, we consider some examples.

\begin{examples} \;
\begin{enumerate}
\item Let $F : \R \to \R$ be a Lipschitz function and set $E = \set{\pr{x, F(x)} : x \in \brac{0, 1}}$.
For any $\eps > 0$, $r \le 1$, and $M \ge \Lip\pr{F}$, we have $R_E(\eps, r, M) = 1$.
\item Let $N\in \mathbb{N}$ be even with $N >10$.
Set $\eps = \frac{1}{2N}$. 
For $i = 1, \ldots, N$, let $x_i = \pr{\frac{2i-1}{2N}, (-1)^i} \in \R^2$. Define $\disp E =\bigcup_{i=1}^N \partial B_{\eps}(x_i)$, where $\partial B_i$ denotes the boundary of $B_i$.
Note that the projection of $E$ onto the $x$-axis is the full interval $[0,1]$.  
However, for any $M \le \frac N{10}$, $ R_E(\eps, 1, M) \le \frac{1}{2}$.
\item Let $E \su \R^2$ be purely unrectifiable.
As shown in \cite[Proposition 1.11]{Tao09},  for every choice of $r$ and $M$, $\disp \lim_{\eps \to 0} R_E(\eps, r, M) = 0$.
Therefore, for any $\de > 0$, there exists $\eps > 0$ so that $R_E(\eps, r, M) \le \de$.
\end{enumerate}
\end{examples}

In conclusion, if $E$ is almost purely unrectifiable, then $R_E(\eps, r, M)$ should be near $0$.
And conversely, if $E$ is almost rectifiable, then we expect $R_E(\eps, r, M)$ to be near $1$.

\subsection{Theorem statement}

We now present the statement of our main result, which can be
compared to \cite[Theorem 1.13]{Tao09}

\begin{thm}[Quantitative Besicovitch Nonlinear Projection Theorem]
\label{QCBT}
Let $E \su \brac{0,1}^2$ be a compact set for which $\mathcal{H}^1(E) \le L$ for some $L \in (0, \iny)$.
Assume that for some sufficiently large $N \in \N$, there is a sequence of scales
\begin{equation}
0 < r_{N}^- \le r_{N}^+ < \cdots < r_{1}^- \le r_{1}^+ \le 1
\end{equation}
satisfying the following properties:
\begin{enumerate}
\item[-] Uniform length bound: For all $n = 1, 2, \ldots, N$, 
\begin{equation}
\label{Ulb}
\mathcal{H}^1_{r_{n}^-, r_{n}^+}(E) \le L
\end{equation}
\item[-] Separation of scales : For all $n = 1, 2, \ldots, N-1$, 
\begin{equation}
\label{Sos}
r_{n+1}^+ \le \tfrac 1 2 r_{n}^-
\end{equation}
\item[-] Near unrectifiability: For all $n = 1, 2, \ldots, N-2$, 
\begin{equation}
\label{unrect}
R_E\pr{r_{n+2}^+, r_{n}^-, \frac{1}{r_{n}^-}} \le N^{-1/100}.
\end{equation}
\end{enumerate}
If $\mathcal{C}$ is a piecewise $C^1$ curve of finite length with a piecewise bilipschitz continuous unit tangent vector, then
$$\FavC(E) \lesssim N^{-1/100} L.$$
\end{thm}

\begin{rem}
For convenience, we will assume that $N^{1/100} \in \N$.
Assuming that $N^{1/100} \ge 3$ will suffice.
\end{rem}

\subsection{Applications}

Before proceeding to a discussion of the proof of this quantitative Besicovitch nonlinear projection theorem, 
 we present two applications of the theorem.

\subsubsection{Application $\#$1} First, we demonstrate how the quantitative result implies the following qualitative version of the theorem.

\begin{thm}[Qualitative Besicovitch Nonlinear Projection Theorem]
Let $E \su \brac{0,1}^2$ be a compact set for which $\mathcal{H}^1(E) <\infty$. 
Assume that $E$ is purely unrectifiable.
If $\mathcal{C}$ is a piecewise $C^1$ curve of finite length with a piecewise bilipschitz continuous unit tangent vector, then $\FavC(E) = 0$.
\end{thm}

\begin{proof}
To apply Theorem~\ref{QCBT}, we need a sequence of scales that satisfies the uniform length bounds, separation of scales, and near unrectifiability.
Fix some $N\in \N$ so that $N^{1/100} \geq 3$.
By Example (3) above, since $E$ is purely unrectifiable, then for any $\de, r, M > 0$, there exists 
\begin{equation}
\label{eps0Defn}
\eps_0 = \eps_0\pr{\de, r, M} > 0
\end{equation}
so that whenever $\eps \le \eps_0$, it holds that $R_E\pr{\eps, r, M} \le \de$.
We will choose $\de = N^{-1/100}$.

For notational convenience, set $r_{0}^- = r_{-1}^- = 1$.
We recursively define each $r_{n}^\pm$ for $n= 1, \ldots, N$, starting from $r_{1}^\pm$, as follows.
Define
$$r_n = \min\set{\frac 1 2 r_{n-1}^-, \eps_0\pr{N^{-1/100}, r_{n-2}^-, \frac 1 {r_{n-2}^-}}},$$
where $\eps_0$ is as defined in \eqref{eps0Defn}.
Since
$$\mathcal{H}^1(E) = \lim_{r_+ \to 0} \mathcal{H}^1_{0, r_+}(E),$$
then there exists $r_{n}^+ \in \pb{0, r_n}$ so that
\begin{equation}
\label{r+Choice}
\abs{\mathcal{H}^1_{0, r_{n}^+}(E) - \mathcal{H}^1(E)} < \frac {\mathcal{H}^1(E)} 2.
\end{equation}
Since $E$ is compact, then 
$$\mathcal{H}^1_{0, r_{n}^+}(E) = \lim_{r_- \to 0}\mathcal{H}^1_{r_-, r_{n}^+}(E).$$
Thus, there exists $r_{n}^- \in (0, r_{n}^+]$ so that 
\begin{equation}
\label{r-Choice}
\abs{\mathcal{H}^1_{r_{n}^-, r_{n}^+}(E) - \mathcal{H}^1_{0, r_{n}^+}(E)} < \frac {\mathcal{H}^1(E)} 2.
\end{equation}
We continue this process until $r_{N}^\pm$ have been defined.

It follows from the triangle inequality, \eqref{r+Choice}, and \eqref{r-Choice}, that for all $1 \le n \le N$, 
\begin{align*}
\mathcal{H}^1_{r_{n}^-, r_{n}^+}(E)
&\le \abs{\mathcal{H}^1_{r_{n}^-, r_{n}^+}(E) - \mathcal{H}^1_{0, r_{n}^+}(E)}
+ \abs{\mathcal{H}^1_{0, r_{n}^+}(E) - \mathcal{H}^1(E)}
+ \mathcal{H}^1(E)
< 2 \mathcal{H}^1(E).
\end{align*}
In particular, we have the required uniform length bounds.
Since $r_{n}^+ \le r_n \le \frac 1 2 r_{n-1}^-$ for all $1 < n \le N$, then we also have separation of scales.
And because $r_{n}^+ \le r_n \le \eps_0\pr{N^{-1/100}, r_{n-2}^-, \frac 1 {r_{n-2}^-}}$, where $\eps_0$ is as defined in \eqref{eps0Defn}, then for all $2 < n \le N$,
$$R_E\pr{r_{n}^+, r_{n-2}^-, \frac 1 {r_{n-2}^-}} \le N^{-1/100}.$$ 
This shows that near unrectifiability is satisfied as well.

Theorem~\ref{QCBT} now implies that $\FavC(E) \lesssim 2 N^{-1/100} \mathcal{H}^1(E)$.
Since we may repeat this process for any $N \in \N$ sufficiently large, then we can show that for any $\eps > 0$, $\FavC(E) < \eps$.
In particular, $\FavC(E) = 0$.
\end{proof}

\subsubsection{Application $\#$2}

For the second application, we use Theorem~\ref{QCBT} to estimate the rate of decay of the Favard curve length of the four-corner Cantor set.
That is, we establish upper bounds for each $\FavC(K_n)$, where $K_n$ denotes the $n^{\text{th}}$ generation.

First, we formally define the four-corner Cantor set in the plane.
We start by describing the middle-half Cantor set in the real line, denoted by $C$.
For any $n \in \N \cup \set{0}$, let $C_n$ denote the $n^{\text{th}}$ generation of the set $C$.
Then $C_0 = \brac{0,1}$ and for any $n \in \N$, 
$$
C_n=\bigcup_{\substack{a_j\in\{0, 3\}\\j=1, \ldots, n}}\brac{\sum_{j=1}^na_j4^{-j}, \sum_{j=1}^na_j4^{-j}+4^{-n}}.
$$
For example, $\disp C_1 = \brac{0, \tfrac 1 4} \cup \brac{\tfrac 3 4, 1}$, the set that is obtained by removing the middle half of $C_0$.
Each $C_{n+1}$ is obtained through the self-similar process of removing the middle half of all intervals that comprise $C_n$.
We define $\disp C  = \bigcap_{n=0}^\iny C_n$, the \textit{middle-half Cantor set}.
Then the \textit{four-corner Cantor set} is the product set given by $K=C\times C$. 
This means that the $n^{\text{th}}$ generation of $K$ is given by 
\begin{equation}
\label{KnDefn}
K_n = C_n \times C_n,
\end{equation}
so we may realize the four-corner Cantor set as $\disp K  = \bigcap_{n=0}^\iny K_n$.

As each $K_n$ is a $2$-set, Theorem~\ref{QCBT} may not be applied with $E = K_n$.
Thus, we define a $1$-set associated to each $K_n$ by taking its boundary.
That is, set $E_n = \del K_n$.
As we will see below, an upper bound on the Favard curve length of each $E_n$ automatically implies the same bound for the curve length of $K_n$.

To apply Theorem~\ref{QCBT} to each $E_n$, we need an upper bound for the rectifiability constants of each $E_n$.
In \cite{Tao09}, the bounds for these constants are proved through a quantitative two-projection theorem. 
We rely on the following corollary to a result of Tao:

\begin{prop}[Rectifiability constant for $E_n = \del K_n$, Corollary to Proposition 1.20 from \cite{Tao09}]
\label{rCEsts}
Let $n \ge m > \ell \ge 0$.
Define $E_n = \del K_n$.
If $1 \le M \le c \brac{\log\pr{m - \ell + 1}}^{1/100}$ for some sufficiently small fixed constant $c > 0$, then
$$R_{E_n}\pr{2^{-m}, 2^{-\ell}, M} \lesssim \brac{\log\pr{m - \ell + 1}}^{-1/100}.$$
\end{prop}

This result follows from the proof of \cite[Proposition 1.20]{Tao09} combined with the fact that $E_n$ is a $1$-set for which $\text{proj}_\om(K_n) = \text{proj}_\om(E_n)$ for any $\om \in \mathbb{S}^1$.

Now we use the previous proposition in combination with Theorem~\ref{QCBT} to produce an upper bound for the Favard curve length of $K_n$.
We use the notation $\log_*$ to denote the inverse tower function defined by
$$\log_* x = \min \set{m \ge 0 : \log^{(m)}x \le 1}.$$

\begin{thm}[Rate of decay for $K_n$, cf. Proposition 1.21 in \cite{Tao09}]
\label{Rate curve Kn}
If $n \gg 1$ and $\mathcal{C}$ is a piecewise $C^1$ curve of finite length with a piecewise bilipschitz continuous unit tangent vector, then
$$\FavC(K_n) \lesssim \pr{\log_* n}^{-1/100}.$$
\end{thm}

Our proof follows \cite[Proposition 1.21]{Tao09}, but we include the details here for completeness.

\begin{proof}
As above, we define $E_n = \del K_n$ and note that $E_n$ is a $1$-set with $\FavC(K_n) = \FavC(E_n)$.
Therefore, it suffices to prove that $\FavC(E_n) \lesssim \pr{\log_* n}^{-1/100}$.

Let $N = \log_* n/C_1 \in \N$ for some sufficiently large constant $C_1$ that will be specified below.
That is, $n = e^{e^{e^{\cdot^{\cdot^{\cdot^{e}}}}}}$, where the tower contains $C_1 N$ elements.
Then we define an increasing sequence $\set{m_j}_{j=1}^N \su \N$ recursively by choosing $\log N \lesssim m_1 \lesssim N$, then setting $m_{j+1} = \lceil2^{C_2 m_j^{100}}\rceil$, where $C_2 = \pr{100 \log 2}^{-1}$.
The starting point $m_1$ is chosen so that $m_{j+1} - m_j \ge N$ for all $j = 1, \ldots, N-1$.
Taking a closer look:
\begin{align*}
m_2 &\approx 2^{C_2 m_1^{100}} = \exp\pr{\tilde C \exp\pr{M}} \\
m_3 &\approx \exp\pr{\tilde C m_2^{100}} \approx \exp\pr{\tilde C  \exp\pr{\exp\pr{M}}}  \\
m_4 &\approx \exp\pr{\tilde C m_3^{100}} \approx \exp\pr{\tilde C \exp\pr{\exp\pr{\exp\pr{M}}} } \\
&\vdots
\end{align*}
where $M = 100 \log m_1$ and $\tilde C =1/100$.
Since $\log_* M \le \log_*(100 \log(CN)) \le \log_*(100 \log(\frac C{C_1} \log_* n))$, then the constant $C_1$ is chosen so that $m_N \le n$.
Then set $r_{j}^\pm = 2^{- m_j}$ and note that $r_{N}^- \ge 2^{-n}$.

Before we apply Theorem~\ref{QCBT}, we check that our sequence of scales satisfies the set of conditions outlined in that theorem.

Observe that for any $2^{-n} \le r \le 1$, it holds that $\mathcal{H}_{r,r}^1(E_n) \lesssim 1$.
This shows that our sequence of scales satisfies the uniform length bound.

For any $1 \le j \le N -1$, 
$$r_{j+1}^+ = 2^{- m_{j+1}} = 2^{- m_{j}} 2^{- \pr{m_{j+1} - m_{j}}} \le 2^{- m_{j}} 2^{- N} \le \frac 1 2  2^{- m_{j}} = \frac 1 2 r_{j}^-,$$
so we also have separation of scales whenever $n$ is large enough so that $N \ge 1$.

For any $1 \le j \le N-2$, 
\begin{align*}
R_{E_n}\pr{r_{j+2}^+, r_{j}^-, \frac 1 {r_{j}^-}}
&= R_{E_n}(2^{- m_{j+2}}, 2^{- m_{j}}, 2^{m_{j}}).
\end{align*}
To apply Proposition~\ref{rCEsts}, we need to check that $1 \le 2^{m_{j}} \le c \brac{\log \pr{m_{j+2} - m_{j} +1}}^{1/100}$.
Since
\begin{align*}
\log(m_{j+2} - m_{j} +1)
&\approx \log \pr{e^{\tilde C m_{j+1}^{100}} - m_{j} +1}
\ge \frac{\tilde C}{2} m_{j+1}^{100},
\end{align*}
then we need $2^{m_{j}} \le c \pr{\frac{ 1}{200}}^{1/100} m_{j+1}$.
As $m_{j+1} \approx 2^{C_2 m_j^{100}}$, the hypothesis holds and we conclude from Proposition~\ref{rCEsts} that 
\begin{align*}
R_{E_n}(2^{- m_{j+2}}, 2^{- m_{j}}, 2^{m_{j}})
&\lesssim \brac{\log\pr{m_{j+2} - m_{j} +1}}^{-1/100}.
\end{align*}
Since $\log(m_{j+2} - m_{j} +1) \ge \frac 1 {200} m_{j+1}^{100} \gg m_{j+1} \ge N \gtrsim \log_* n$, then
\begin{align*}
R_{E_n}(2^{- m_{j+2}}, 2^{- m_{j}}, 2^{m_{j}})
&\lesssim \pr{\log_* n}^{-1/100}
\end{align*}
and the near unrectifiability condition also holds.

Therefore, the sequence of scales satisfies the set of conditions outlined in Theorem~\ref{QCBT}.
An application of Theorem~\ref{QCBT} with $L =1$ then shows that $\FavC(E_n) \lesssim \pr{\log_* n}^{-1/100}$, as required.
\end{proof}

In \cite{CDT21}, we use different techniques to prove a much faster rate of decay for the Favard curve length of the four-corner Cantor set.  This and other rates-type results are discussed in the next subsection.

\subsection{Rates in the literature}

There has been substantial interest in finding upper and lower bounds for the rate of decay of the Favard length of self-similar $1$-sets, such as the four-corner Cantor set. 
It remains an open problem to obtain sharp asymptotic estimates for these rates.  

Theorem~\ref{Rate curve Kn}, as well as the faster decay rates obtained in \cite{CDT21}, shed some light on the upper bound problem in the nonlinear setting. 
To put these results into context, we present the best known upper bounds in the linear setting for the four-corner Cantor set.  
As above, we use $K_n$ to denote the $n^{\text{th}}$ generation of the four-corner Cantor set as defined in \eqref{KnDefn}. 

\begin{thm}[Navarov, Peres and Volberg \cite{NPV10}]
\label{NPV}
For each $p<1/6$, there exists a constant $c>0$ so that for every $n \in \N$, $\disp \Fav(K_n) \le c n^{-p}.$
\end{thm}

Several additional works have investigated analogous upper bounds for the rate of decay of other sets: the $1$-dimensional Sierpinski gasket in \cite{BoV10}, more general $1$-dimensional irregular self-similar sets in \cite{BoV12}, product Cantor sets \cite{BLV14}, \cite{LZ10}, and random Cantor sets in \cite{Z18}. 
A common thread through each of these results (with the exception of \cite{Z18}), as well as the result of Theorem~\ref{NPV}, is the analysis of $L^p$-norms of the projection multiplicity functions.
The projection multiplicity functions count the number of components at a certain scale that orthogonally project onto a given point.
A nice survey of this area and the techniques employed can be found in \cite{Laba12}, see also \cite{BoV12}.   

The best known lower bounds for the Favard length of the four-corner Cantor set are as follows.

\begin{thm}[Bateman and Volberg \cite{BaV10}]\label{BV}
There exists a constant $c > 0$ so that for every $n \in \N$, $\disp \Fav(K_n) \ge c n^{-1} \log{n}.$
\end{thm}

Additional lower bound results apply to $s$-sets, those sets $A\subset \R^2$ for which $\mathcal{H}^s(A) \in (0, \iny)$.
In \cite{Mat90}, Mattila attains lower bounds on the Favard length of neighborhoods of arbitrary $s$-sets when $s\le 1$. 
His technique involves defining a measure on the projection space, then using a pushforward to relate the energy of this measure to the original set. 
See also \cite{Bon19} for related results.

In a joint work with Cladek \cite{CDT21}, the authors of this paper obtain upper and lower bounds on the rate of decay of the Favard curve length of the $n^{\text{th}}$ generation in the construction of the four-corner Cantor set.  
The upper bound in \cite{CDT21} is in line with the upper bound for the classic problem that appeared in \cite{NPV10}.

\begin{thm}[Cladek, Davey, Taylor \cite{CDT21}]
\label{CDTupper}
Let $\mathcal{C}$ be a piecewise $C^1$ curve of finite length with a piecewise bilipschitz continuous unit tangent vector.
For each $p<1/6$, there exists a constant $c>0$ so that for every $n \in \N$, $\disp \FavC(K_n) \le c n^{-p}$.
\end{thm}

The proof of Theorem~\ref{CDTupper} relies on a one-to-one correspondence between the family of linear projections and the curve projections on sufficiently small components of $K_n$. 
Along with Cladek, we also establish the following lower bound. 

\begin{thm}[Cladek, Davey, Taylor \cite{CDT21}]
\label{CDTlower}
Let $\mathcal{C}$ be a piecewise $C^1$ curve of finite length with a piecewise bilipschitz continuous unit tangent vector.
There exists a constant $c > 0$ so that for every $n \in \N$, $\FavC(K_n) \ge c n^{-1}$.
\end{thm}

The proof of Theorem~\ref{CDTlower} involves studying interactions between pairs of squares, much in the spirit of the techniques introduced in \cite{BaV10} that are used to prove Theorem~\ref{BV} above.  
In the curved setting, the argument becomes much more complex.
We expect that further investigations in the curved setting will yield an improved lower bound on the order of $n^{-1} \log{n}$.

In \cite{BoV11}, Bond and Volberg estimate from below the probability that a circle of radius $r$ will intersect the $n^{\text{th}}$ generation in the construction of the four-corner Cantor set. 
Their lower bound is of the form $n^{-1} \log{n}$.  
However, in their setting, the radius $r>0$ grows with the generation $n$.  

In collaboration with Bongers \cite{BonTay20}, the second-listed author introduces a technique for producing lower bounds on the rate of decay of the Favard curve length in a much more general setting that applies to arbitrary $s$-sets for $s\le 1$.  
This work extends the results of Mattila in \cite{Mat90} by replacing orthogonal projection maps with more general families of projection operators.

\subsection{Proof Approach}

Here we describe the big ideas that are used to prove Theorem~\ref{QCBT}.  
We draw inspiration from Tao's \cite[Theorem 1.13]{Tao09} as well as the original proofs of the Besicovitch projection theorem, which can be found in \cite[Theorem 6.13]{Fal85} and \cite[Theorem 18.1]{Mat95}.  

A key property of purely $1$-unrectifiable sets that is used in the proof the original qualitative Besicovitch projection theorem (see \cite[Theorem 6.13]{Fal85} and \cite[Theorem 18.1]{Mat95}) is that such sets have ``tangents almost nowhere''.
This means that almost every point in the set is approached in almost every direction by other points in the set.  
This idea is formalized by introducing double-sectors about the points in $E$ and investigating the size of the intersection of $E$ with such sets.
Introducing a curved variant of these double-sectors is critical to our analysis.
 
In the proof of \cite[Theorem 1.13]{Tao09}, the first step is to divide the set $E \times \mathbb{S}^1$ into normal  and non-normal pairs.
Roughly speaking, a pair $\pr{e, \om} \in E \times \mathbb{S}^1$ is called \textit{normal} if there is a bulk of points of $E$ in a small neighborhood of $e$ which concentrate along the direction that is normal to $\om$. 
As an example, consider when the part of $E$ in a neighborhood of $e$ is entirely contained in the line through $e$ that is orthogonal to $\om$.
In this setting, the orthogonal projection of the neighborhood of $e$ in the direction $\om$ gives only a singleton.
The idea is that for a normal pair $\pr{e, \om}$, its neighborhood should have a suitably small projection onto the direction $\om$.
A pair is called non-normal if it is not normal.  

Since we are considering curve projections, we need to adapt the notion of normal pairs to our setting.
We call a pair $\pr{e, \al} \in E \times \R$ a \textit{curve pair} if the bulk of $E$ near $e$ concentrates along the curve centered along $x = \al$ that passes through $e$. 
The formal definition of the curve double-sector is given in \eqref{circSectors} and the definition of normal pairs (which we call curve pairs) is provided by Definition~\ref{curve_pairs_defn}. 
In practice, our curve pairs are defined in an analogous way to Tao's normal pairs, where we take $\om = \om\pr{\al} \in \mathbb{S}^1$ to be the normal direction at the point $e$ to the curve centered at $\pr{\al, \Phi_\al(e)}$.

Once the notion of a normal pair has been introduced in \cite{Tao09}, each pair $\pr{e, \om} \in E \times \mathbb{S}^1$ is either normal or non-normal.
To treat the non-normal points, an exceptional set of low density points is first removed.
Using a Vitali-type argument, it is shown that the exceptional set has small measure.
A technical argument shows that the remaining points are Lipschitz in nature, and the assumed bound on the rectifiability constant is then used to estimate the measure of these remaining points.  
Since each projection is a contraction, an upper bound on the measure of the non-normal pairs immediately yields an upper bound on the Favard length.
The general argument for our curve projections, which appears in Section~\ref{goodPoints}, follows this idea while introducing a series of technical modifications.
In fact, this part of our article contains many new ideas that significantly distinguish it from the corresponding parts of \cite{Tao09}.
If we are trying to compare these arguments to their qualitative counterparts, the non-normal pairs reflect the nature of Mattila's $A_{1,\de}$ sets \cite[Chapter 18]{Mat95}. 
Falconer \cite[Chapter 6]{Fal85} shows that almost every point in $E$ is a point of radiation, so the non-normal pairs correspond to the points that are not points of radiation, or the directions that are not condensation directions.
In the qualitative setting, the set of all of these pairs has measure zero.
%

Now we describe the approach to the normal pairs.
First, high multiplicity lines (defined at each scale) are introduced.
These sets can be thought of as quantitative versions of Mattila's $A_3$ sets or Falconer's condensation directions of the first kind.
In \cite{Tao09}, a ``sliding" pigeonhole principle (see Lemma~\ref{pHole}) is used to select a single scale around which the high multiplicity lines have a sufficiently small measure.   
Then the neighborhood of the underlying set is analyzed using a Fubini-type argument.
Next, the lines that are not of high multiplicity but are also not of zero multiplicity are considered.
A counting argument combined with the pigeonhole principle is used to select the next scale in such a way that the resulting set has a sufficiently small measure.  
Points that lie in high density strips (defined at each scale) are then analyzed.  
An application of the Hardy-Littlewood maximal inequality, which can be viewed as a quantitative version of the Lebesgue differentiation theorem, shows that these sets also have small Favard length.
Again, the sliding pigeonhole principle is used to choose a third and final scale around which this analysis is carried out.  
To finish the argument, the remaining normal pairs are analyzed.
These normal pairs can be compared to Mattila's $A_{2,\de}$ sets or Falconer's condensation directions of the second kind.
The main observation here is that these remaining normal pairs are concentrated around a special, fine scale set, and another application of the Hardy-Littlewood maximal inequality completes the argument.

For our curve pairs, the approach is very similar to Tao's.
Instead of sets of high-multiplicity lines, positive-multiplicity lines, and high-density strips, we consider sets of high-multiplicity curves, positive-multiplicity curves, and high-density curve strips.
Although our proof roughly follows Tao's, the nonlinear nature of our projections introduces a number of technical hurdles that don't appear in the linear setting.
We also chose an exposition that is quite different from the one in \cite{Tao09}. 
Our first step is to completely decompose the set $E$, and then we analyze each of the components.
For a visual representation of the decomposition, see Figure \ref{trees}. 
A more detailed explanation of the ideas and notions discussed above, as well as a rigorous presentation of the selection of scales, is given in Section~\ref{decomp}.

\subsection{Organization of the paper}

The remainder of this article is organized as follows.
In the next section, Section~\ref{setup_section}, we make some simplifying assumptions about our curves and reintroduce the curve projections in a more basic form.
We then define our curve double-sectors and introduce the measures that will be used.
Section~\ref{decomp} describes how we decompose the set into subsets that can be analyzed as described above.
This section illuminates our use of multiscale analysis.
In Section~\ref{goodPoints}, we analyze the non-curve elements.
The key observation here is that most of these points cluster around a Lipschitz curve, so by the near unrectifiability assumption, they must have small measure.
Section~\ref{FavCirc} contains the analysis of the selected neighborhoods of the high-multiplicity curve set and the high-density curve strip set.
This section contains a Fubini-type argument and an application of the Hardy-Littlewood maximal inequality.
In Section~\ref{FAnalysis}, another application of the Hardy-Littlewood maximal inequality is used to show that the remaining curve pairs cluster around a fine scale set, and consequently have a small measure.
Our observations are combined in Section~\ref{conclusion} where we complete the proof.
Some technical details have been collected in the Appendix \ref{apx}.

\subsection*{Acknowledgements.}
This material is based upon work supported by the National Security Agency under Grant No. H98230-19-1-0119, The Lyda Hill Foundation, The McGovern Foundation, and Microsoft Research, while the authors were in residence at the Mathematical Sciences Research Institute in Berkeley, California, during the summer of 2019.

\smallskip

\section{Preparation}
\label{setup_section}

Before we decompose the set $E$, we first make a number of simplifying assumptions about the curve that we are working with.
These simplifications allow us to describe the curve projection as a real-valued function.
Then we define the curve double-sectors that will be used in the decomposition.
Next, we introduce the relevant measures and collect some observations about their relationships.
In the subsequent section, these tools are used to decompose the set.

\subsection{Simplifying the curve}

Let the curve $\mathcal{C}$ be as given.
That is, $\mathcal{C}$ is a piecewise $C^1$ curve of finite length with a piecewise bilipschitz continuous unit tangent vector.
Then we can write $\disp \mathcal{C} = \bigsqcup_{i=1}^N \mathcal{C}_i$, where each $\mathcal{C}_i$ is a $C^1$ graph with a strictly monotonic bilipschitz continuous derivative over some orthonormal basis.
In other words, for each $i$, $\mathcal{C}_i = \set{t \om_1^i + \vp_i(t) \om_2^i : t \in I_i}$, where $\vp_i$ is $C^1$, $\vp_i'$ is $\la_i$-bilipschitz (and therefore strictly monotonic), $I_i$ is a finite interval, and $\pr{\om_1^i, \om_2^i}$ is a pair of orthonormal vectors.
Since $\disp \FavC(E) := \sum_{i=1}^N \Fav_{\mathcal{C}_i}(E)$, then we make the simplifying assumption that $\mathcal{C}$ itself is such a graph.
That is,
\begin{equation}
\label{simpleC}  
\mathcal{C} = \set{\pr{t, \vp(t)} : t \in I},  
\end{equation}
where $I$ is a closed and bounded interval, $\vp$ is $C^1$, and $\vp'$ is $\la$-bilipschitz so that for any $s, t \in I$,
\begin{equation}
\label{biLipCondition}
\la^{-1} \abs{s - t} \le \abs{\vp'(s) - \vp'(t)} \le \la \abs{s - t}.
\end{equation}
In fact, since $\vp'$ is continuous on a compact set, then it is bounded.
Moreover, since $\vp'$ is bilipschitz continuous, then $\vp'$ is strictly monotonic and $\vp''$ exists a.e., so that $\la \ge \abs{\vp''} \ge \la^{-1} > 0$ a.e. in $I$.

We will assume that $\abs{\vp'(t)} \le 1 - \de$, where $\de > 0$ is defined in \eqref{deltaDefn}, for all $t \in I$ since there is no loss in doing so.
Observe then that by the mean value theorem, there exists an $h \in I$ between $t$ and $s$ such that
\begin{equation}
\label{LipCondition}
\abs{\vp(s) - \vp(t)} = \abs{\vp'(h)(s-t)} < \abs{s-t}.
\end{equation}
In particular, $\vp$ is $1$-Lipschitz.

\begin{rem}
\label{lambda assumption}
It is clear that $\la \ge 1$.
We will assume throughout the proof that $\la \le 2^{35}$.
Our techniques can handle larger values of $\la$, but we would need to adjust our choices of constants, indices, etc.
\end{rem}

\subsection{The projection map}

The projection map $\Phi_\al$ associated to $\mathcal{C}$ is defined in \eqref{PaProj}.
However, since the curve is given by a graph, we may now define $\Phi_\al$ explicitly.
For $p = (p_1, p_2)$, the projection is either a singleton or the empty set:
\begin{equation}
\label{PhialDefn}
\Phi_\al\pr{p} = \left\{ \begin{array}{ll} \set{p_2 - \vp(p_1 - \al)} & p_1 - \al \in I \\ \emptyset & \textrm{ otherwise} \end{array}\right..
\end{equation}
Moreover, 
$$\Phi_\al^{-1}(\be) = (\al, \be) + \mathcal{C} = \set{\pr{\al + t, \be + \vp(t)} : t \in I}.$$

Fix a compact set $E \su \brac{0,1}^2$ and let $A = \brac{0, 1} - I$.
Since $I$ is assumed to be bounded, then so too is $A$.
Define the $3$-dimensional set of pairs associated to non-empty projections as 
\begin{equation}
\label{ESetDef}
\mathcal{E} = \set{(e, \al) \in E \times A : \Phi_\al(e) \ne \emptyset} = \set{(e_1, e_2, e_1 - t) : e = (e_1, e_2) \in E, t \in I} .
\end{equation}
For $\al \in A$, the map 
$$\Phi_\al : E_\al :=\set{e \in E: \pr{e, \al} \in \mathcal{E}} \to \R$$ 
is well-defined by identifying each singleton set with its element.
Observe that for any $\pr{e, \al} \in \mathcal{E}$, $\Phi_\al^{-1}(\Phi_\al(e)) = (\al, \Phi_\al(e)) + \mathcal{C}$, a non-empty curve that passes through $e$.

Recalling Definition~\ref{FavC}, if $S \su \brac{0, 1}^2$, then
\begin{equation}
\label{2dFavC}
\FavC(S)
= \int_{\R} \abs{\Phi_\al(S)} d\al
= \int_{A} \abs{\Phi_\al(S)} d\al.
\end{equation}
In particular, if $S \su E$, then \eqref{2dFavC} holds in place of the first formula provided in Definition~\ref{FavC}.
Examining the second formula provided by Definition~\ref{FavC}, if $\mathcal{S} \su \R^2 \times J$, for some interval $J \su \R$, then $S_\al \ne \emptyset$ iff $\al \in J$. 
Therefore,
\begin{equation}
\label{3dFavC}
\FavC(\mathcal{S})
= \int_{J} \abs{\Phi_\al(S_\al)} d\al.
\end{equation}
For example, if $\mathcal{S} \su \mathcal{E}$, then the above formula with $J = A$ replaces the one presented in Definition~\ref{FavC}.

\subsection{Extending the curve}
\label{extending_the_curve}

Note that if $e_1 - \al$ is near or at an endpoint of $I$, then $e$ will be near or at an endpoint of the curve $\Phi_\al^{-1}(\Phi_\al(e))$.
Since we will (for technical reasons) want to avoid being near the endpoint of curves, we introduce extensions of our curves as follows.

Set 
\begin{equation}
\label{deltaDefn}
\de = 10^{-5} + 2^{-100}
\end{equation} 
and let $I_+$ denote the $\de$-neighborhood of $I$.
That is, if $I = \brac{a, b}$, then $I_+ = \brac{a-\de, b+\de}$.
Define $\vp_+ : I_+ \to \R$ so that $\vp_+$ extends $\vp$ and maintains all of the properties of $\vp$ that we described above.
In particular, we set
\begin{equation*}
\vp_+(t) = \left\{\begin{array}{ll} 
\vp(a) + \vp'(a) (t - a) + \frac{\sgn\pr{\vp''}}{2\la} (t - a)^2 & a - \de \le t \le a \\
\vp(t) & a \le t \le b \\
\vp(b) + \vp'(b) (t - b) + \frac{\sgn\pr{\vp''}}{2\la} (t - b)^2 & b \le t \le b + \de
\end{array}\right.
\end{equation*}
so that $\vp(t) = \vp_+(t)$ for every $t \in I$, $\vp_+$ is $C^1$, $\abs{\vp_+'(t)}\le 1$ for all $t \in I_+$, $\vp_+$ is $1$-Lipschitz, and $\vp_+'$ is $\la$-bilipschitz. 
Let $\mathcal{C}_+$ denote the extended curve given by $\mathcal{C}_+ = \set{\pr{t, \vp_+(t)} : t \in I_+}$.

We now repeat the definitions from above for the extended curve.
For $p = (p_1, p_2)$, the extended projection is defined as
\begin{equation}
\label{Phial+Defn}
\Phi_{\al,+}(p) = \left\{ \begin{array}{ll} \set{p_2 - \vp_+(p_1 - \al)} & p_1 - \al \in I_+ \\ \emptyset & \textrm{ otherwise} \end{array}\right..
\end{equation}
Then 
$$\Phi_{\al,+}^{-1}(\be) = (\al, \be) + \mathcal{C}_+ = \set{(\al + t, \be + \vp_{+}(t)) : t \in I_+}.$$

With $A_+ = \brac{0, 1} - I_+$, set 
\begin{equation*}
\mathcal{E}_+ = \set{(e, \al) \in E \times A_+ : \Phi_{\al,+}(e) \ne \emptyset} = \set{(e_1, e_2, e_1 - t) : e = (e_1, e_2) \in E, t \in I_+} .
\end{equation*}
For $\pr{e, \al} \in \mathcal{E}_+$, we treat $\Phi_{\al}^+(e)$ as a real number by identifying each singleton set with its element.
For any $\pr{e, \al} \in \mathcal{E}_+$, $\Phi_{\al,+}^{-1}(\Phi_{\al,+}(e)) = (\al, \Phi_{\al,+}(e)) + \mathcal{C}_+$ is a non-empty curve that passes through $e$.
If $(e, \al) \in \mathcal{E}$, then $\Phi_{\al,+}^{-1}(\Phi_{\al,+}(e)) = \Phi_{\al,+}^{-1}(\Phi_{\al}(e))$ is a non-empty curve that passes through $e$ and extends beyond $e$ in both directions by at least $\de$ measured along the $x$-axis.

As these extended curves will be used extensively below, for any $\pr{e, \al} \in \mathcal{E}_+$, we define
\begin{equation}
\label{CealDef}
C_{e, \al} =
\Phi_{\al,+}^{-1}(\Phi_{\al,+}(e))
= (\al, \Phi_{\al,+}(e)) + \mathcal{C}_+
= \set{(\al + t, \Phi_{\al,+}(e) + \vp_+(t)) : t \in I_+}. 
\end{equation}

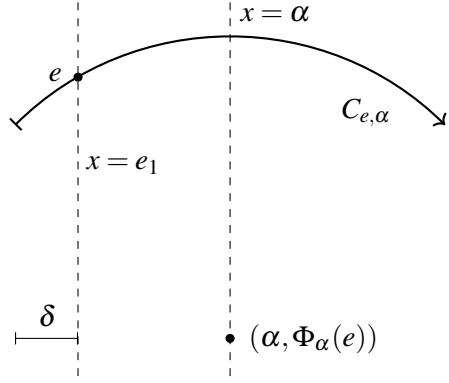
\begin{wrapfigure}{r}{0.5 \textwidth}
\begin{tikzpicture}
\draw[thick, <-|] (2.8284, 2.8284) arc (45:135:4cm);
\draw [fill=black] ( -2,3.4641) circle (1.5pt);
\draw[color=black] (-2.3,3.5) node {$e$};
\draw [fill=black] (0,0) circle (1.5pt);
\draw[color=black] (1.1,0) node {$\pr{\al, \Phi_\al(e)}$};
\draw[dashed] (0,-.5) -- (0, 4.5);
\draw[dashed] (-2,-.5) -- (-2, 4.5);
\draw[>= latex, |-|] (-2.8284,0) -- (-2,0);
\draw[color=black] (-2.4,0.3) node {$\de$};
\draw[color=black] (-1.4,2.3) node {$x = e_1$};
\draw[color=black] (.6,4.3) node {$x = \al$};
\draw[color=black] (1.8,3) node {$C_{e,\al}$};
\end{tikzpicture}
\centering
\caption{The image of $C_{e, \al}$ when $e_1 - \al$ is at the endpoint of $I$; $C_{e,\al}$ extends by $\de$ beyond $e$.}
\label{curvePic}
\end{wrapfigure}
\leavevmode
This is the extended curve centered at $(\al, \Phi_{\al,+}(e))$ that passes through $e$.
Often, we will only work with $(e, \al) \in \mathcal{E}$.
In this case, $\Phi_{\al}^+(e) = \Phi_\al(e)$ and then 
\vspace{3mm} \\
$\disp C_{e, \al} = \Phi_{\al,+}^{-1}(\Phi_{\al}(e)) = (\al, \Phi_{\al}(e)) + \mathcal{C}_+.$
\vspace{3mm} \\
By construction, if $(e, \al) \in \mathcal{E}$, then $e$ is never an endpoint of $C_{e, \al}$ and is always at least $\de$ (measured horizontally) from the end of the curve.
See Figure \ref{curvePic}.

\subsection{Curve double-sectors}
\label{curve_double_sectors}

Now we introduce the curve double-sectors.
These sets are constructed by looking at the curves $C_{e, \al'}$ in a neighborhood of $e$, where $\al'$ ranges over a small neighborhood of $\al$.
Given $\pr{e, \al} \in \mathcal{E}$, $r > 0$ and $M \ge \frac 1 \de$, set
\begin{equation}
\label{circSectors}
\mathcal{X}_{e, \al}\pr{r, M} = \set{z \in C_{e,\al'}  \su \R^2 : \abs{\al - \al'} \le \frac 1 M } \cap B_r\pr{e},
\end{equation}
where $C_{e, \al}$ is the curve defined in \eqref{CealDef} that passes through $e$.
The lower bound on $M$ ensures that $\pr{e, \al'} \in \mathcal{E}_+$ for all such $\al'$.
Indeed, if $\abs{\al - \al'} \le \frac 1 M\le \de$, then since $(e, \al) \in \mathcal{E}$ implies that $e-\al \in I$, it follows that $e_1 - \al' \in I_+$.
That is, every $C_{e,\al'}$ used to define this set is a well-defined non-empty curve.
See Figure \ref{Bowties} for a visualization of these sets.

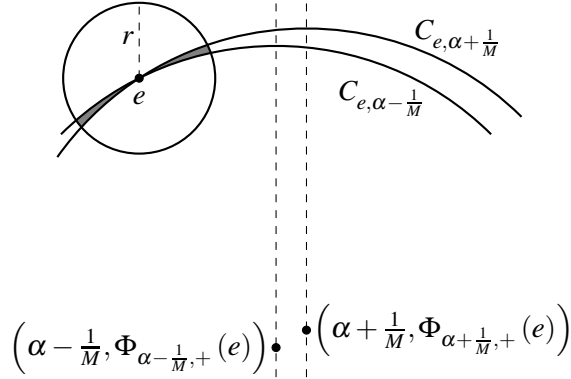
\begin{figure}[H]
\begin{tikzpicture}
\draw[thick] (3.0284, 2.9518) arc (45:145:4cm);
\draw[thick] (2.6284, 2.720386) arc (45:135:4cm);
\draw[thick] (-1, 3.4641) arc (0:360:1cm);
\draw [fill=black] ( -2,3.4641) circle (1.5pt);
\draw[color=black] (-2,3.2) node {$e$};
\draw[dashed] ( -2,3.4641) -- ( -2, 4.4641);
\draw[color=black] (-2.15, 4) node {$r$};
\fill[black, opacity = 0.5]
(-2,3.4641) arc (123.367: 109: 4cm) -- 
(-1.10262, 3.90536) arc (26.18: 19.56:1cm) -- 
(-1.05773,3.79894) arc (102.38: 116.74:4cm) -- 
   cycle;
\fill[black, opacity = 0.5]
(-2,3.4641) arc (123.367: 137.728: 4cm) -- 
(-2.75985, 2.814) arc (220.55: 213.925: 1cm) -- 
(-2.82977, 2.906) arc (131.105: 116.74: 4cm) -- 
   cycle;
\draw [fill=black] (0.2, 0.1234) circle (1.5pt);
\draw[color=black] (2, 0.1234) node {$\pr{\al + \frac 1 M, \Phi_{\al + \frac 1 M, +}\pr{e}}$};
\draw[color=black] (-2,-0.108) node {$\pr{\al - \frac 1 M, \Phi_{\al - \frac 1 M, +}\pr{e}}$};
\draw [fill=black] (-0.2,-0.108014) circle (1.5pt);
\draw[dashed] (0.2,-.5) -- (0.2, 4.5);
\draw[dashed] (-0.2,-.5) -- (-0.2, 4.5);
\draw[color=black] (1.2,3.2) node {$C_{e,\al - \frac 1 M}$};
\draw[color=black] (2.2,4) node {$C_{e,\al + \frac 1 M}$};
\end{tikzpicture}
\caption{$\mathcal{X}_{e, \al}\pr{r, M}$ is the shaded region bounded by $C_{e,\al - \frac 1 M}$, $C_{e,\al + \frac 1 M}$, and the boundary of $B_r\pr{e}$.}
\label{Bowties}
\end{figure}

Related to the curve double-sectors are straight double-sectors orthogonal to $\om \in \mathbb{S}^1$ given by
\begin{equation}
\label{bowties}
X_{e, \om}\pr{r, M} = \set{z \in \R^2 : \abs{\pr{z - e} \cdot \om }\le \frac 1 M  \abs{z - e}} \cap B_r(e).
\end{equation}

Before proceeding, it is important to check that the sets $\mathcal{X}_{e, \al}\pr{r, M}$ are not degenerate. 
In the next lemma, we establish that the curve double-sectors can be approximated by straight double-sectors with comparable amplitudes.
These relationships are illustrated in Figure \ref{Containment}.

\begin{lem}[Curve double-sectors are comparable to straight double-sectors]
\label{bowtiecontainment}
Let $\pr{e, \al} \in \mathcal{E}$ and set $\om = \frac{\pr{\vp'\pr{e_1 - \al}, -1}}{\sqrt{1 + \brac{\vp'\pr{e_1-\al}}^2}}$, the unit vector that is perpendicular to the tangent vector of $C_{e, \al}$ at $e$.
Assume that $r, M > 0$ are chosen so that $\de \ge \frac 1 M + r$, where $\de$ is as defined in \eqref{deltaDefn}.
Then
$\disp  \mathcal{X}_{e, \al}\pr{r, M} \su X_{e, \om}\pr{r, \frac{M}{\la\pr{1 + Mr}}}$.
If we further assume that $r, M > 0$ are chosen so that $r \le \frac{1}{2 \la^2 M}$, then
$\disp X_{e, \om}\pr{r, c_1 M} \su \mathcal{X}_{e, \al}\pr{r, M}$,
where $c_1 = \la \sqrt{8\brac{1 + \pr{1 + \frac{2\la}{M}}^2}}$.
\end{lem}

\begin{figure}[H]
\begin{tikzpicture}
\fill[black, opacity = 0.15]
(0, 4) -- (2.344, 4.868) -- 
(2.344, 4.868) arc (20.32: 4.27: 2.5cm) -- 
(2.49306, 4.18611) arc( 87.23: 101.31: 10.198 cm) -- 
   cycle;
\fill[black, opacity = 0.15]
(0, 4) -- (-2.344, 4.868) -- 
(-2.344, 4.868) arc (159.68: 175.73: 2.5cm) -- 
(-2.49306, 4.18611) arc( 92.77: 78.69: 10.198 cm) -- 
   cycle;
\fill[black, opacity = 0.15]
(0, 4) -- (2.344, 3.132) -- 
(2.344, 3.132) arc (-20.32: -18.35: 2.5cm) -- 
(2.37287, 3.21293) arc( 64.61: 78.69: 10.198 cm) -- 
   cycle;
\fill[black, opacity = 0.15]
(0, 4) -- (-2.344, 3.132) -- 
(-2.344, 3.132) arc (200.32: 198.35: 2.5cm) -- 
(-2.37287, 3.21293) arc(115.39: 101.31: 10.198 cm) -- 
   cycle;

\fill[black, opacity = 0.35]
(0, 4) -- (2.498, 4.0925) -- 
(2.498, 4.0925) arc (2.12: 4.27: 2.5cm) -- 
(2.49306, 4.18611) arc( 87.23: 101.31: 10.198 cm) -- 
   cycle;
\fill[black, opacity = 0.35]
(0, 4) -- (2.498, 3.90747) -- 
(2.498, 3.90747) arc (-2.12: -18.35: 2.5cm) -- 
(2.37287, 3.21293) arc(64.61: 78.69:  10.198 cm) -- 
   cycle;
\fill[black, opacity = 0.35]
(0, 4) -- (-2.498, 4.0925) -- 
(-2.498, 4.0925) arc (177.88: 175.73 : 2.5cm) -- 
(-2.49306, 4.18611) arc( 92.77: 78.69: 10.198 cm) -- 
   cycle;
\fill[black, opacity = 0.35]
(0, 4) -- (-2.498, 3.90747) -- 
(-2.498, 3.90747) arc(182.12: 198.35 : 2.5cm) --
(-2.37287, 3.21293) arc(115.39: 101.31: 10.198 cm) -- 
  cycle;

\fill[black, opacity = 0.8]
(0, 4) -- (2.498, 3.90747) -- 
(2.498, 3.90747) arc (-2.12: 2.12: 2.5cm) -- 
(2.498, 4.0925) -- (0,4) -- 
   cycle;
\fill[black, opacity = 0.8]
(0, 4) -- (-2.498, 3.90747) -- 
(-2.498, 3.90747) arc (182.12: 177.88: 2.5cm) -- 
(-2.498, 4.0925) -- (0,4) -- 
   cycle;

\draw[dashed, thick] (0, 4) circle (2.5cm);
\draw[thick] (-2.63, 3.0865) arc (117: 85: 10.198 cm);
\draw[thick] (2.63, 3.0865) arc (63: 95: 10.198 cm);
\draw[->] (0,4) -- ( 0, 3);
\draw[color=black] (0.2,3.5) node {$\om$};
\draw [fill=black] ( 0,4) circle (1.5pt);
\draw[color=black] (0,4.2) node {$e$};
\end{tikzpicture}
\centering
\caption{The region $\mathcal{X}_{e, \al}\pr{r, M}$ (shown with medium shading) is bounded by the black curves. 
The regions $X_{e, \om}\pr{r, \frac{M}{\la\pr{1 + Mr}}}$ (lightly shaded) and $X_{e, \om}\pr{r, c_1 M}$ (darkly shaded) contain and are contained in $\mathcal{X}_{e, \al}\pr{r, M}$, respectively.}
\label{Containment}
\end{figure}
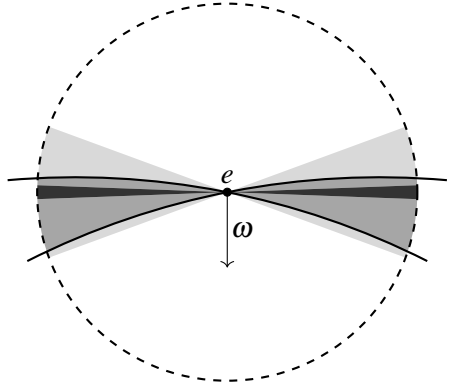

We note that the second containment of Lemma~\ref{bowtiecontainment} will be used in Section~\ref{goodPoints} while the first containment will be used to prove Corollary~\ref{bowtiestripcontainment} below, which will be invoked in Section~\ref{FAnalysis}. 

\begin{proof}
As explained above, since $\de \ge \frac 1 M$, then all of the curves used in the definition \eqref{circSectors} are non-empty and pass through $e$.
In fact, $e$ is always at least $\de - \frac 1 M \ge r$ (measured horizontally) from the end of all such $C_{e, \al'}$, so none of the curves used to define $\mathcal{X}_{e, \al}\pr{r, M}$ reach their endpoints before exiting the ball $B_r\pr{e}$.

Before showing the first claimed set inclusion, we use a Taylor expansion to produce a useful observation described by \eqref{curveExp}. 
We make use of the parametrization of $C_{e, \al'}$ given in \eqref{CealDef}.
Note that the parameter choice $t = e_1 - \al'$ corresponds to $e$ along $C_{e, \al'}$.
Thus, with $s = t + \al'$, a Taylor expansion of the function $f(s) = \Phi_{\al', +}(e) + \varphi_+(s - \al')$ about $s=e_1$ shows that for a.e. point $z = \pr{z_1, z_2}$ on the curve $C_{e,\al'}$ near $e$
\begin{equation}
\label{curveExp}
\begin{aligned}
z_2 &= e_2 + \vp_{+}'(e_1 - \al')(z_1 - e_1) + \vp_{+}''(t_0)(z_1-e_1)^2 \\
&= e_2 + \vp'(e_1 - \al)(z_1 - e_1) + \brac{ \vp_{+}'(e_1 - \al') - \vp_{+}'(e_1 - \al)}(z_1 - e_1) + \vp_{+}''(t_0)(z_1-e_1)^2,
\end{aligned}
\end{equation}
where $t_0$ is a number between $z_1+\al'$ and $e_1+\al'$.

Now we show that $\mathcal{X}_{e, \al}\pr{r, M} \su X_{e, \om}\pr{r, \frac{M}{\la\pr{1 + Mr}}}$.
If $z \in \mathcal{X}_{e, \al}\pr{r, M}$, then $z \in C_{e, \al'}$, where $\abs{\al - \al'} \le \frac 1 M$ and $\abs{z - e} \le r $.
It follows from the expansion of $z_2-e_2$ derived in \eqref{curveExp} combined with the fact that $\vp_{+}'$ is $\la$-Lipschitz and $\abs{\vp_{+}''} \le \la$ a.e., that for any such $z$ along $C_{e,\al'}$
\begin{align*}
\abs{(z - e) \cdot  \om}
&= \frac{\abs{(z - e) \cdot \pr{\vp'(e_1 - \al), -1} }}{\sqrt{1 + \abs{\vp'(e_1-\al)}^2}}
\le \abs{(z - e) \cdot \pr{\vp'(e_1 - \al), -1} } \\
&= \abs{\vp'(e_1 - \al) (z_1 - e_1) - (z_2 - e_2)} \\
&\le \abs{\vp_{+}'(e_1 - \al) - \vp_{+}'(e_1 - \al')} \abs{z_1 - e_1} + \abs{\vp''\pr{t_0}} \abs{z_1-e_1}^2 \\
&\le \frac \la M \abs{z_1 - e_1} + \la \abs{z_1 - e_1}^2
\le \la\pr{\frac 1 M + r} \abs{z - e}.
\end{align*}
That $\mathcal{X}_{e, \al}\pr{r, M} \su X_{e, \om}\pr{r, \frac{M}{\la\pr{1 + Mr}}}$ follows from this observation.

Going forward, we assume that $r \le \frac{1}{2 \la^2 M}$.
Since $\vp'_{+}$ is $\la$-Lipschitz, then $\abs{\vp_{+}''} \le \la$ a.e. 
Therefore, whenever $\abs{\al - \al'} \le \frac 1 M$ and $\abs{z_1 - e_1} \le \abs{z - e} \le r \le \frac{1}{2 \la^2 M} < \frac 1 {M}$, it follows from \eqref{curveExp} that
$$\abs{z_2 - e_2} \le \pr{\abs{\vp'(e_1 - \al)} + \frac{2\la}{M} } \abs{z_1 - e_1} \le \pr{1 + \frac{2\la}{M}}  \abs{z_1 - e_1},$$
where we have used the assumption that $|\varphi'| \le 1$.   
In particular, with $\ga = \sqrt{1 +\pr{1  + \frac{2\la}{M}}^2}$,
\begin{equation}
\label{horizontallowerBd}
\abs{z_1 - e_1} \ge \ga^{-1} \abs{z - e}.
\end{equation}

Finally, to show that $X_{e, \om}\pr{r, c_1 M} \su \mathcal{X}_{e, \al}\pr{r, M}$, we show that elements on the boundary curves of $\mathcal{X}_{e, \al}\pr{r, M}$ belong to the closure of the complement of $X_{e, \om}\pr{r, c_1 M}$.
That is, consider $z \in C_{e, \al'}$ where $\al' = \al \pm \frac 1 M$.
From \eqref{curveExp} combined with the fact that $\vp_{+}'$ is $\la$-bilipschitz and $\abs{\vp_{+}''} \le \la$ a.e., we see that
\begin{align*}
\abs{\pr{z - e}\cdot \pr{\vp'\pr{e_1 - \al}, -1} }
&= \abs{\brac{\vp_{+}'\pr{e_1 - \al \mp \frac 1 M} - \vp'(e_1 - \al)}(z_1 - e_1) + \vp_{+}''(t_0)(z_1-e_1)^2} \\
&\ge \abs{\vp_{+}'\pr{e_1 - \al \mp \frac 1 M} - \vp_{+}'(e_1 - \al)} \abs{z_1 - e_1} - \abs{\vp_{+}''(t_0)} \abs{z_1-e_1}^2 \\
&\ge \frac 1{\la M} \abs{z_1 - e_1} - \la \abs{z_1-e_1}^2
\ge \frac 1 {2 \la \ga M} \abs{z - e},
\end{align*}
where we have used the assumption that $r \le \frac 1 {2 \la^2 M}$ and the observation in \eqref{horizontallowerBd}.
Since $\sqrt{1 + \abs{\vp'(e_1-\al)}^2} \le \sqrt 2$, then
\begin{equation*}
\abs{\pr{z - e}\cdot \om}
= \abs{ \frac{\pr{z - e}\cdot \pr{\vp'\pr{e_1 - \al}, -1} }{\sqrt{1 + \abs{\vp'(e_1 - \al)}^2}}}
\ge  \frac 1 {2 \sqrt 2 \la \ga M} \abs{z - e},
\end{equation*}
 showing that $\mathcal{X}_{e, \al}\pr{r, M} \supset X_{e, \om}\pr{r, \sqrt 8 \la \ga  M}$, as required.
\end{proof}

Next we make an observation about the height of straight double-sectors.
This result will be combined with the previous one to prove the final estimate of the subsection.

\begin{lem}[Height of a straight double-sector]
\label{vertSlice}
Let $e \in \R^2$, $\om \in \mathbb{S}^1$ make an angle in $\brac{- \frac{3\pi}{4}, - \frac \pi 4}$ with the positive $x$-axis, $r > 0$, and $\mu \in \pr{0, \frac 1{\sqrt 2}}$.
Then any vertical slice of $X_{e, \om}\pr{r, \mu^{-1}}$ has length at most $\sqrt 8 \mu r$.
\end{lem}

\begin{proof}
There is no loss in assuming that $e = 0$.
By symmetry, we may further assume that $\om \in \mathbb{S}^1$ makes an angle in $\brac{- \frac{\pi}{2}, - \frac \pi 4}$ with the positive $x$-axis so that $\om^\perp \in \mathbb{S}^1$ makes an angle $\te \in \brac{0, \frac \pi 4}$ with the positive $x$-axis.

Define $\be_0 = \arcsin\pr{\mu} \in \pr{0, \frac \pi 4}$.
For any $\be \in \brac{- \be_0, \be_0}$, the point $P_\be = \pr{r \cos\pr{\te + \be}, r \sin\pr{\te + \be}}$ lies on the round boundary of $X_{0, \om}\pr{r, \mu^{-1}}$.
To determine the maximal vertical segment in $X_{0, \om}\pr{r, \mu^{-1}}$ originating at $P_\be$, we find the coordinates of the other endpoint, denoted by $Q_\be$. 
The point  $Q_\be$ has the same $x$-coordinate as $P_\be$ and lies on the line through the origin that makes an angle of $\theta - \be_0$ with the $x$-axis.  
That is, $Q_\be = \pr{\rho_\be \cos\pr{\te -\be_0}, \rho_\be \sin\pr{\te - \be_0}}$, where $\disp \rho_\be = r \frac{\cos\pr{\te + \be}}{\cos\pr{\te - \be_0}}$.
Then the vertical distance between $P_\be$ and $Q_\be$ is given by
\begin{align*}
v_\be 
&= r \sin\pr{\te + \be} - \rho_\be \sin\pr{\te - \be_0} 
= r\frac{\sin\pr{\be+ \be_0}}{\cos\pr{\te - \be_0}} .
\end{align*}
This distance is maximized when $\be = \be_0$, so that
\begin{align*}
v_\be 
\le \frac{r\sin\pr{2\be_0}}{\cos\pr{\te - \be_0}} 
= \frac{2r\sin\pr{\be_0} \cos(\be_0)}{\cos \te \cos \be_0 + \sin \te \sin \be_0} 
= \frac{2r\mu }{\cos \te + \sin \te \frac{\mu}{\sqrt{1 - \mu^2}}} 
&\le \sqrt 8 \mu r,
\end{align*}
as claimed.
\end{proof}

By combining the previous two results, we arrive at an important set containment result that will be used in Section~\ref{FAnalysis}.

\begin{cor}[Curve strips contain curve double-sectors]
\label{bowtiestripcontainment}
Let $\pr{e, \al} \in \mathcal{E}$.
Set $\om = \frac{\pr{\vp'\pr{e_1 - \al}, -1}}{\sqrt{1 + \brac{\vp'\pr{e_1-\al}}^2}}$, the unit vector that is perpendicular to the tangent vector of $C_{e, \al}$ at $e$.
Assume that $r, M > 0$ are chosen so that $\frac 1 M + r < \min\set{\frac 1 {\sqrt 2 \la}, \de}$.
Then with $J = \brac{\Phi_\al\pr{e} - \sqrt 8 \la\pr{\frac 1 M + r} r, \Phi_\al\pr{e} + \sqrt 8 \la\pr{\frac 1 M + r} r}$,
\begin{align*}
\mathcal{X}_{e, \al}\pr{r, M} \su \Phi_{\al,+}^{-1}\pr{J}.
\end{align*} 
\end{cor}

\begin{proof}
Since $\de \ge \frac 1 M + r$, then Lemma~\ref{bowtiecontainment} shows that $\mathcal{X}_{e, \al}\pr{r, M} \su X_{e, \om}\pr{r, \frac{M}{\la\pr{1 + Mr}}}$.
Since $\om$ corresponds to an angle in $\brac{- \frac{3\pi}{4}, - \frac \pi 4}$ and $\frac{\la\pr{1 + Mr}}{M} \in \pr{0, \frac 1 {\sqrt 2}}$, then Lemma~\ref{vertSlice} shows that any vertical slice of $X_{e, \om}\pr{r, \frac{M}{\la\pr{1 + Mr}}}$ has length at most $\sqrt 8 \la\pr{\frac 1 M + r} r$.

Let $z \in \mathcal{X}_{e, \al}\pr{r, M}$, then $z \in X_{e, \om}\pr{r, \frac{M}{\la\pr{1 + Mr}}}$.
Since $C_{e, \al} = \Phi_{\al,+}^{-1}\pr{\Phi_\al\pr{e}}$ is the curve that passes through $\mathcal{X}_{e, \al}\pr{r, M}$, then because $z \in X_{e, \om}\pr{r, \frac{M}{\la\pr{1 + Mr}}}$, the the vertical distance between $z$ and $C_{e, \al}$ is at most $\sqrt 8 \la\pr{\frac 1 M + r} r$.
Since $\Phi_{\al,+}^{-1}$ doesn't change vertical distances, then $z \in \Phi_{\al,+}^{-1}\pr{J}$ and the conclusion follows.
\end{proof}

\subsection{The measures}
\label{measObs}

Here we introduce the measures that we will work with and collect some observations about their relationships to the curve projections.

Let $\mu$ be the $1$-dimensional Hausdorff measure $\mathcal{H}^1$ restricted to $E \su \brac{0,1}^2$.
In other words, $\mu$ is supported on $E$ and, under the assumptions of Theorem~\ref{QCBT}, $\mu(E) \le L$.
Let $\nu = \abs{\cdot \cap A} \abs{I}^{-1}$, the $1$-dimensional Lebesgue measure restricted to $A = \brac{0, 1} - I$, reweighted by dividing through by the measure of $I$.
If we let $\mu \times \nu$ denote the product measure on $E \times \R$, then 
\begin{equation}
\label{prodMeas}
\pr{\mu \times \nu} \pr{\mathcal{E}} \le L.
\end{equation}

In the next section, we break down $E$, or $\mathcal{E}$, into subsets that will be individually analyzed.
Before describing the decomposition, we make the following observation.

\begin{lem}[Projection is bounded by measure]
\label{smallProjProp}
For any $S \su E \su \R^2$ and any $\al \in \R$, $\abs{\Phi_\al\pr{S}} \lesssim \mu\pr{S}$, where the implicit constant depends on $\la$ and is independent of $\al$.
\end{lem}

The idea behind this observation is that the projection is Lipschitz in nature, so it cannot increase the measure of a set by too much.

\begin{proof}
By definition, there exists a countable collection $\mathcal{B} = \set{B_n}$ of balls that covers $S =S \cap E$ and satisfies $\sum r\pr{B_n} \lesssim \mu\pr{S}$.
By monotone convergence and the separability of $\R$, it suffices to show that whenever $\widetilde{\mathcal{B}}$ is a finite subcollection of $\mathcal{B}$, we have
$\disp \abs{\bigcup_{B \in \widetilde{\mathcal{B}}} \Phi_\al\pr{B}} \lesssim \mu\pr{S}$.

Choose $\al \in A$ and $B \in \widetilde{\mathcal{B}}$.
If $B \cap \set{ \pr{\al + I} \times \R} = \emptyset$, then $\Phi_\al\pr{B} = \emptyset$. 
If $B \su \set{ \pr{\al + I} \times \R}$, then since $\vp$ is $1$-Lipschitz, $\Phi_\al\pr{B} = J$ for some interval $J$ with $\abs{J} \lesssim r\pr{B}$.
Otherwise, $\Phi_\al\pr{B} = J$, where $J$ satisfies a size condition as before.
In all three cases, we see that $\abs{\Phi_\al\pr{B}} \lesssim r\pr{B}$.
It follows that
$$\abs{\bigcup_{B \in \widetilde{\mathcal{B}}} \Phi_\al\pr{B}} 
\le \sum_{B \in \widetilde{\mathcal{B}}} \abs{\Phi_\al\pr{B}} 
\lesssim \sum_{B \in \widetilde{\mathcal{B}}} r\pr{B} 
\le  \sum_{B_n \in \mathcal{B}}  r\pr{B_n} \lesssim \mu\pr{S},$$
as required.
\end{proof}

\begin{cor}[Favard length is bounded by measure]
\label{smallFavard}
For any $S \su E \su \R^2$, $\FavC\pr{S} \lesssim \mu\pr{S}$.
Similarly, for any $\mathcal{S}  \su \mathcal{E} \su \R^3$, $\FavC\pr{\mathcal{S}} \lesssim \pr{\mu \times \nu }\pr{\mathcal{S}}$.
\end{cor}

\begin{proof}
By the definition of Favard curve length described by \eqref{2dFavC} and Lemma~\ref{smallProjProp}, 
\begin{align*}
\FavC\pr{S} 
&= \int_A \abs{\Phi_\al\pr{S}} d\al 
\lesssim \int_A \mu\pr{S} d \al 
\lesssim \mu\pr{S},
\end{align*}
where we have used that $A = \brac{0,1} - I$ is bounded in the last step.

With $S_\al = \set{s \in \R^2 : \pr{s, \al} \in \mathcal{S}} \su E$, it follows from the definition of Favard curve length described by \eqref{3dFavC} that
\begin{align*}
\FavC\pr{\mathcal{S}} 
= \int_{A} \abs{\Phi_\al\pr{S_\al}} d\al
\lesssim \int_{A} \mu\pr{S_\al} d\al
\lesssim \pr{\mu \times \nu }\pr{\mathcal{S}},
\end{align*}
where we have again used Lemma~\ref{smallProjProp} and the boundedness of $A$.
\end{proof}

The takeaway is that we have three different approaches to estimating the Favard curve length.
The first way is to work directly with the projection.
In the other two approaches, we either show that the $\mu$-measure or the $\pr{\mu \times \nu}$-measure is small, and then conclude from either Lemma~\ref{smallProjProp} or Corollary~\ref{smallFavard} that the Favard curve length is comparably small.
We use these distinct approaches on the different pieces of $E$ and $\mathcal{E}$.

\smallskip

\section{Decomposition of the Set}
\label{decomp}

Using the tools that were established in the previous section, we now decompose the compact set $E$ and the corresponding set of pairs $\mathcal{E}$ associated to $E$ as defined in \eqref{ESetDef}.
Recall our hypotheses, that for some sufficiently large $N \in \N$, there is a sequence of scales
\begin{equation}
0 < r_{N}^- \le r_{N}^+ < \cdots < r_{1}^- \le r_{1}^+ \le 1
\end{equation}
satisfying the uniform length bound, $\mathcal{H}^1_{r_{n}^-, r_{n}^+}(E) \le L$ for all $n = 1, 2, \ldots, N$; and separation of scales, $r_{n+1}^+ \le \tfrac 1 2 r_{n}^-$ for all $n = 1, 2, \ldots, N-1$.
The curve double-sectors are denoted by $\mathcal{X}_{e, \al}\pr{r, M}$ and defined in \eqref{circSectors}.
We begin with a definition.

\begin{defn}[Curve pairs]\label{curve_pairs_defn}
\label{CircPairs}
Let $n \in \set{101, 102, \ldots, N - 100}$ and assume that $M > 10^5$.
We say that a pair $\pr{e, \al} \in \mathcal{E}$ is a {\bf curve pair} at scale $n$ with Lipschitz constant $M$ if there exists $r \in \brac{r_{n+100}^- , r_{n-100}^+}$ so that 
\begin{equation*}
\mu\pr{\mathcal{X}_{e, \al}\pr{r, M/10^4} \setminus \mathcal{X}_{e, \al}\pr{r_{n+100}^-, M/10^4}} > N^{-1/100} r/M.
\end{equation*}
Let  $\Cur_{n, M} \su \mathcal{E}$ denote the set of all curve pairs at scale $n$ with Lipschitz constant $M$.
\end{defn}

\begin{defn}[Non-curve pairs]\label{nc_pairs_defn}
A pair $\pr{e, \al} \in \mathcal{E}$ is called a {\bf non-curve pair} at scale $n$ with Lipschitz constant $M$ if it is does not belong to $\Cur_{n, M}$.  We let the set of all such pairs be denoted by $\NCur_{n,M}$.
\end{defn}

Now,
$$\mathcal{E} = \Cur_{n, M} \sqcup \NCur_{n, M}.$$
Although this decomposition of $\mathcal{E}$ holds for any scale $n \in \set{101, 102, \ldots, N - 100}$, we will make a specific choice for $n$ and $M$ below.

\begin{rem}
For a non-curve pair $(e,\alpha)$, while a neighborhood of $e$ may still concentrate along {\it some} curve, we name them as such because there is not a clustering of points along the specific curve $C_{e,\al}$.
\end{rem}

\subsection{Non-curve elements}\label{ncintro}

For each $\al \in A$, we define
\begin{equation}
\label{NalDef}
N_\al := \set{e \in E : \pr{e, \al} \in \NCur_{n, M}},
\end{equation}
where $\NCur_{n, M}$ is as in Definition~\ref{nc_pairs_defn}.
We refer to these points as the non-curve elements (with respect to $\al$).  
Roughly speaking, this is the set of points $e \in E$ whose neighborhoods don't cluster about the curve $C_{e,\al}$.

By the uniform length bound described in \eqref{Ulb}, there exists a finite collection $\mathcal{B}_n$ of open balls of radius between $r_{n}^-$ and $r_{n}^+$ that cover E, such that 
\begin{equation}
\label{BCover}
\sum_{B \in \mathcal{B}_n} r\pr{B} \lesssim L.
\end{equation}
We use this cover to define an exceptional subset of low-density elements in $N_\al $ as follows.

\begin{defn}[Low density intervals]
\label{lowDJ_defn} 
For each $B \in \mathcal{B}_n$, we say that an interval $J \su \R$ is of \textbf{low density relative to $B$ and $\al$} if $\abs{J} \le r(B)$ and
\begin{equation}
\label{lowDJ}
\mu \pr{ B \cap N_\al \cap \Phi_\al^{-1}\pr{5 J}} \le 10^{10} N^{-1/100} \abs{J},
\end{equation}
where $5J$ denotes the interval with the same center as $J$ but $5$ times its radius.
Let $\mathcal{J}_B$ denote the set of all intervals of low density relative to $B$ and $\al$.
\end{defn}

Then we define the exceptional set $G_\al$ as
\begin{equation}
\label{EalDef}
G_\al = \bigcup_{B \in \mathcal{B}_n} \bigcup_{J \in \mathcal{J}_B} B \cap N_\al \cap \Phi_\al^{-1}\pr{J}.
\end{equation}
If we define $K_\al = N_\al \setminus G_\al$,
 then it is clear that
$$N_\al = K_\al \sqcup G_\al.$$
A Vitali covering argument is used to show that the exceptional points, $G_\al$, have a small $\mu$-measure. 
Then we show that the remaining points in $K_\al$ are Lipschitz in nature.
By the near unrectifiability assumption given in \eqref{unrect}, the $\mu$-measure of $K_\al$  must  be small.
These details are presented in Section~\ref{goodPoints}.
This part of our article contains a number of novel ideas that differentiate it from the corresponding arguments given in \cite{Tao09}.

\subsection{Curve pairs}

The way in which we break down the curve pairs is somewhat complex.
Our decomposition will consist of (neighborhoods of) subsets of $\mathcal{E}$ associated to high multiplicity curves, positive multiplicity curves, and high density curve strips at various scales.
Here we use the pigeonhole principle to choose scales, and we therefore need to work on different scales at each stage of the decomposition.
We start by defining the different kinds of subsets that we use to decompose our set.

\begin{defn}[High multiplicity curves]
\label{hmSC}
Let $n \in \set{1, 2, \ldots, N}$.
A curve $C \su \R^2$ is said to be of {\bf high multiplicity at a scale index at most $n$} if $E \cap C$ contains a subset of cardinality at least $N^{1/100}$ that is $r_{n}^-$-separated.
That is, for any two points in this subset of $E \cap C$, the distance between these points is at least $r_{n}^-$.
Let 
$$H_n = \set{\pr{e, \al} \in \mathcal{E}  : C_{e, \al} \textrm{ is of high multiplicity at a scale index at most } n}.$$
\end{defn}

Using a diagonalization argument and that $E$ is compact, it can be shown that each $H_n$ is closed, and therefore is itself compact; see Appendix \ref{apx} for details.
Note that these sets are also nested in the sense that 
$$H_1 \subset H_2 \subset \ldots \subset H_N \subset \mathcal{E}.$$

If a pair is not associated with a high multiplicity curve, but is also not associated with a curve that only intersects $E$ at one point, then it is associated to what we call a positive multiplicity curve.
We use our scales to quantify such pairs and the associated curves as follows.

\begin{defn}[Positive multiplicity curves]
\label{pmSC}
Let $n \in \set{1, 2, \ldots, N} $.
We say that a pair $\pr{e, \al} \in \mathcal{E}$ has {\bf positive multiplicity at scale index $n$} if there exists a $y \in E \cap C_{e, \al}$ such that $\abs{y - e} \in \brac{r_{n+ N^{-7/100}N}^-, r_{n-N^{-7/100}N}^+}$.
Let 
$$P_n = \set{\pr{e, \al} \in \mathcal{E}  : C_{e, \al} \textrm{ has positive multiplicity at a scale index } n}.$$
\end{defn}

To allow for some wiggle room, we also introduce high density curve strips.
 
\begin{defn}[High density strips]
\label{hdSS}
Let $n \in \set{1, 2, \ldots, N}$.
For $\al \in A$ and an interval $J \su \R$ with $\abs{J} \ge r_{n}^-$, a curve strip $\Phi_{\al,+}^{-1}\pr{J} \su \R^2$ is said to have {\bf high density at scale index $n$} if
$$\mu\pr{\Phi_{\al,+}^{-1}\pr{J}} \ge N^{1/100} \abs{J}.$$
Let
$$D_n = \set{\pr{e, \al} \in \mathcal{E} : C_{e, \al} \subset \Phi_{\al,+}^{-1}\pr{J} \textrm{ for some } \Phi_{\al,+}^{-1}\pr{J} \textrm{ with high density at scale index } n}.$$
\end{defn}

In a sense, these high density curve strips resemble the high multiplicity curves when the counting measure is replaced by the $\mu$-measure.
In fact, we have that each $D_n$ is compact (since each $D_n$ is closed, as shown in Appendix \ref{apx}) and that 
$$D_1 \subset D_2 \subset \ldots \subset D_N \subset \mathcal{E}.$$

When we decompose our set, we use two different kinds of neighborhoods: standard neighborhoods and parametric neighborhoods.

\begin{defn}[Neighborhoods]
Let $\mathcal{S} \subset \mathcal{E}$ and $\eps > 0$. 
\begin{itemize}
\item[-] The {\bf $\eps$-neighborhood of $\mathcal{S}$} is defined as
$$\mathcal{N}_\eps\pr{\mathcal{S}} = \set{\pr{b, \be} \in \R^2 \times \R : \norm{\pr{e - b,\al - \be}} < \eps \textrm{ for some } \pr{e, \al} \in \mathcal{S}},$$
where $\norm{\cdot}$ denotes the Euclidean norm in $\R^3$. 
\item[-] The {\bf $\eps$-parametric neighborhood of $\mathcal{S}$} is defined as
$$\mathcal{M}_\eps\pr{\mathcal{S}} = \set{\pr{e, \be} \in E \times \R \,  : \abs{\al - \be} < \eps \textrm{ for some } \pr{e, \al} \in \mathcal{S}}.$$
\end{itemize}
\end{defn}

As we will see below, the parametric neighborhoods are used with the high multiplicity curves, while the standard neighborhoods are used with the high density strips.
The reason why we require different kinds of neighborhoods becomes evident in the technical arguments that appear in Section~\ref{FavCirc}.

Now we state the \textit{sliding pigeonhole principle} that will be used repeatedly when we choose our scales.

\begin{lem}[Pigeonhole Principle]
\label{pHole}
Let $\pr{X, \mu}$ be a measure space.
Suppose $E_0 \subset E_1 \subset \ldots \subset E_N \subset X$ is sequence of measurable sets with $N \ge 2$.
If $\eps \in \brac{\frac 1 N, \frac 1 2}$, then there exists $n, m \in \set{0, 1, \ldots, N} $ with $m - n \ge \eps N$ such that $\mu\pr{E_m \setminus E_n} \lesssim \eps \mu\pr{E_N}$.
\end{lem}

\begin{proof}
Note that we can write $E_N = E_0 \sqcup \pr{E_1 \setminus E_0} \sqcup \ldots \sqcup \pr{E_N \setminus E_{N-1}}$, where the union is disjoint.
Observe that for any $\ell \in \set{1, 2, \ldots, N} $ and any $k \in \N$, $E_{\ell} \setminus E_{\ell-1} \su E_{n+k} \setminus E_n$ whenever $\max\set{\ell - k,0} \le n \le \min\set{\ell - 1, N-k}$.
This means that each set of the form $E_{\ell} \setminus E_{\ell-1}$ can belong to at most $k$ sets of the form $E_{n+k} \setminus E_n$.
It follows that
$$\sum_{n=0}^{N-k} \mu\pr{E_{n+k} \setminus E_n} \le k \sum_{\ell = 1}^N \mu\pr{E_{\ell} \setminus E_{\ell-1}} \le k \mu\pr{E_N}.$$
We deduce from the pigeonhole principle that there exists $n \in \set{0, 1, \ldots, N-k}$ for which $\mu\pr{E_{n+k} \setminus E_n} \le \frac{k}{N-k} \mu\pr{E_N}.$
We reach the conclusion of the lemma by setting $k = \lceil \eps N \rceil$.
\end{proof}

We've reached the description of the decomposition and the role of scales.
\smallskip

{\bf Step 1:} Our first sets in the decomposition of the normal pairs will be associated to the high multiplicity curves.
Let $Z_0 := \brac{0.1 N, 0.9 N} \cap \Z$ and note that since $N^{1/100} \in \Z_{\ge 3}$, then for any $n \in \Z_0$, we have that $n \pm N^{-3/100} N \in \set{1, \ldots, N}$.
An application of the pigeonhole principle from Lemma~\ref{pHole} in combination with \eqref{prodMeas} implies that there exists a stable scale index $n_0 \in Z_0$ for which
\begin{equation}
\label{HPHole}
\pr{\mu \times \nu } \pr{H_{n_0 + N^{-3/100} N} \setminus H_{n_0 - N^{-3/100} N}} \lesssim N^{-3/100} L.
\end{equation}

With this scale index $n_0$ fixed, we define parametric neighborhoods of the smaller set as
\begin{align}
& \widetilde H = \mathcal{M}_{r_{n_0 - N^{-3/100} N + 10}^-}\pr{H_{n_0 - N^{-3/100} N}} 
\label{wHDef}\\
& H = \mathcal{M}_{\frac 1 2 r_{n_0 - N^{-3/100} N + 10}^-}\pr{H_{n_0 - N^{-3/100} N}} .
\label{HDef}
\end{align}
Since $H_{n_0 - N^{-3/100} N} \su H \su \widetilde H$, if we define $\Delta H := H_{n_0 + N^{-3/100} N} \setminus H$, then $\Delta H \su H_{n_0 + N^{-3/100} N} \setminus H_{n_0 - N^{-3/100} N}$.
Combining this observation with \eqref{HPHole} shows that
\begin{equation}
\label{DHBd}
\pr{\mu \times \nu }\pr{\Delta H} \lesssim N^{-3/100} L.
\end{equation}
In Section~\ref{tildeHFav}, we will estimate $\FavC(\widetilde H)$ and show that it is also small.
While the specific choice of $n_0 \in Z_0 $ is not used to estimate $\FavC(\widetilde H)$, it is used to control $(\mu\times \nu)(\Delta H)$, which will be important in Step 4, where we handle the remaining curve pairs.   
 \\

{\bf Step 2:} Our next stage of the decomposition uses the positive multiplicity curves.
For this step, we restrict to the range of indices to $Z_1 := \brac{n_0 - 0.9 N^{-3/100} N, n_0 + 0.9 N^{-3/100} N} \cap \Z$.
If $\pr{e, \al} \in \mathcal{E} \setminus \pr{H \cup \Delta H}$, then $\pr{e, \al} \notin H_{n_0 + N^{-3/100} N}$.
By Definition~\ref{hmSC}, this means that the curve $C_{e,\al}$ contains at most $N^{1/100}$ points of $E$ that are $r_{n_0 + N^{-3/100} N}^-$-separated.
Let $y \in E \cap C_{e,\al}$.
If $\abs{y - e} \sim r_{n_0 + N^{-3/100} N}^-$, then there can be at most $O\pr{N^{-7/100}N}$ indices $n$ such that $\abs{y - e} \in \brac{r_{n+ N^{-7/100}N}^-, r_{n-N^{-7/100}N}^+}$.
Note that if $\tilde y \in E \cap C_{e, \al}$ is another point for which $\abs{\tilde y - e} \approx \abs{y - e}$, so that $y$ and $\tilde y$ are not scale-separated, then $\tilde y$ is associated to roughly the same set of indices as $y$.
Repeating the argument for all of the scale-separated points in $E \cap C_{e,\al}$, we see that there are at most $N^{1/100} \times O\pr{ N^{-7/100} N}$ indices $n$ such that $\abs{y - e} \in \brac{r_{n+ N^{-7/100}N}^-, r_{n-N^{-7/100}N}^+}$ for some $y \in E \cap C_{e, \al}$.
Comparing this with Definition~\ref{pmSC}, we conclude that there are at most $O\pr{N^{-6/100} N}$ indices $n$ in our range such that $\pr{e, \al} \in P_n$.
It follows that
\begin{align*}
\sum_{n \in Z_1} \pr{\mu \times \nu }\pr{P_n \setminus \pr{H \cup \Delta H}}
&\lesssim {\pr{\mu \times \nu} \pr{\mathcal{E} \setminus \pr{H \cup \Delta H}}} N^{-6/100} N \\
&\le {\pr{\mu \times \nu} \pr{\mathcal{E}}} N^{-6/100} N
\le N^{-6/100} N L,
\end{align*}
where we have applied \eqref{prodMeas}.
The standard pigeonhole principle then implies that there exists $n_1 \in Z_1$ such that
\begin{equation}
\label{posBD}
\pr{\mu \times \nu }\pr{P_{n_1} \setminus \pr{H \cup \Delta H}} \lesssim N^{-3/100} L.
\end{equation} 
\\
Next, we make the elementary observation that $(P_{n_1} \setminus H )\su \pr{P_{n_1} \setminus \pr{H \cup \Delta H}} \sqcup \Delta H$ and apply  \eqref{DHBd} and \eqref{posBD} to conclude that
\begin{equation}
\label{posBD2}
\pr{\mu \times \nu }\pr{P_{n_1} \setminus H} \lesssim N^{-3/100} L.
\end{equation} \\

{\bf Step 3:} Now we use the high density curve strips to further decompose the curve pairs.
For this step, we restrict our range of indices to $Z_2 := \brac{n_1 - 0.9 N^{-7/100} N, n_1 + 0.9 N^{-7/100} N}\cap \Z$
and we observe that since $n_1 \in Z_1$ and $n_0 \in Z_0$, then
\begin{equation}
\label{n2Range}
\begin{aligned}
Z_2 &= \brac{n_1 - 0.9 N^{-7/100} N, n_1 + 0.9 N^{-7/100} N} \cap \Z \\
&\su  \brac{n_0 - 0.9 N^{-3/100} N - 0.9 N^{-7/100} N, n_0 + 0.9 N^{-3/100} N + 0.9 N^{-7/100} N} \cap \Z \\
 &\su \brac{0.1 N\pr{1 - 9 N^{-3/100} - 9 N^{-7/100}}, 0.9 N\pr{1 + N^{-3/100} + N^{-7/100} }} \cap \Z \\
&\su \brac{\frac{161}{2430}N, \frac{2269}{2430}N} \cap \Z,
\end{aligned} 
\end{equation}
where we have used the assumption that $N^{1/100} \ge 3$ to reach the last line.
Moreover, for any $n \in Z_2$, $n \pm N^{-10/100} N \in \set{1, 2, \ldots, N}$.
By \eqref{prodMeas} and Lemma~\ref{pHole}, there exists $n_2 \in Z_2$ so that
\begin{equation}
\label{DPHole}
\pr{\mu \times \nu }\pr{D_{n_2 + N^{-10/100} N} \setminus D_{n_2 - N^{-10/100} N}} \lesssim N^{-3/100} L.
\end{equation}
Following the constructions from the high density curves, we fix $n_2$ and define (standard) neighborhoods of the smaller set as 
\begin{align}
& \widetilde D = \mathcal{N}_{r_{n_2 - N^{-10/100} N + 10}^-}\pr{D_{n_2 - N^{-10/100} N}} 
\label{wDDef} \\
& D = \mathcal{N}_{\frac 1 2 r_{n_2 - N^{-10/100} N + 10}^-}\pr{D_{n_2 - N^{-10/100} N}} .
\label{DDef}
\end{align}
Since $D_{n_2 - N^{-10/100} N} \su D \su \widetilde D$, then with $\Delta D := D_{n_2 + N^{-10/100} N} \setminus D$, we have $\Delta D \su D_{n_2 + N^{-10/100} N} \setminus D_{n_2 - N^{-10/100} N}$ and we conclude from \eqref{DPHole} that
\begin{equation}
\label{DDBd}
\pr{\mu \times \nu} \pr{\Delta D} \lesssim N^{-3/100} L.
\end{equation}
We analyze $\FavC\pr{\widetilde D}$ in an upcoming section and show that it is also small, see Section~\ref{tildeDFav}. \\

{\bf Step 4:} 
We now handle the remaining curve pairs.  
First, we define an exceptional set of the finer scale elements of the sets we have just introduced.
Let
\begin{equation}
\label{DeltaDef}
\Delta = \Delta H \cup \pr{P_{n_1} \setminus H} \cup \Delta D.
\end{equation}
By combining \eqref{DHBd}, \eqref{posBD2}, and \eqref{DDBd}, we see that
\begin{equation}
\label{DelBd}
\pr{\mu \times \nu} \pr{\Delta} \lesssim N^{-3/100} L.
\end{equation}

Now, with $n_2$ as selected above, define the Lipschitz constant to be 
\begin{equation}
\label{LipConstant}
M_2 := \frac{10^4}{r_{n_2-200}^-}.
\end{equation}
By \eqref{n2Range} and the assumption that $N^{1/100} \ge 3$, we have $n_2 \ge \frac{161}{2430}N \gg 400$, so that $M_2$ is well-defined.
Moreover, since $n_2 \le \frac{2269}{2430}N  \ll N - 100$, then we may define curve pairs with respect to $n_2$, see Definition~\ref{curve_pairs_defn}.
If we set
\begin{equation}
\label{FDef}
F = \Cur_{n_2, M_2} \setminus \pr{\widetilde H \cup \widetilde D} \su \mathcal{E},
\end{equation} 
then 
\begin{equation}
\label{CurDecomp}
\Cur_{n_2, M_2} \su (\widetilde H \cup \widetilde D) \sqcup F.
\end{equation}
We use that $\Delta$ has a small measure to prove that $F$ has a small measure as well.
These details are available in Section~\ref{FAnalysis}.

\subsection{Summary of decomposition}

For our set $E \su \brac{0,1}^2$, we have the associated set $\mathcal{E} \su E \times A$ with the property that for each $\pr{e,\al} \in \mathcal{E}$, the projection $\Phi_\al(e)$ is non-empty.  
Moreover, $\pr{\mu \times \nu} \pr{\mathcal{E}} \le \mu(E) = L$.

We choose index scales sequentially via the pigeonhole principle where $n_0 \in Z_0 $, $n_1 \in Z_1$, and $n_2 \in Z_2$.
The Lipschitz constant $M_2$ is chosen to depend on $n_2$.

\begin{figure}[h]
\centering
\begin{tikzpicture}
\tikzstyle{level 1}=[level distance=1.5cm, sibling distance=4cm]
\tikzstyle{level 2}=[level distance=1.5cm, sibling distance=2cm]
  \node {$\mathcal{E}$}
    child { node {$\disp \NCur_{n_2, M_2}$} }
    child { node {$\Cur_{n_2, M_2}$}
      child { node {$\widetilde H$} } 
      child { node {$\widetilde D$} } 
      child { node {$F$} } };
\end{tikzpicture}
\qquad
\begin{tikzpicture}
\tikzstyle{level 1}=[level distance=2cm, sibling distance=2.5cm]
  \node {$N_\al \su E$}
    child { node {$G_\al$} }
    child { node {$K_\al$} };
\end{tikzpicture}
\caption{Visual representations of how we decompose $\mathcal{E}$ and $N_\al$ for analysis.}
\label{trees}
\end{figure}
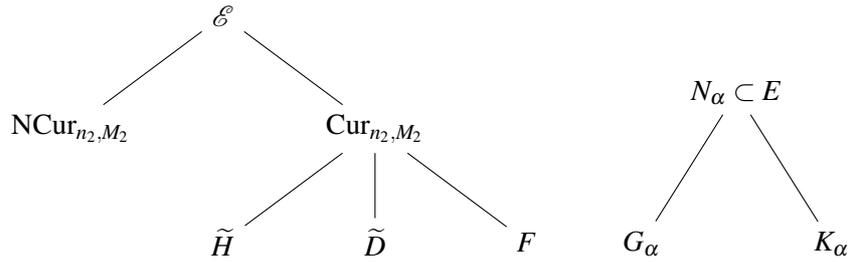

To decompose $\mathcal{E}$, we first write $\mathcal{E} = \Cur_{n_2, M_2} \sqcup \NCur_{n_2, M_2}$, where the union is disjoint.
Then we write $\Cur_{n_2, M_2} \su (\widetilde H \cup \widetilde D ) \sqcup F$ and define $\disp N_\al = \set{ e : \pr{e, \al} \in \NCur_{n_2, M_2}} \su E$, where $N_\al = G_\al \sqcup K_\al$.
A visual representation of this decomposition is given in Figure~\ref{trees}.
The next section will be devoted to estimating the measures of the non-curve elements, $N_\al$.
First we analyze the exceptional set $G_\al$, then we analyze the Lipschitz-like set $K_\al$.
The Favard curve lengths of $\widetilde H$ and $\widetilde D$ are estimated in Section~\ref{FavCirc}.
Section~\ref{FAnalysis} contains the analysis of $F$, which uses that $\Delta$ has a small measure.
The proof is completed in Section~\ref{conclusion}.

\smallskip

\section{The Non-Curve Elements}
\label{goodPoints}

Here we estimate the measures of the sets $N_\al = K_\al \sqcup G_\al$ defined in Subsection~\ref{ncintro} with $n = n_2$.  
We show that the exceptional set of low-density elements, $G_\al$, has small $\mu$-measure via a straight-forward Vitali covering argument.
Next, we turn to the main effort of this section,
 which is to show that the set $K_\al = N_\al \setminus G_\al$ is Lipschitz in nature.  
 It will follow then from the near unrectifiability assumption, \eqref{unrect}, that $K_\al$ also has small measure. 
 Specifically, we show that $\mu\pr{G_\al} \lesssim N^{-1/100} L$ and  $\mu\pr{K_\al} \lesssim N^{-1/100} L$, where $\mu$ denotes the $1$-dimensional Hausdorff measure $\mathcal{H}^1$ restricted to $E \su \brac{0,1}^2$.

\subsection{Estimating the measure of $G_\al$}

We prove that the $\mu$-measure of $G_\al$ is small.
This proof relies on a Vitali covering, as well as the definition of the exceptional set.

\begin{prop}[$G_\al$ has small measure]
\label{EalProp}
For $G_\al$ as defined in \eqref{EalDef}, we have $\mu\pr{G_\al} \lesssim N^{-1/100} L$.
\end{prop}

\begin{proof}
We first show that for any $B \in \mathcal{B}_{n_2}$,
\begin{align}
\label{JBEst}
\mu\pr{ \bigcup_{J \in \mathcal{J}_B} B \cap N_\al \cap \Phi_\al^{-1}\pr{J}} \lesssim N^{-1/100} r\pr{B}.
\end{align}
Recall that $\mathcal{J}_B$ is the set of all intervals $J \su \R$ with $\abs{J} \le r(B)$ satisfying the low density condition described by \eqref{lowDJ}.  
To prove \eqref{JBEst}, there is no loss in restricting to those intervals $J$ for which $B \cap \Phi_\al^{-1}\pr{J} \ne \emptyset$.
Moreover, by monotone convergence and the separability of $\R$, it suffices to show that \eqref{JBEst} holds for any finite subcollection of $\mathcal{J}_B$.
Let $\widetilde{\mathcal{J}}_B \su \mathcal{J}_B$ be such a finite subcollection.

By the Vitali covering theorem, there exists a finite, \textit{disjoint} collection $\disp \set{J_k}_{k=1}^K \su \mathcal{J}_B$ so that $\disp \set{5J_k}_{k=1}^K$ covers $\widetilde{\mathcal{J}}_B$ and $B \cap \Phi_\al^{-1}\pr{J_k} \ne \emptyset$ for each $k$.
By the defining property \eqref{lowDJ},
\begin{align*}
\mu \pr{ B \cap N_\al \cap \Phi_\al^{-1}\pr{5 J_k}} 
\le 10^{10} N^{-1/100} \abs{J_k} \quad \textrm{ for each } k = 1, \ldots K.
\end{align*}
It then follows from set containment and basic properties of measures that
\begin{align*}
\mu \pr{\bigcup_{J \in \widetilde{\mathcal{J}_B}} B \cap N_\al \cap \Phi_\al^{-1}\pr{J}}
&\le \mu \pr{\bigcup_{k=1}^K  B \cap N_\al \cap \Phi_\al^{-1}\pr{5 J_k}}
\le \sum_{k=1}^K \mu \pr{ B \cap N_\al \cap \Phi_\al^{-1}\pr{5 J_k}} \\
&\lesssim N^{-1/100} \sum_{k=1}^K \abs{J_k}.
\end{align*}
Since the $J_k$ are disjoint with $\abs{J_k} \le r\pr{B}$ and $\Phi_\al^{-1}\pr{J_k} \cap B \ne \emptyset$ for each $k$, then by the Lipschitz nature of the projection, $\disp \sum_{k=1}^K \abs{J_k} \lesssim r\pr{B}$, leading to \eqref{JBEst}.
It then follows from \eqref{EalDef}, \eqref{JBEst} and \eqref{BCover} that 
\begin{align*}
\mu\pr{G_\al}
&= \mu\pr{\bigcup_{B \in \mathcal{B}_{n_2}} \bigcup_{J \in \mathcal{J}_B} B \cap N_\al \cap \Phi_\al^{-1}\pr{J}}
\le \sum_{B \in \mathcal{B}_{n_2}} \mu\pr{\bigcup_{J \in \mathcal{J}_B} B \cap N_\al \cap \Phi_\al^{-1}\pr{J}} \\
&\lesssim N^{-1/100} \sum_{B \in \mathcal{B}_{n_2}} r\pr{B}
\lesssim N^{-1/100} L,
\end{align*}
as required.
\end{proof}

\subsection{Estimating $K_\al $, the Lipschitz-like elements}
\label{LipSet}

In this subsection, we consider $K_\al$ as in Subsection~\ref{ncintro} and show that for a small ball $B$, all of the elements in $B \cap K_\al$ lie in a narrow band about the graph of some Lipschitz function.  
We then invoke the near unrectifiability condition described by \eqref{unrect} to show that $K_\al$ must have small measure.

Roughly speaking, these arguments follow their counterparts from \cite{Tao09}.
However, given the nonlinear nature of our projections, many additional details and steps have been added.
In fact, this section contains many of the new and interesting ideas of the paper.

We first briefly recall the setup given in Section~\ref{setup_section}. 
For any $\pr{x, \al} \in \mathcal{E}$, $C_{x, \al}$ denotes the extended curve through $x$ defined by 
$$C_{x, \al} 
= \Phi_{\al,+}^{-1}\pr{\Phi_\al\pr{x}}  
= \pr{\al, \Phi_\al\pr{x}} + \mathcal{C}_+
= \set{\pr{\al + t, x_2 - \vp\pr{x_1 - \al} + \vp_+\pr{t}} : t \in I_+}
,$$
where 
$$\Phi_\al\pr{x} = x_2 - \vp\pr{x_1 - \al}, \quad \text{ and } \quad \mathcal{C}_+ = \set{\pr{t, \vp_+\pr{t}} : t \in I_+}. $$
Moreover, $I$ is a closed and bounded interval, $\vp$ is $C^1$, $\abs{\vp'} \le 1$, and $\vp'$ is $\la$-bilipschitz so that for any $s, t \in I$,
$$\la^{-1} \abs{s - t} \le \abs{\vp'\pr{s} - \vp'\pr{t}} \le \la \abs{s - t}.$$
The function $\vp_+$ extends $\vp$ to $I_+$, a $\de$-neighborhood of $I$, and maintains all of these properties.
For any $\pr{x, \al} \in \mathcal{E}$, $\vp_+\pr{x_1 - \al} = \vp\pr{x_1 - \al}$, so we may drop the cumbersome subscript notation in such settings.
Plugging $t = x_1- \al$, shows that  $x \in C_{x,\al}$ and that the slope of the tangent line to the curve $C_{ x,\al}$ at $x=(x_1, x_2)$ is $\varphi'( x_1-\al)$.  

Let $\omega_1^x$ denote the unit vector that points in the direction of the tangent vector, $\pr{1, \varphi'(x_1-\al)}$, and let $\omega_2^x$ be the clockwise rotation of $\omega^x_1$ through an angle of $\frac{\pi}{2}$.
Note that $\om_2$ corresponds to the vector $\om$ that appears in Lemma~\ref{bowtiecontainment}.

The ultimate aim of this section is to show that  $\mu\pr{K_\al} \le N^{-1/100}L$. 
To this end, we fix a ball $B \in \mathcal{B}_{n_2}$, show that  $\mu\pr{K_\al \cap B} \le N^{-1/100}r(B)$, and then sum over $B\in \mathcal{B}_{n_2}$ and apply \eqref{BCover} to reach the conclusion. 
Recall, $B \in \mathcal{B}_{n_2}$ implies that $r(B)\in [r_{n_2^-}, r_{n_2^+}]$, where $r(B)$ denotes the radius of $B$.  
For ease of notation, we write $n$ instead of $n_2$ within this section. 

\begin{rem}
\label{n2Smallness}
As pointed out in the previous section, $n_2 \gg 400$.
In particular, it follows from the separation of scales estimate in \eqref{Sos} that $r_{n}^\pm = r_{n_2}^\pm \le 2^{-400}$.
Moreover, $r_{n-200}^\pm = r_{n_2-200}^\pm \le 2^{-200}$.
\end{rem}

The following technical lemma serves as the main tool in showing that points in $K_\al$ are Lipschitz in nature.

\begin{lem}[$\Phi_\al$ cones]
\label{PhialLipLemm}  
Suppose $x\in  B\cap K_{\al}$.  
For each $y \in B\cap K_{\al}$, it holds that
\begin{equation}
\label{LipIneq}
\abs{\pr{x - y} \cdot \om^x_1 } \le \frac {\la M}{200} \abs{\Phi_\al\pr{x} - \Phi_\al\pr{y}} + \frac{1}{60} r_{n+2}^+. 
\end{equation}
\end{lem}

\begin{proof}
Recall that $\om_1^x = \frac{\pr{1, \vp'\pr{x_1 - \al}}}{\sqrt{1 + \brac{\vp'\pr{x_1 - \al}}^2}}$ and $\om_2^x = \frac{\pr{\vp'\pr{x_1 - \al}, -1}}{\sqrt{1 + \brac{\vp'\pr{x_1 - \al}}^2}}$.
For the duration of this proof, we will drop the superscripts and simply write $\om_1$ and $\om_2$, keeping in mind that the point $x$ is fixed.
For any point $y = \pr{y_1, y_2} \in \R^2$,
\begin{align*}
y \cdot \om_1 &= \frac{y_1 + y_2 \vp'\pr{x_1 - \al}}{\sqrt{1 + \brac{\vp'\pr{x_1 - \al}}^2}} 
\quad \text{ and } \quad 
y \cdot \om_2 = \frac{y_1 \vp'\pr{x_1 - \al} - y_2}{\sqrt{1 + \brac{\vp'\pr{x_1 - \al}}^2}},
\end{align*} 
so that
\begin{align}
\label{basisConversion}
y_1 = \frac{y \cdot \om_1 + \vp'\pr{x_1 - \al} y \cdot \om_2}{\sqrt{1 + \brac{\vp'\pr{x_1 - \al}}^2}}
\quad \text{ and } \quad 
y_2 &=\frac{\vp'\pr{x_1 - \al}  y \cdot \om_1 - y \cdot \om_2}{\sqrt{1 + \brac{\vp'\pr{x_1 - \al}}^2}}.
\end{align}

We show that for all $y \in B \cap N_\al$, either $\abs{\pr{x - y} \cdot \om_1} \le  \frac 1 {60} r_{n+2}^+$ or
$$\abs{\pr{x - y} \cdot \om_1} \le \frac{\la M}{200} \abs{\Phi_\al\pr{x} - \Phi_\al\pr{y}}.$$
Define the set
\begin{equation}
\label{BalDef}
B_\al = \set{y \in B \cap N_\al : \abs{\pr{x - y} \cdot \om_1} > \frac 1 {60} r_{n+2}^+ \text{ and } \abs{\pr{x - y} \cdot \om_1 } > \frac {\la M}{200} \abs{\Phi_\al\pr{x} - \Phi_\al\pr{y}}}.
\end{equation}
If $B_\al = \emptyset$, then we are done. 
So assume to the contrary that $B_\al \ne \emptyset$ and set
\begin{equation}
\label{Req} 
R = \sup \set{ \abs{\pr{x - y} \cdot \om_1}: y \in B_\al}. 
\end{equation}
Since $B_\al$ is assumed to be non-empty, then $R > \frac 1 {60} r_{n+2}^+ > 0$.
Choose $y \in B_\al$ so that $\abs{\pr{x - y} \cdot \om_1} \ge \frac R 2$.

By the definition of $N_\al$ given in Subsection~\ref{ncintro}, we have the following bounds on the curve double-sectors (defined in \eqref{circSectors} and pictured in Figure \ref{Bowties}) about $x$ and $y$:
$$\mu\pr{  \pr{\mathcal{X}_{x, \al}\pr{r, M/10^4} \setminus \mathcal{X}_{x, \al}\pr{r_{n+100}^-, M/10^4}}\cap N_\al} \le N^{-1/100} r/M,$$
and 
$$\mu\pr{  \pr{\mathcal{X}_{y, \al}\pr{r, M/10^4} \setminus \mathcal{X}_{y, \al}\pr{r_{n+100}^-, M/10^4}}\cap N_\al} \le N^{-1/100} r/M$$
for all $r_{n+100}^- \le r \le r_{n-100}^+$. 
Thus, if we denote the union by
\begin{equation}
\label{YrDef}
\begin{aligned}
Y_r &= \brac{\mathcal{X}_{x, \al}\pr{r, M/10^4} \setminus \mathcal{X}_{x, \al}\pr{r_{n+100}^-, M/10^4}} \bigcup  \brac{\mathcal{X}_{y, \al}\pr{r, M/10^4} \setminus \mathcal{X}_{y, \al}\pr{r_{n+100}^-, M/10^4}} \\
&= \bigcup_{e \in \set{x, y}} \set{z \in C_{e,\al'}  \su \R^2 : \abs{\al - \al'} \le \frac {10^4} M \text{ and }  r_{n+100}^- \le \abs{z - e} \le r},
\end{aligned}
\end{equation}
then 
\begin{equation}
\label{small_union} 
\mu(Y_r \cap N_\al) \le 2N^{-1/100} r/M 
\end{equation}
for all $r_{n+100}^- \le r \le r_{n-100}^+$. 
We will use these bounds with the choice $r = 5 \abs{\pr{x - y} \cdot  \om_1}$ and note that this $r$ is in the desired range.

We have arrived at the heart of the argument.  
The plan now is to define an interval $J$ so that $\Phi_\al(x), \Phi_\al\pr{y} \in J$ and then to use the bound in \eqref{small_union} to show that $J$ is of low density relative to $B$ and $\al$ in the sense of Definition~\ref{lowDJ_defn}.  
This will imply that $x \in G_\al$.
However, we assumed that $x\in K_\al := N_\al \backslash G_\al$, so this will give the desired contradiction.  

Set $w =  \frac{100}{\la M} \abs{\pr{x-y} \cdot \om_1}$ and define
$$J = \brac{  \frac{ \Phi_\al(x)+\Phi_\al(y)}{2} \,\,- \,\,w,\,\,  \,\,\frac{ \Phi_\al(x)+\Phi_\al(y)}{2}  \,\, +\,\, w  }.$$ 
Observe that $\Phi_\al(x),\Phi_\al(y) \in J$ if and only if $\abs{\Phi_\al\pr{x} - \Phi_\al\pr{y}} \le \frac {200} {\la M}  \abs{\pr{x-y} \cdot \om_1}$.
Since $y \in B_\al$, see \eqref{BalDef}, then this clearly holds.

We next verify that 
\begin{equation}
\label{containment}
 \Phi_\al^{-1}(5J)\cap B\cap N_\al \,\,\,   \subset \,\,\,Y_r \cap N_\al.
 \end{equation}
Choose an arbitrary point $z \in \Phi_\al^{-1}(5J)\cap B\cap N_\al$.
Since $z$ belongs to the strip  $\Phi_\al^{-1}(5J)$, we have $|\Phi_\al(z) - \frac{\Phi_\al(x) + \Phi_\al\pr{y}}2| \le 5w$. 
It follows from the triangle inequality and the bound from above that
\begin{equation}
\label{diff_proj}
\abs{\Phi_\al(x) - \Phi_\al(z)} 
\le \abs{\frac{\Phi_\al(x) + \Phi_\al\pr{y}}2- \Phi_\al(z)} + \frac{\abs{\Phi_\al\pr{x} - \Phi_\al(y) }}2
\le 6w = \frac{600}{\la M} \abs{\pr{x-y} \cdot \om_1}.
\end{equation}
The same bound holds for $\abs{\Phi_\al(y) - \Phi_\al(z)}$.

If $z \notin B_\al$, then by \eqref{BalDef} either $\abs{\pr{x-z} \cdot \om_1} \le \frac 1 {60} r_{n+2}^+ < \abs{\pr{x-y} \cdot \om_1}$ or
\begin{align*}
\abs{\pr{x-z} \cdot \om_1} \le \frac {\la M}{200} \abs{\Phi_\al\pr{x} - \Phi_\al\pr{z}}
\le 3 \abs{\pr{x-y} \cdot \om_1}.
\end{align*}

If $z \in B_\al$, then $\abs{\pr{x-z} \cdot \om_1} \le R\le 2 \abs{\pr{x-y} \cdot \om_1}$.
In every case, 
\begin{equation}
\label{xzom1}
\abs{\pr{x-z} \cdot \om_1} \le 3 \abs{\pr{x-y} \cdot \om_1},
\end{equation}
and then the triangle inequality shows that
\begin{align}
\label{yzom1}
\abs{\pr{y-z} \cdot \om_1}
&\le \abs{\pr{x-y} \cdot \om_1} + \abs{\pr{x-z} \cdot \om_1}
\le 4 \abs{\pr{x-y} \cdot \om_1}.
\end{align}

Now, either $\abs{\pr{x-z} \cdot \om_1} \geq \frac 1 2 \abs{\pr{x-y} \cdot \om_1}$ or $\abs{\pr{y-z} \cdot \om_1} \geq \frac 1 2 \abs{\pr{x-y} \cdot \om_1}$.
Assume first that $\abs{\pr{x-z} \cdot \om_1} \geq \frac 1 2 \abs{\pr{x-y} \cdot \om_1}$.
We will show that $z$ lies in a small curve sector about $x$. 
In particular, we verify that
\begin{equation}
\label{z in x-sector}
z \in \left(\mathcal{X}_{x, \al}\pr{5 \abs{\pr{x - y} \cdot \om_1}, M/10^4} \setminus \mathcal{X}_{x, \al}\pr{r_{n+100}^-, M/10^4}\right)\cap N_\al,
\end{equation}
which will imply that $z\in Y_r\cap N_\al$ for $r = 5 \abs{\pr{x - y} \cdot \om_1}$.

Since $x, z \in \mathcal{E}$, then $x_1 - \al$ and $z_1 - \al \in I$, so the mean value theorem shows that for some $h \in I$ between $x_1 - \al$ and $z_1 - \al$, 
\begin{align*}
\Phi_\al\pr{x} - \Phi_\al\pr{z}
&= x_2 - \vp\pr{x_1 - \al} - z_2 + \vp\pr{z_1 - \al}
= \pr{x_2 - z_2} - \vp'\pr{h}\pr{x_1 - z_1} \\
&= \frac{\vp'\pr{x_1 - \al}  \pr{x-z} \cdot \om_1 - \pr{x-z} \cdot \om_2}{\sqrt{1 + \brac{\vp'\pr{x_1 - \al}}^2}} 
- \vp'\pr{h}\frac{\pr{x-z} \cdot \om_1 + \vp'\pr{x_1 - \al} \pr{x-z} \cdot \om_2}{\sqrt{1 + \brac{\vp'\pr{x_1 - \al}}^2}},
\end{align*}
where we have used \eqref{basisConversion} with $y$ replaced by $x-z$.
Simplifying this expression shows that
\begin{equation}   
\label{PhialDiff}
\begin{aligned}
\pr{\Phi_\al(x)- \Phi_\al(z)} \sqrt{1 + \brac{\vp'\pr{x_1 - \al}}^2}
&= \brac{\vp'\pr{x_1 - \al}  - \vp'(h)} \pr{x-z} \cdot \om_1  \\
&- \brac{1 + \vp'\pr{x_1 - \al}\vp'(h)} \pr{x-z} \cdot \om_2,
\end{aligned}
\end{equation}
and in particular 
\begin{equation}
\label{expEquation}
\begin{aligned}
\abs{1 + \vp'\pr{x_1 - \al}\vp'(h) }  \abs{ \pr{x-z} \cdot \om_2} 
&\le \abs{\Phi_\al(x)- \Phi_\al(z)} \sqrt{1 + \brac{\vp'\pr{x_1 - \al}}^2} \\
&+\abs{ \vp'\pr{x_1 - \al}  - \vp'(h)} \abs{\pr{x-z} \cdot \om_1 }.
\end{aligned}
\end{equation}

To bound the left-hand-side of \eqref{expEquation} from below, observe that $1 - 2 \la r_n^+ \le \abs{1 + \vp'\pr{x_1 - \al}\vp'(h) } .$  Indeed, 
\begin{align*}
1 - 2 \la r_n^+ 
&\le 1 - \abs{\pr{\vp'(x_1 - \al) - \vp'(h)}\vp'(x_1 - \al)} \\
&\le 1 - \pr{\vp'(x_1 - \al) - \vp'(h)}\vp'(x_1 - \al) \\
&\le 1 - \pr{\vp'(x_1 - \al) - \vp'(h)}\vp'(x_1 - \al) + \brac{\vp'(x_1 - \al) }^2\\
&= 1 + \vp'(h)\vp'(x_1 - \al),
\end{align*}
where we have used
$\abs{\vp'\pr{x_1 - \al} - \vp'\pr{h}}\abs{\vp'\pr{x_1 - \al}} \le \abs{\vp'\pr{x_1 - \al} - \vp'\pr{h} } \le \la \abs{x_1 - z_1} \le 2 \la r_n^+$.

To bound the right-hand-side of \eqref{expEquation} from above, we use \eqref{diff_proj} and the assumption that $\abs{\pr{x-y} \cdot \om_1} \le2\abs{\pr{x-z} \cdot \om_1}$ to see that 
\begin{align*}   
& \abs{\Phi_\al(x)- \Phi_\al(z)} \sqrt{1 + \brac{\vp'\pr{x_1 - \al}}^2}
+ \abs{\vp'\pr{x_1 - \al}  - \vp'(h)} \abs{\pr{x-z} \cdot \om_1}  \\
\le& \frac{600 \sqrt 2}{\la M} \abs{\pr{x-y} \cdot \om_1} + 2 \la r_n^+ \abs{\pr{x-z} \cdot \om_1} \\
\le& \pr{\frac{1200 \sqrt 2}{\la M} + 2 \la r_n^+}\abs{\pr{x-z} \cdot \om_1},
\end{align*}
where we recall that $|\varphi'| \le 1$.

Putting the upper and lower bounds for the expression in \eqref{expEquation} together, we deduce that 
$$(1 - 2 \la r_n^+) \abs{ \pr{x-z} \cdot \om_2} \le \pr{\frac{1200 \sqrt 2}{\la M} + 2 \la r_n^+}\abs{\pr{x-z} \cdot \om_1}.$$

Remarks \ref{lambda assumption} and \ref{n2Smallness} imply that $2 \la r_n^+ \le 2^{-364}$ from which we conclude that
$$\abs{ \pr{x-z} \cdot \om_2} \le \frac{2000}{\la M}\abs{\pr{x-z} \cdot \om_1}.$$
Recalling the definition of $M$ from \eqref{LipConstant}, another application of Remarks \ref{lambda assumption} and \ref{n2Smallness} implies that $\frac{2 \la} M = 2 \cdot 10^{-4} r_{n-200}^- \la \le 2^{-164}10^{-4}$.
Letting $c_1$ be as in Lemma~\ref{bowtiecontainment}, it follows that $\frac{c_1}{\la} = \sqrt{8\brac{1 + \pr{1 + \frac{2\la}{M}}^2}} \le 5.$    
Therefore,
\begin{align*}   
\abs{\pr{x-z} \cdot \om_2} 
&\le \frac {2000} {\la M} \abs{\pr{x-z} \cdot \om_1} 
= \frac{c_1}{\la}\frac {2000} {c_1 M} \abs{\pr{x-z} \cdot \om_1} 
\le \frac {10^4} {c_1 M} \abs{x-z}.
\end{align*}
Moreover, by the assumption that $\abs{\pr{z-x} \cdot \om_1} \ge \frac 1 2 \abs{\pr{y-x} \cdot \om_1}$, since $y\in B_\al$, and by the scale-separation assumption, we have
$$\abs{x - z} \ge \abs{\pr{z-x} \cdot \om_1} \ge \frac 1 2 \abs{\pr{y-x} \cdot \om_1} > \frac{1}{120} r_{n+2}^+ > r_{n+100}^-$$
and
\begin{align*}
\abs{x - z} 
= \sqrt{\abs{\pr{x - z} \cdot \om_1}^2 + \abs{\pr{x - z} \cdot \om_2}^2 }
\le \sqrt{1 + \pr{\frac{2000}{\la M}}^2 } \abs{\pr{x - z} \cdot \om_1}
\le 5 \abs{\pr{x - y} \cdot \om_1},
\end{align*}
where we have used \eqref{xzom1}.
In particular, $\disp z \in X_{x, \om_2}\pr{5 \abs{\pr{x - y} \cdot \om_1}, c_1 M/ 10^4} \setminus B_{r_{n+100}^-}(x)$.
Since $\la \le 2^{35}$, then $\la^2 10^5 < 2^{87} < 2^{199}$, and we see that 
$$r = 5\abs{\pr{x - y} \cdot \om_1} \le 5 \abs{x-y} \le 10 r_n^+ \le \frac{r_{n - 200}^-}{2\la^2 10^4} = \frac{1}{2 \la^2 M}.$$
Furthermore, since $M > 10^5$ (see Definition~\ref{CircPairs}) and $r = 5\abs{\pr{x - y} \cdot \om_1} \le 5 \abs{x-y} \le 10 r_n^+ < 2^{-196}$, then
$$\de = 10^{-5} + 2^{-100} \ge \frac 1 M + r.$$
Then we can use the second containment in Lemma~\ref{bowtiecontainment} to deduce that \eqref{z in x-sector} holds.

On the other hand, if $\abs{\pr{y-z} \cdot \om_1} \geq \frac 1 2 \abs{\pr{x-y} \cdot \om_1}$, then we may repeat the arguments from above with $y$ in place of $x$ to show that $\abs{\pr{y-z} \cdot \om_2} \le \frac {10^4} {c_1 M} \abs{y-z}$, $\abs{y - z} \ge r_{n+100}^-$, and $\abs{y - z} \le 5 \abs{\pr{x - y} \cdot \om_1}$, where the last inequality uses \eqref{yzom1}.
Then we use Lemma~\ref{bowtiecontainment} again to deduce that 
$$z \in \pr{\mathcal{X}_{y, \al}\pr{5 \abs{\pr{x - y} \cdot \om_1}, M/10^4} \setminus \mathcal{X}_{y, \al}\pr{r_{n+100}^-, M/10^4}}\cap N_\al.$$
It follows that \eqref{containment} holds with $r = 5 \abs{\pr{x - y} \cdot \om_1}$.
Therefore,
\begin{align*}
\mu\pr{\Phi_\al^{-1}(5J)\cap B\cap N_\al} 
&\le \mu\pr{Y_{5 \abs{\pr{x - y} \cdot \om_1}} \cap N_\al}
\le 10 N^{-1/100} \abs{\pr{x - y} \cdot \om_1} /M
= \frac \la {20} N^{-1/100} \abs{J} \\
&\le 10^{10} N^{-1/100} \abs{J},
\end{align*}
where we have used that $\la \le 2^{35} \le 2 \cdot 10^{11}$.
This shows that $J$ is low density relative to $B$ and $\al$ from which we deduce that $x \in G_\al$.
As this is a contradiction to the assumption that $x \in K_\al$, the proof is complete.
\end{proof}

Now we use a Taylor approximation to straighten out the previous result. 

\begin{cor}[Standard cones]
\label{LipCor}  
Suppose $x\in  B\cap K_{\al}$.  
For each $y \in B\cap K_{\al}$, it holds that
\begin{equation}
\label{LipIneq2}
\abs{\pr{x - y} \cdot \om^x_1 } \le \frac {\la M} {50} \abs{\pr{x - y} \cdot \om^x_2} + \frac 1 {30} r_{n+2}^+. 
\end{equation}
\end{cor}

\begin{proof}
As in the previous proof, we will drop the superscripts and simply write $\om_1$ and $\om_2$, keeping in mind that the point $x$ is fixed.
As in the display preceding \eqref{PhialDiff}, there exists an $h \in I$ between $x_1 - \al$ and $y_1 - \al$ so that
\begin{align*}
\Phi_\al(x)- \Phi_\al(y)
= \frac{\vp'\pr{x_1 - \al}  - \vp'(h)}{\sqrt{1 + \brac{\vp'\pr{x_1 - \al}}^2}} \pr{x-y} \cdot \om_1  
- \frac{1 + \vp'\pr{x_1 - \al}\vp'(h)}{\sqrt{1 + \brac{\vp'\pr{x_1 - \al}}^2}} \pr{x-y} \cdot \om_2
\end{align*}
and then, by the Lipschitz condition on $\varphi'$ in \eqref{biLipCondition},
\begin{align*}
\abs{\Phi_\al(x)- \Phi_\al(y)}
&\le \la \abs{x_1 - y_1} \abs{\pr{x-y} \cdot \om_1} + 2 \abs{\pr{x-y} \cdot \om_2}.
\end{align*}
Substituting this bound into \eqref{LipIneq} and bounding $|x-y| \le 2r_n^+$ shows that 
\begin{align*}
\abs{\pr{x - y} \cdot \om_1 } 
\le \frac {\la M}{100} \la r_n^+ \abs{\pr{x-y} \cdot \om_1} + \frac{\la M}{100} \abs{\pr{x-y} \cdot \om_2} + \frac 1 {60} r_{n+2}^+. 
\end{align*}
Observe that $\frac M {100} \la^2 r_n^+ = 10^2 \la^2 \frac{r_n^+}{r_{n-200}^-} \le 2^{-123}  \ll \frac 1 2$, where we have used \eqref{LipConstant}, Remark \ref{lambda assumption}, and \eqref{Sos}.
Thus, we may rearrange to reach the conclusion.
\end{proof}

And here we show that this Lipschitz-result is independent of the basis vector we choose.
See Figure~\ref{ChangeofBasis} for an illustration of this result. 

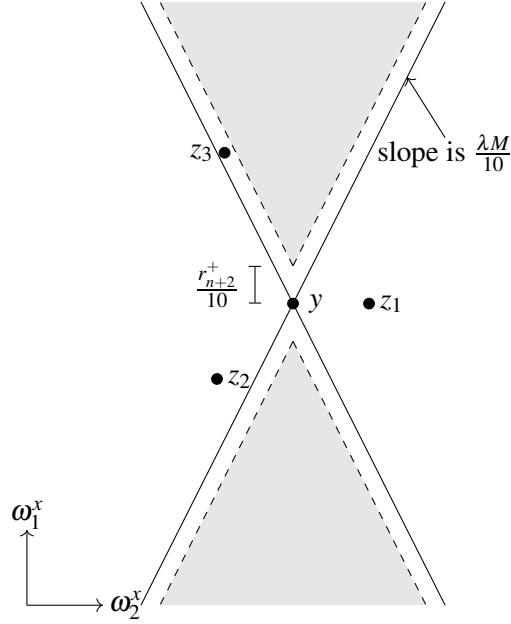
\begin{figure}[H]
\begin{tikzpicture}
\draw[->] (-0.5,-2) -- (-0.5, -1);
\draw[color=black] (-0.5, -0.8) node {$\om_1^x$};
\draw[->] (-0.5,-2) -- (0.5, -2);
\draw[color=black] (0.8, -2) node {$\om_2^x$};
\draw [fill=black] (3, 2) circle (2 pt);
\draw[color=black] (3.3, 2) node {$y$};
\draw [fill=black] (4, 2) circle (2 pt);
\draw[color=black] (4.3, 2) node {$z_1$};
\draw [fill=black] (2, 1) circle (2 pt);
\draw[color=black] (2.3, 1) node {$z_2$};
\draw [fill=black] (2.1, 4) circle (2 pt);
\draw[color=black] (1.8, 4) node {$z_3$};
\draw[color=black] (5, 4) node {$\text{slope is } \frac{\la M}{10}$};
\draw[->] (5, 4.2) -- (4.5, 5);
\draw (1, -2) -- (5, 6);
\draw (1, 6) -- (5, -2);

\draw[|-|] (2.5,2) -- (2.5, 2.5);
\draw[color=black] (2, 2.25) node {$\frac{r_{n+2}^+}{10}$};

\fill[black, opacity = 0.1]
 (1.25, -2) -- (3, 1.5) --
(3, 1.5) -- (4.75, -2) --
(4.75, -2) -- (1.25, -2) -- 
   cycle;
\fill[black, opacity = 0.1]
 (1.25, 6) -- (3, 2.5) --
(3, 2.5) -- (4.75, 6) --
(4.75, 6) -- (1.25, 6) -- 
   cycle;
\draw[dashed] (3, 2.5) -- (4.75, 6);
\draw[dashed] (3, 2.5) -- (1.25, 6);
\draw[dashed] (3, 1.5) -- (4.75, -2);
\draw[dashed] (3, 1.5) -- (1.25, -2);
\end{tikzpicture}
\centering
\caption{Corollary~\ref{LipCor2} shows that any point $z\in K_\al$ can lie in the non-shaded region around $y$.}
\label{ChangeofBasis}
\end{figure}
  
\begin{cor}[Standard cones with arbitrary basis]
\label{LipCor2}  
Fix $x\in  B\cap K_{\al}$.  
For each $y, z \in B\cap K_{\al}$, it holds that
\begin{equation}
\label{LipIneq3}
\abs{\pr{y - z} \cdot \om^x_1 } \le \frac{\la M}{10} \abs{\pr{y - z} \cdot \om^x_2} + \frac 1 {10} r_{n+2}^+. 
\end{equation}
\end{cor}

\begin{proof}
Recall that for any point $x \in B \cap N_\al$, $\om_1^x = \frac{\pr{1, \vp'\pr{x_1 - \al}}}{\sqrt{1 + \brac{\vp'\pr{x_1 - \al}}^2}}$ and $\om_2^x = \frac{\pr{\vp'\pr{x_1 - \al}, -1}}{\sqrt{1 + \brac{\vp'\pr{x_1 - \al}}^2}}$.
As we will be switching bases, we maintain the superscript notation.
Observe that
\begin{equation}
\label{xyz1Eq}
\begin{aligned}
\abs{\pr{y - z} \cdot \om^x_1 }
&= \abs{\pr{y - z} \cdot  \frac{\pr{1, \vp'\pr{y_1 - \al} + \vp'\pr{x_1 - \al} - \vp'\pr{y_1 - \al}}}{\sqrt{1 + \brac{\vp'\pr{x_1 - \al}}^2}} } \\
&\le \abs{\pr{y - z} \cdot \om_1^y } \sqrt{\frac{1 + \brac{\vp'\pr{y_1 - \al}}^2}{1 + \brac{\vp'\pr{x_1 - \al}}^2}}
+ \abs{\pr{y - z} \cdot  \frac{\pr{0, \vp'\pr{x_1 - \al} - \vp'\pr{y_1 - \al}}}{\sqrt{1 + \brac{\vp'\pr{x_1 - \al}}^2}} } \\
&\le \sqrt 2 \abs{\pr{y - z} \cdot \om_1^y }
+ \la \abs{x_1 - y_1} \abs{y_2 - z_2} \\
&\le \sqrt 2 \abs{\pr{y - z} \cdot \om_1^y }
+ 2 \la r_n^+ \abs{\frac{\vp'\pr{x_1 - \al}  \pr{y-z} \cdot \om^x_1 - \pr{y-z} \cdot \om^x_2}{\sqrt{1 + \brac{\vp'\pr{x_1 - \al}}^2}}} \\
&\le \sqrt 2 \abs{\pr{y - z} \cdot \om_1^y }
+ 2 \la r_n^+ \abs{\pr{y-z} \cdot \om^x_1}
+ 2 \la r_n^+ \abs{\pr{y-z} \cdot \om^x_2},\\
\end{aligned}
\end{equation}
where we have used \eqref{basisConversion} with $y$ replaced by $y - z$ to rewrite $y_2 - z_2$.
Combining this with Corollary~\ref{LipCor} shows that
\begin{equation*}
\begin{aligned}
\abs{\pr{y - z} \cdot \om^x_1 }
&\le \sqrt 2 \pr{\frac{\la M}{50} \abs{\pr{y - z} \cdot \om^y_2} + \frac 1 {30} r_{n+2}^+} 
+ 2 \la r_n^+ \abs{\pr{y-z} \cdot \om^x_1}
+ 2 \la r_n^+ \abs{\pr{y-z} \cdot \om^x_2} \\
&\le  \frac{\sqrt 2 \la M} {50} \abs{\pr{y - z} \cdot \om^y_2} 
+ 2 \la r_n^+ \abs{\pr{y-z} \cdot \om^x_1}
+ 2 \la r_n^+ \abs{\pr{y-z} \cdot \om^x_2}
+ \frac{1}{20} r_{n+2}^+.
\end{aligned}
\end{equation*}
A similar computation to \eqref{xyz1Eq} shows that
\begin{equation*}
\begin{aligned}
\abs{\pr{y - z} \cdot \om^y_2 }
&= \abs{\pr{y - z} \cdot  \frac{\pr{ \vp'\pr{x_1 - \al} + \vp'\pr{y_1 - \al} - \vp'\pr{x_1 - \al}, -1}}{\sqrt{1 + \brac{\vp'\pr{y_1 - \al}}^2}} } \\
&\le \abs{\pr{y - z} \cdot \om_2^x } \sqrt{\frac{1 + \brac{\vp'\pr{x_1 - \al}}^2}{1 + \brac{\vp'\pr{y_1 - \al}}^2}}
+ \abs{\pr{y - z} \cdot  \frac{\pr{\vp'\pr{y_1 - \al} - \vp'\pr{x_1 - \al}, 0}}{\sqrt{1 + \brac{\vp'\pr{y_1 - \al}}^2}} } \\
&\le \sqrt 2 \abs{\pr{y - z} \cdot \om_2^x }
+ 2 \la r_n^+ \abs{\pr{y-z} \cdot \om^x_1}
+ 2 \la r_n^+ \abs{\pr{y-z} \cdot \om^x_2}.
\end{aligned}
\end{equation*}
Substituting this bound into the previous expression gives
\begin{equation*}
\begin{aligned}
\abs{\pr{y - z} \cdot \om^x_1 }
&\le  \frac{\sqrt 2 \la M} {50} \pr{\sqrt 2 \abs{\pr{y - z} \cdot \om_2^x }
+ 2 \la r_n^+ \abs{\pr{y-z} \cdot \om^x_1}
+ 2 \la r_n^+ \abs{\pr{y-z} \cdot \om^x_2}} \\
&+ 2 \la r_n^+ \abs{\pr{y-z} \cdot \om^x_1}
+ 2 \la r_n^+ \abs{\pr{y-z} \cdot \om^x_2}
+ \frac{1}{20} r_{n+2}^+ \\
&\le \brac{\frac{\la M} {25} + 2 \la r_n^+ \pr{  \frac{\la M} {35} + 1}} \abs{\pr{y-z} \cdot \om^x_2} 
+ 2 \la r_n^+ \pr{\frac{\la M} {35} + 1} \abs{\pr{y-z} \cdot \om^x_1} 
+ \frac{1}{20} r_{n+2}^+.
\end{aligned}
\end{equation*}
 Observe that $2 \la r_n^+ \pr{ \frac{\la M} {35} + 1} = 2 \la r_n^+ \pr{ \frac{\la 10^4} {35 r_{n-200}^-} + 1} \le 2^{10} \la^2  \frac{r_n^+} {r_{n-200}^-} \le  2^{-120}$, where we have used \eqref{LipConstant}, Remark \ref{lambda assumption}, and \eqref{Sos}.
In particular, $2 \la r_n^+ \pr{ \frac{\la M} {35} + 1}  \le \frac 1 2 \le \frac {\la M}{100}$, and then
\begin{equation*}
\begin{aligned}
\abs{\pr{y - z} \cdot \om^x_1 }
&\le \frac{\la M}{10} \abs{\pr{y-z} \cdot \om^x_2} 
+ \frac {1}{10} r_{n+2}^+,
\end{aligned}
\end{equation*}
as required.
\end{proof}

We now define a piecewise linear function and use Corollary~\ref{LipCor2} to show that it is Lipschitz with an appropriate constant. 
Recall that $B$ is a ball with radius $r \in \brac{r_n^-, r_n^+}$.
For fixed $x \in B \cap K_\al$, let $\om_1 = \om_1^x$ and $\om_2 = \om_2^x$, so that $\pr{\om_1, \om_2}$ defines a frame on $B$.
Set $w = \frac{7 r_{n+2}^+}{4 \la M}$, then divide $B$ into strips $S$ of width $w$ that are parallel to the direction $\om_1$.
There will be on the order of $r(B)/w$ strips in this collection. 
Next we select and name a subset of the strips from the collection so that at least half of them are non-empty.

When we say that a strip $S$ is \textit{to the left of} a strip $T$, we mean with respect to the direction $\om_2$.
That is, for any $y \in S$ and any $z \in T$, $y \cdot \om_2 \le z \cdot \om_2$.
Similarly, we say that $T$ is \textit{to the right of} $S$ if $S$ is to the left of $T$.
If we say that $S$ \textit{abuts} $T$, then we mean that $S$ and $T$ are adjacent strips, meaning that they share a boundary line that runs parallel to $\om_1$.

Starting from the leftmost strip in our collection and moving to the right, let $S_1^-$ denote the first strip for which $S_1^- \cap K_\al \ne \emptyset$.
Let $S_1^+$ denote the strip that abuts and is to the right of $S_1^-$.
Let $S_2^-$ denote the next strip that is to the right of $S_1^+$ for which $S_2^- \cap K_\al \ne \emptyset$. 
Set $S_2^+$ to be the strip that abuts and is to the right of $S_2^-$.
Continuing on, we have a collection of strips $\set{S_i^-}_{i=1}^N$ for which $S_i^- \cap K_\al \ne \emptyset$ for each $i = 1, \ldots, N$.
We also have a collection $\set{S_i^+}_{i=1}^N$ so that $S_i^+$ abuts and is to the right of $S_i^-$ for each $i = 1, \ldots, N$.

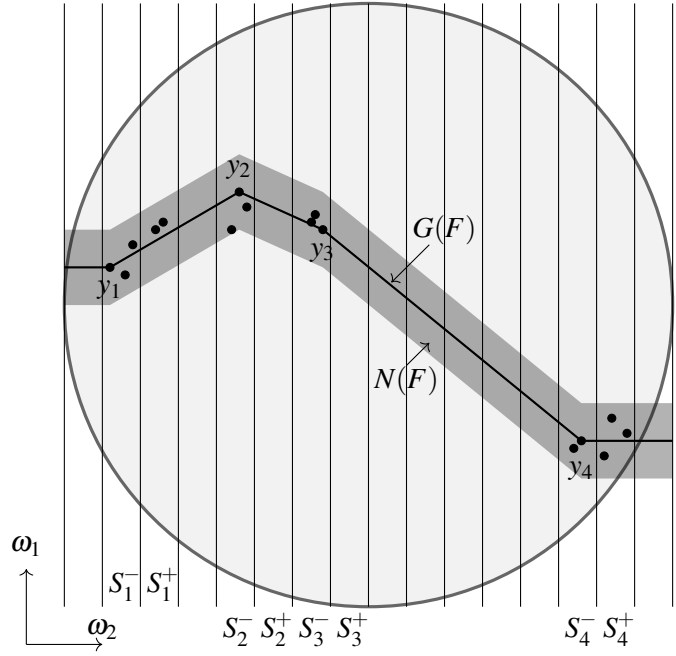
\begin{figure}[H]
\begin{tikzpicture}
\draw[->] (-4.5,-4.5) -- (-4.5, -3.5);
\draw[color=black] (-4.5, -3.2) node {$\om_1$};
\draw[->] (-4.5,-4.5) -- (-3.5, -4.5);
\draw[color=black] (-3.5, -4.3) node {$\om_2$};
\filldraw[color=black!60, fill=black!5, very thick](0,0) circle (4);
\draw (-4,-4) -- (-4, 4);
\draw (-3.5,-4) -- (-3.5, 4);
\draw (-3,-4) -- (-3, 4);
\draw (-2.5,-4) -- (-2.5, 4);
\draw (-2,-4) -- (-2, 4);
\draw (-1.5,-4) -- (-1.5, 4);
\draw (-1,-4) -- (-1, 4);
\draw (-0.5,-4) -- (-0.5, 4);
\draw (0,-4) -- (0, 4);
\draw (4,-4) -- (4, 4);
\draw (3.5,-4) -- (3.5, 4);
\draw (3,-4) -- (3, 4);
\draw (2.5,-4) -- (2.5, 4);
\draw (2,-4) -- (2, 4);
\draw (1.5,-4) -- (1.5, 4);
\draw (1,-4) -- (1, 4);
\draw (0.5,-4) -- (0.5, 4);

\fill[black, opacity = 0.3]
(-4, 1) -- (-3.4, 1) --
(-3.4, 1) -- (-1.7, 2) --
(-1.7, 2) -- (-0.6, 1.5) -- 
(-0.6, 1.5) -- (2.8, -1.3) --
(2.8, -1.3) -- (4, -1.3) --
(4, -1.3) -- (4, -2.3) --
(4, -2.3) -- (2.8, -2.3) --
(2.8, -2.3) -- (-0.6, 0.5) --
(-0.6, 0.5) -- (-1.7, 1) --
(-3.4, 0) -- (-4, 0) --
(-4, 0) -- (-4, 1) -- 
   cycle;

\draw [fill=black] (-3.4, 0.5) circle (1.5 pt);
\draw [fill=black] (-3.1, 0.8) circle (1.5 pt);
\draw [fill=black] (-3.2, 0.4) circle (1.5 pt);
\draw [fill=black] (-2.7, 1.1) circle (1.5 pt);
\draw [fill=black] (-2.8, 1) circle (1.5 pt);
\draw [fill=black] (-1.8, 1) circle (1.5 pt);
\draw [fill=black] (-1.7, 1.5) circle (1.5 pt);
\draw [fill=black] (-1.6, 1.3) circle (1.5 pt);
\draw [fill=black] (-0.7, 1.2) circle (1.5 pt);
\draw [fill=black] (-0.75, 1.1) circle (1.5 pt);
\draw [fill=black] (-0.6, 1) circle (1.5 pt);
\draw [fill=black] (2.8, -1.8) circle (1.5 pt);
\draw [fill=black] (2.7, -1.9) circle (1.5 pt);
\draw [fill=black] (3.1, -2) circle (1.5 pt);
\draw [fill=black] (3.4, -1.7) circle (1.5 pt);
\draw [fill=black] (3.2, -1.5) circle (1.5 pt);

\draw[fill=black, thick] (-4, 0.5) -- (-3.4, 0.5);
\draw[fill=black, thick] (-3.4, 0.5) -- (-1.7, 1.5);
\draw[fill=black, thick] (-1.7, 1.5) -- (-0.6, 1);
\draw[fill=black, thick] (-0.6, 1) -- (2.8, -1.8);
\draw[fill=black, thick] (2.8, -1.8) -- (4, -1.8);

\draw[color=black] (-3.4, 0.2) node {$y_1$};
\draw[color=black] (-1.7, 1.8) node {$y_2$};
\draw[color=black] (-0.6, 0.7) node {$y_3$};
\draw[color=black] (2.8, -2.2) node {$y_4$};

\draw[color=black] (1, 1) node {$G(F)$};
\draw[->] (0.8,0.8) -- (0.3, 0.3);

\draw[color=black] (0.5, -1) node {$N(F)$};
\draw[->] (0.5, -0.8) -- (0.8, -0.5);

\draw[color=black] (-3.2, -3.7) node {$S_1^-$};
\draw[color=black] (-2.7, -3.7) node {$S_1^+$};
\draw[color=black] (-1.7, -4.3) node {$S_2^-$};
\draw[color=black] (-1.2, -4.3) node {$S_2^+$};
\draw[color=black] (-0.7, -4.3) node {$S_3^-$};
\draw[color=black] (-0.2, -4.3) node {$S_3^+$};
\draw[color=black] (2.8, -4.3) node {$S_4^-$};
\draw[color=black] (3.3, -4.3) node {$S_4^+$};
\end{tikzpicture}
\centering
\caption{When $B$ is divided into strips, each $S_i^-$ is non-empty, while $S_i^+$ is immediately to the right and could be empty.
Points $y_i$ are chosen from each $S_i^-$, then connected to make the graph of $F$, $G\pr{F}$.
The neighborhood of the graph, $N\pr{F}$, is the shaded region.}
\label{WedgeCor}
\end{figure}
\leavevmode

Now we use a selection of points from $\set{S_i^-}_{i=1}^N$ to define a piecewise linear function.
For each $i = 1, \ldots, N$, choose a point $y_i \in S_i^- \cap K_\al$.
By connecting these points with straight lines, we define a piecewise linear function over the interval $U := B \cdot \om_2$:
\begin{equation}
\label{FDefn}
F\pr{t} = \left\{\begin{array}{ll}  
y_1 \cdot \om_1 & t \le y_1 \cdot \om_2 \\
y_i \cdot \om_1 + \frac{\pr{y_{i+1} - y_i} \cdot \om_1}{\pr{y_{i+1} - y_i} \cdot \om_2} \pr{t - y_i \cdot \om_2} & y_i \cdot \om_2 < t \le y_{i+1} \cdot \om_2 \\
y_N \cdot \om_1 & t > y_N \cdot \om_2
\end{array}\right..
\end{equation}
Note that, by construction, $\pr{y_{i+1} - y_i} \cdot \om_2 \ge w$ for each $i = 2, \ldots, N$.
We use $m_i$ to denote the slope over the $i^{\rm{th}}$ interval.
That is,
\begin{equation}
\label{miDefn}
m_i= \frac{\pr{y_{i+1} - y_i} \cdot \om_1}{\pr{y_{i+1} - y_i} \cdot \om_2}.
\end{equation}
  
We first observe that $F$ is Lipschitz.

\begin{lem}[Lipschitz function]
\label{FLip}
The function $F$ defined in \eqref{FDefn} is Lipschitz continuous with constant at most $\frac{11\la M}{70}$.
\end{lem}

\begin{proof}
Since $F$ is piecewise linear, we simply need to find an upper bound for each of the slopes $m_i$ defined in \eqref{miDefn}.
Observe that by Corollary~\ref{LipCor2},
\begin{align}
\label{slopeBound}
\abs{m_i} 
&= \abs{\frac{\pr{y_{i+1} - y_i} \cdot \om_1}{\pr{y_{i+1} - y_i} \cdot \om_2} }
\le \abs{\frac{\frac{\la M}{10} \pr{y_{i+1} - y_i} \cdot \om_2 + \frac{r_{n+2}^+}{10}}{\pr{y_{i+1} - y_i} \cdot \om_2}}
\le \frac{\la M}{10} + \frac{r_{n+2}^+}{10 w}
= \frac{11\la M}{70} ,
\end{align}
where we have used the lower bound on the denominator.
This shows that $F$ is $\frac{11\la M}{70}$-Lipschitz, as required.
\end{proof}

Now we'll show that all points in $B \cap K_\al$ lie in a small region around the graph of $F$.
Let 
$$G\pr{F} = \set{t \om_2 + F\pr{t} \om_1 : t \in U},$$ 
the graph of $F$ over $\pr{\om_2, \om_1}$.
Then define a neighborhood (measured with respect to the $\om_1$ direction) of the graph of $F$ to be 
\begin{equation}
\label{FNeig}
N\pr{F} = \set{z \in B : \abs{F\pr{z \cdot \om_2} - z \cdot \om_1} \le r_{n+2}^+}.
\end{equation}

\begin{lem}[Neighborhood containment]
\label{graphNLemma}
For $F$ as given in \eqref{FDefn} and it neighborhood defined in \eqref{FNeig}, it holds that $B \cap K_\al \su N\pr{F}$.
\end{lem}

\begin{proof}
Let $z \in B \cap K_\al$.
Then $z \cdot \om_2 \in U$, so there are three possibilities: 
\begin{enumerate}
\item there exists $i \in \set{1, \ldots, N-1}$ so that $z \cdot \om_2 \in \pb{y_i \cdot \om_2, y_{i+1} \cdot \om_2}$;
\item $z \cdot \om_2 \le y_1 \cdot \om_2$; or
\item $z \cdot \om_2 > y_N \cdot \om_2$.
\end{enumerate}
Assume that the first case holds.
Then $z \in S_i^- \cup S_i^+ \cup S_{i+1}^-$.
If $z \in S_i^- \cup S_i^+$, then $\abs{\pr{z - y_i} \cdot \om_2} \le 2w$.
Using \eqref{FDefn}, we see from Corollary~\ref{LipCor2} and \eqref{slopeBound} that
\begin{align*}
\abs{F\pr{z \cdot \om_2} - z \cdot \om_1} 
&= \abs{\pr{y_i - z} \cdot \om_1 + m_i \pr{z  - y_i} \cdot \om_2} \\
&\le \abs{\pr{y_i - z} \cdot \om_1} + \abs{m_i} \abs{ \pr{z  - y_i} \cdot \om_2} \\
&\le \abs{\frac{\la M}{10} \pr{y_{i} - z} \cdot \om_2 + \frac{r_{n+2}^+}{10}} + \frac{11 \la M}{70} \abs{ \pr{z  - y_i} \cdot \om_2} \\
&\le \frac{\la M}{10} 2 w + \frac{r_{n+2}^+}{10} + \frac{11 \la M}{70} 2w 
= r_{n+2}^+.
\end{align*}
On the other hand, if $z \in S_{i+1}^-$, then $\abs{\pr{y_{i+1} - z} \cdot \om_2} \le w$.
Note that we can rearrange \eqref{FDefn} to get that
$$F\pr{z \cdot \om_2} = y_{i+1} \cdot \om_1 - m_i \pr{y_{i+1} - z} \cdot \om_2.$$
Proceeding as above, we see that
\begin{align*}
\abs{F\pr{z \cdot \om_2} - z \cdot \om_1} 
&\le \abs{\frac{\la M}{10} \pr{y_{i+1} - z} \cdot \om_2 + \frac{r_{n+2}^+}{10}} + \frac{11 \la M}{70} \abs{ \pr{y_{i+1} - z} \cdot \om_2} 
\le \frac{11 r_{n+2}^+}{20}.
\end{align*}
If the second case holds, then $z \in S_1^-$ and $\abs{\pr{y_{1} - z} \cdot \om_2} \le w$.
Since $F\pr{z \cdot \om_2} = y_1 \cdot \om_1$ by \eqref{FDefn}, then Corollary~\ref{LipCor2} shows that
\begin{align*}
\abs{F\pr{z \cdot \om_2} - z \cdot \om_1} 
&= \abs{\pr{y_{1} - z} \cdot \om_1}
\le \abs{\frac{\la M}{10} \pr{y_1 - z} \cdot \om_2 + \frac{r_{n+2}^+}{10}} 
\le \frac{\la M}{10} w + \frac{r_{n+2}^+}{10} 
= \frac{11 r_{n+2}^+}{40}. 
\end{align*}
Finally, if the third case holds, then $z \in S_N^- \cup S_N^+$ and $\abs{\pr{y_{N} - z} \cdot \om_2} \le 2w$.
Using $F\pr{z \cdot \om_2} = y_{N} \cdot \om_1$ by \eqref{FDefn}, Corollary~\ref{LipCor2} again shows that
\begin{align*}
\abs{F\pr{z \cdot \om_2} - z \cdot \om_1} 
&= \abs{\pr{y_N - z} \cdot \om_1}
\le \abs{\frac{\la M}{10} \pr{y_N - z} \cdot \om_2 + \frac{r_{n+2}^+}{10}} 
\le \frac{\la M}{10} 2w + \frac{r_{n+2}^+}{10} 
= \frac{9 r_{n+2}^+}{20}. 
\end{align*}
In all cases, we have shown that $\abs{F\pr{z \cdot \om_2} - z \cdot \om_1} \le r_{n+2, +}$, proving that $B \cap K_\al \su N\pr{F}$.
\end{proof}

These observations lead us to the following.

\begin{prop}[$K_\al$ has small measure]
\label{KalProp}
For $K_\al = N_\al \setminus G_\al$, we have $\mu\pr{K_\al} \lesssim N^{-1/100} L$.
\end{prop}

\begin{proof}
By combining our previous results, we see that 
\begin{align*}
\mu\pr{B \cap K_\al} 
&\le \mu\pr{N\pr{F}}  &\text{(by Lemma~\ref{graphNLemma})} \\
&\le R_E\pr{r_{n+2}^+, r_n^-, \la M} r\pr{B}  &\text{(by~Lemma~\ref{FLip} and the definition of $R_E$)} \\
&\le R_E\pr{r_{n+2}^+, r_n^-, \frac 1 {r_n^-}} r\pr{B} & \text{(by \eqref{LipConstant})} \\
&\le N^{-1/100} r\pr{B} & \text{(by \eqref{unrect}).} 
\end{align*}
Summing over all $B$, we see that $\mu\pr{K_\al} \lesssim N^{-1/100} L$.
\end{proof}
%

\smallskip

\section{Favard Curve Length Bounds}
\label{FavCirc}

In this section, we estimate the Favard curve lengths of the sets $\widetilde H$ and $\widetilde D$.
For $\widetilde H$, we rely on a Fubini-type argument reminiscent of \cite[Lemma 8.4]{Fal85} and \cite[Theorem 7.7]{Mat95}.
To estimate the curve projection of $\widetilde D$, we use the Hardy-Littlewood maximal inequality, which may  be interpeted as a quantification of the Lebesgue differentiation theorem.

\subsection{Estimating the Favard curve length of $\widetilde H$}
\label{tildeHFav}

Recall that $\widetilde H$ is a parametric neighborhood of the collection of points corresponding to high-multiplicity curves, recall Definition~\ref{hmSC}.
See \eqref{wHDef} for the precise definition of $\widetilde H$.
Here we use a Fubini-type argument to establish the following bound.

\begin{prop}[$\widetilde H$ has small Favard length]
\label{FavHProp}
For $\widetilde H$ as given in \eqref{wHDef}, it holds that $\FavC\pr{\widetilde H} \lesssim N^{-1/100} L$.
\end{prop}

\begin{proof}
Set $n_0^- = n_0 - N^{-3/100}N$, an abbreviation for the scale around which we are working. 
Since $H_{n_0^-} \su \mathcal{E}$, then it follows from the definition of $\widetilde H$ given in \eqref{wHDef} that $\widetilde{H} \su \R^2 \times A_0$, where $A_0$ is the $r_{n_0^- + 10}^-$ neighborhood of the bounded interval $A$.
By the definition of Favard curve length described by \eqref{3dFavC}, since $A_0$ is a bounded interval, then it suffices to show that for any $\al \in A_{0}$, 
\begin{equation*}
\abs{\Phi_{\al}\pr{\widetilde H_{\al}}} \lesssim N^{-1/100} L,
\end{equation*}
where $\widetilde H_{\al} = \set{e : \pr{e, \al} \in \widetilde H}$. 

Fix $\al \in A_0$.
Given $\pr{e, \al} \in \widetilde H$, there exists $(e,\al')\in H_{n_0^-}$ so that $\abs{\al - \al'} \le r_{n_0^-+10}^-$.
If $\pr{e, \al'} \in H_{n_0^-}$, then by Definition~\ref{hmSC}, there exists a set of points $\set{e_j}_{j=1}^{N^{1/100}} \su C_{e, \al'} \cap E$ that are $r_{n_0^-}$-separated.
By the uniform length bound given in \eqref{Ulb}, there exists a collection $\mathcal{B}$ of balls $B$ that cover $E$ for which $r\pr{B} \in \brac{r_{n_0^-+5}^-, r_{n_0^-+5}^+}$ and $\disp \sum_{B \in \mathcal{B}} r\pr{B} \lesssim L$.
Since $\mathcal{B}$ covers $E$, then for each point $e_j$, there exists a ball $B \in \mathcal{B}$ such that $e_j \in B$.
As $r\pr{B} \le r_{n_0^-+5}^+ \le \frac 1 {32} r_{n_0^-}$ by \eqref{Sos}, then each $e_j$ belongs to a distinct ball.
Therefore, there are $N^{1/100}$ distinct, non-overlapping balls $\set{B_j}_{j=1}^{N^{1/100}}$ associated to the pair $\pr{e, \al'} \in H_{n_0^-}$.

Note that the distance between a fixed point on one curve and another curve over the same parameter range is bounded above by the distance between their centers, see Figure~\ref{tildeHProof} (left).
Therefore, for any $j$, since $e_j \in C_{e, \al'}$, then
\begin{align*}
\dist\pr{e_j, C_{e,\al}} 
&\le \dist\pr{\pr{\al, \Phi_\al(e)}, \pr{\al', \Phi_{\al',+}(e)}} 
= \sqrt{ \abs{\al - \al'}^2 +\pr{\Phi_\al(e) - \Phi_{\al',+}(e)}^2} \\
&= \sqrt{ \abs{\al - \al'}^2 +\pr{\vp\pr{e_1 - \al} - \vp_+\pr{e_1 - \al'}}^2}
\le \sqrt{ \abs{\al - \al'}^2 +\abs{\al - \al'}^2} \\
&< 2 r_{n_0^-+5}^-
\le 2^{-4} r\pr{B_j},
\end{align*}
where we have used that $\vp_+$ is $1$-Lipschitz and the separation of scales \eqref{Sos}.
Since each $B_j$ intersects $C_{e,\al'}$ at $e_j$, then $5 B_j$ intersects $C_{e,\al}$ along a small curve of length $\gtrsim r\pr{B_j}$, see Figure~\ref{tildeHProof} (right).
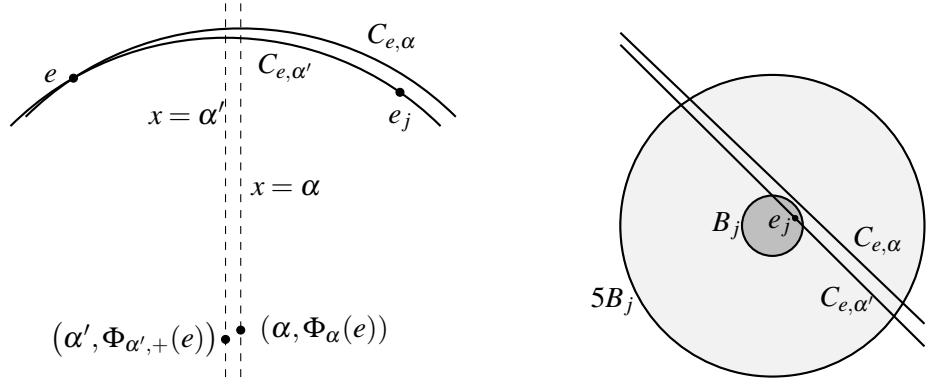
\begin{figure}
\begin{tikzpicture}
\draw[thick] (2.8284, 2.8284) arc (45:135:4cm);
\draw[thick] (3.0284, 2.9518) arc (45:135:4cm);
\draw [fill=black] ( -2,3.4641) circle (1.5pt);
\draw[color=black] (-2.3,3.5) node {$e$};
\draw [fill=black] ( 2.2943, 3.2766) circle (1.5pt);
\draw[color=black] (2.3,2.9) node {$e_j$};
\draw [fill=black] (0.2,0.1234) circle (1.5pt);
\draw[color=black] (1.3,.1234) node {$\pr{\al, \Phi_{\al}(e)}$};
\draw [fill=black] (0,0) circle (1.5pt);
\draw[color=black] (-1.2,0) node {$\pr{\al', \Phi_{\al', +}(e)}$};
\draw[dashed] (0,-.5) -- (0, 4.5);
\draw[dashed] (0.2,-.5) -- (0.2, 4.5);
\draw[color=black] (-.5,3) node {$x = \al'$};
\draw[color=black] (0.8,2) node {$x = \al$};tikzpicture shade 
\draw[color=black] (0.8,3.6) node {$C_{e,\al'}$};
\draw[color=black] (2.2,4) node {$C_{e,\al}$};
\end{tikzpicture}
\qquad\qquad
\begin{tikzpicture}
\draw (2, 0) arc (0:360:2cm);
\filldraw[color=black, fill=black!5, thick](0,0) circle (2 cm);
\filldraw[color=black, fill=black!25, thick](0,0) circle (.4 cm);
\draw [fill=black] (0.3,0.10) circle (1pt);
\draw[thick] (2,-1.6) -- (-2, 2.4);
\draw[thick] (2,-1.3) -- (-2, 2.56);
\draw[color=black] (0.1,0) node {$e_j$};
\draw[color=black] (-0.6,0) node {$B_j$};
\draw[color=black] (-2.1,-1) node {$5 B_j$};
\draw[color=black] (1,-1) node {$C_{e, \al'}$};
\draw[color=black] (1.4,-.2) node {$C_{e, \al}$};
\end{tikzpicture}
\centering
\caption{The images of $C_{e, \al}$ and $C_{e, \al'}$ when $\al$ and $\al'$ are close (left), and a visualization of the behavior near some $e_j$ (right).}
\label{tildeHProof}
\end{figure}
That is, for each $j$
$$\int_{C_{e,\al}} 1_{5 B_j} d \mathcal{H}^1 \gtrsim r\pr{B_j}.$$
The separation of the points ensures separation of the balls, and so we may sum over $j$ to get that for any $e \in \widetilde H_\al$,
\begin{equation}
\label{density_sum}
\int_{C_{e,\al}} \sum_{j=1}^{N^{1/100}} \frac{1_{5 B_j} }{r\pr{B_j}} d \mathcal{H}^1 \gtrsim N^{1/100}.
\end{equation}
Observe that 
$$\R^2 \supset \bigsqcup_{\be \in \R} \Phi_{\al,+}^{-1}\pr{\be} \supset \bigsqcup_{\be \in \Phi_\al\pr{\widetilde H_\al}} \Phi_{\al,+}^{-1}\pr{\be} = \bigsqcup_{\be \in \Phi_\al\pr{\widetilde H_\al}} C_{e, \al},$$
where $e = e_\be \in \widetilde H_\al$ is some point for which $\Phi_\al(e) = \be$.
Indeed, if $\be \in \Phi_\al\pr{\widetilde H_\al}$, then there exists $e \in \widetilde H_\al$ so that $\Phi_\al\pr{e} = \be$ and it follows that $\Phi_{\al,+}^{-1}\pr{\be} = C_{e, \al}$.
An application of Fubini's theorem then shows that
\begin{align*}
\int_{\R^2} \sum_{j=1}^{N^{1/100}} \frac{1_{5 B_j} }{r\pr{B_j}} dx 
&\ge \int_{\Phi_{\al}\pr{\widetilde H_{\al}}} \int_{C_{e, \al}} \sum_{j=1}^{N^{1/100}} \frac{1_{5 B_j} }{r\pr{B_j}} d \mathcal{H}^1 \, d\be \\
&\gtrsim \int_{\Phi_{\al}\pr{\widetilde H_{\al}}} N^{1/100} \, d\be
= N^{1/100} \abs{\Phi_{\al}\pr{\widetilde H_{\al}}},
\end{align*}
where we have applied \eqref{density_sum}.
On the other hand,
\begin{align*}
\sum_{j=1}^{N^{1/100}} \int_{\R^2} \frac{1}{r\pr{B_j}} 1_{5 B_j} dx
\lesssim \sum_{j=1}^{N^{1/100}} r\pr{B_j}
\le \sum_{B \in \mathcal{B}} r\pr{B} \lesssim L,
\end{align*}
so by combining the previous two inequalities, we see that $\abs{\Phi_{\al}\pr{\widetilde H_{\al}}} \lesssim N^{-1/100} L$,
as required.
\end{proof}

\subsection{Estimating the Favard curve length of $\widetilde D$}
\label{tildeDFav}

Recall that $\widetilde D$ is a standard neighborhood of the collection of points that are contained in high-density curve strips, see Definition~\ref{hdSS}.
The set $\widetilde D$ is defined in \eqref{wDDef}.
Here we use an argument involving the Hardy-Littlewood maximal inequality to establish the following bound.

\begin{prop}[$\widetilde D$ has small Favard length]
\label{FavDProp}
For $\widetilde D$ as given in \eqref{wDDef}, it holds that $\FavC\pr{\widetilde D} \lesssim N^{-1/100} L$.
\end{prop}

\begin{proof}
For brevity, set $n_2^- = n_2 - N^{-10/100} N$.
Since $D_{n_2^-} \su \mathcal{E}$, then it follows from the definition of $\widetilde D$ given in \eqref{wDDef} that $\widetilde D \su \R^2 \times A_2$, where $A_2$ denotes the $r_{n_2^-+10}^-$ neighborhood of the interval $A$.
In particular, we may use \eqref{3dFavC} to define $\FavC\pr{\widetilde D}$.
Since $A_2$ is a bounded interval, it suffices to show that for any $\al \in A_{2}$, 
$$\abs{\Phi_\al\pr{\widetilde D_\al}} \lesssim N^{-1/100} L,$$
where $\widetilde D_\al = \set{e : \pr{e, \al} \in \widetilde D}$.

Fix $\al \in A_{2}$.
Let $\pr{e, \al} \in \widetilde D$ and then $e \in \widetilde{D}_\al$.
Assume that $\pr{e, \al} \in \mathcal{E}$ for otherwise $\Phi_\al(e) = \emptyset$ and there is nothing to show.
By the definition of $\widetilde D$ given in \eqref{wDDef}, there exists $\pr{e', \al'} \in D_{n_2^-}$ so that $\norm{\pr{e-e', \al - \al'}} \le r_{n_2^-+10}^-$.
Since $\pr{e', \al'} \in D_{n_2^-}$, then by Definition~\ref{hdSS}, there exists an interval $J \su \R$ containing $\Phi_{\al'}\pr{e'}$ with $\abs{J} \ge r_{n_2^-}^-$ and $\mu\pr{\Phi^{-1}_{\al',+}\pr{J}} \ge N^{1/100} \abs{J}$. 
Since $e' \in \Phi_{\al'}^{-1}\pr{J}$ implies that $e' \in \Phi_{\al',+}^{-1}\pr{J}$, and $\norm{e - e'} \le r_{n_2^-+10}^- \le 2^{-10} \abs{J}$, then $e \in \Phi^{-1}_{\al',+}\pr{2J}$.
Moreover,
\begin{equation}
\label{JProp}
\mu\pr{\Phi^{-1}_{\al',+}\pr{2J}} 
\ge \mu\pr{\Phi^{-1}_{\al',+}\pr{J}} 
\ge N^{1/100} \abs{J}.
\end{equation}

\nid {\bf Claim:} $\Phi_\al\pr{\Phi_{\al',+}^{-1}\pr{2J}} \su 3J$. 
In particular, since $e \in \Phi^{-1}_{\al',+}\pr{2J}$, then $\Phi_\al(e) \in 3J$.\\ 
If $p \in \Phi_{\al',+}^{-1}\pr{2J}$, then $p = \pr{\al' + t, \be + \vp_+(t)}$ for some $t \in I_+$ and some $\be \in 2J$.
By \eqref{PhialDefn}, if $t + \al' - \al \in I$, then $\Phi_\al\pr{p} = \be + \vp_+(t) - \vp\pr{\al' + t - \al}$; otherwise, the projection is empty.
Since $\vp$ is $1$-Lipschitz, $\abs{\Phi_\al\pr{p} - \be} \le \abs{\al - \al'}$.
Since our separation of scales implies that, $\abs{\al - \al'} \le 2^{-10} r_{n_2^-}^- \le 2^{-10} \abs{J}$, then $\abs{\Phi_\al\pr{p} - \be} \le 2^{-10} \abs{J}$ and the claim follows.

Let $\mu_1$ be the pushforward of the measure $\mu$ to $\R$ under the projection $\Phi_\al$.
Since $\mu\pr{\R^2} \le L$, then $\mu_1\pr{\R}= \mu(\Phi_\al^{-1}(\R) )\lesssim L$.
The Hardy-Littlewood maximal function of $\mu_1$ is defined by 
$$\mathrm{M} \mu_1\pr{x} = \sup_{r > 0} \frac{1}{2r} \mu_1\pr{\brac{x-r, x+r}}.$$
Since $\Phi_\al(e) \in 3J$ by the claim, then $3J \su \brac{\Phi_\al(e) - r, \Phi_\al(e) + r}$ for some $r \le 3\abs{J}$.
Then
\begin{align*}
\mu_1\pr{\brac{\Phi_\al(e) - r, \Phi_\al(e) + r}}
\ge \mu_1\pr{3 J}
\ge \mu_1\pr{\Phi_\al\pr{\Phi_{\al',+}^{-1}\pr{2J}} }
= \mu\pr{\Phi^{-1}_{\al',+}\pr{2J}}
\ge N^{1/100} \abs{J},
\end{align*}
where we have applied set containment, the claim, the definition of $\mu_1$, and  \eqref{JProp}. 
It follows that,
\begin{align*}
\mathrm{M} \mu_1\pr{\Phi_\al(e)} \ge \frac{\mu_1\pr{\brac{\Phi_\al(e) - r, \Phi_\al(e) + r}}}{2r} 
\ge \frac{N^{1/100}} 6.
\end{align*}

The Hardy-Littlewood maximal inequality for measures states that $\abs{\set{\be : \mathrm{M} \mu_1\pr{\be} \ge \la}} \le \frac 1 \la \mu_1\pr{\R}$.
In particular, $\abs{\set{\be : \mathrm{M} \mu_1\pr{\be} \ge \frac{N^{1/100}} 6}} \lesssim N^{-1/100} L$.
Since we showed that $\mathrm{M} \mu_1\pr{\Phi_\al(e)} \ge \frac{N^{1/100}} 6$ for an arbitrary $e \in \widetilde D_\al$ for which $\Phi_\al(e) \ne \emptyset$, then 
$$\Phi_\al\pr{\widetilde D_\al} \su \set{\be : \mathrm{M} \mu_1\pr{\be} \ge \frac{N^{1/100}} 6}$$
and we conclude that $\abs{\Phi_\al\pr{\widetilde D_\al}} \lesssim N^{-1/100} L$, as required.
\end{proof}

We conclude this section by pointing out that although Lemmas \ref{FavHProp} and \ref{FavDProp} are proved for $\widetilde H$ and $\widetilde D$ as defined in \eqref{wHDef} and \eqref{wDDef}, respectively, the selection of scales is not important to the arguments that we made in this section.
For example, if we set $H' = \mathcal{M}_{r_{n + 10}^-}\pr{H_{n}}$ for any $n \in \brac{0.1 N, 0.9 N}$, the arguments in Lemma~\ref{FavHProp} show that $\FavC\pr{H'} \lesssim N^{-1/100} L$.
While the specific choices of $n_0$ and $n_2$ are not used in this section, these choices were important for guaranteeing that $\Delta$ has small measure. 
This fact will be important in the next section where we analyze the measure of $F$.

\smallskip

\section{The Remaining Curve Pairs}
\label{FAnalysis}

Within this section, we estimate the $\pr{\mu \times \nu}$-measure of the set $F$.
Recall that $F$ is defined to contain the curve pairs that don't belong to $\widetilde H$ or $\widetilde D$, the sets that have already been analyzed, see \eqref{FDef}. 
Our main tool in this endeavor is the following technical lemma, 
which can be viewed as a type of density theorem on slices.   
This lemma will in turn be used to control the measure of $F$ by the measure of $\Delta$. 

\begin{lem}[Parameter mass]
\label{DeltaClusterLemma}
Let $\Delta \su \mathcal{E}$ be as in \eqref{DeltaDef}.
For any $\pr{e, \al} \in F$ as given in \eqref{FDef}, define the set
\begin{center}
{$\disp \Om_{e, \al} = \set{\al' \in A : \abs{\al - \al'} \le 10^4 r_{n_2-200}^- \textrm{ and } \pr{e, \al'} \in \Delta}.$}
\end{center}
Then $\nu\pr{\Om_{e, \al}} \gtrsim N^{-2/100} r_{n_2-200}^-$.
\end{lem}

An overview  of the proof is as follows. 
To establish a lower bound on $ \nu\pr{\Om_{e, \al}}$, we first introduce a cover of the parameter set $\Om_{e, \al}$ by a finite collection of carefully chosen intervals.  
We then observe that, to each such interval $I_k$, the corresponding curve double-sector, $\set{y \in C_{e, \al'} : \al' \in I_k}$ intersected with the small ball, $B_r(e)$, is contained in a narrow strip described by the inverse image of an $r$-dilate of $I_k$,  an interval $J_k$.
Further, by arranging matters so that this curve strip is \textbf{not} of high density, we can bound the $\mu$-measure of the truncated curve sector above by the $\nu$-measure of the interval $J_k$ in the parameter space.  
The fact that $(e,\al)$ is a curve pair is crucial.  
Now we proceed with the proof.

\begin{proof}
Let $\pr{e,\al} \in F \su \Cur_{n_2, M_2}$.
If $\Om_{e, \al} = I_*:= \set{\al' \in A: \abs{\al - \al'} \le 10^4 r_{n_2-200}^-}$, then the result is immediate, so assume that $\Om_{e, \al} \ne I_*$.
By Definition~\ref{CircPairs} and \eqref{LipConstant}, there exists $r \in \brac{r_{n_2+100}^-, r_{n_2-100}^+}$ so that
\begin{equation}
\label{lowBdSectors}
\mu\pr{\mathcal{X}_{e, \al}\pr{r, 1/r_{n_2-200}^-} \setminus \mathcal{X}_{e, \al}\pr{r_{n_2+100}^-, 1/r_{n_2-200}^-}} 
> 10^{-4} r N^{-1/100} r_{n_2-200}^-.
\end{equation}
We fix this scale $r$.
\smallskip

Since $\Delta$ is the union of three parametrically closed sets, then $\Delta$ itself is parametrically closed, see Appendix \ref{apx} for details.
It follows that the set $\Om_{e, \al}$ is a compact subset of $A$.
Indeed, as a subset of $A$, $\Om_{e, \al}$ is bounded; further, since $\Om_{e, \al}$ is the continuous image of the closed set $\Delta$ under the projection map in the last coordinate, it is also closed.
As such, given an open cover of intervals contained in $\set{\al' : \abs{\al - \al'} \le 10^5 r_{n_2-200}^-}$, we may select a finite subcover of intervals $I_1, I_2, \ldots, I_K$ that cover $\Om_{e, \al}$ with the following properties:
\begin{enumerate}
\item[(i)]\label{coverProp}
$\disp \nu\pr{\bigcup_{k=1}^K I_k} \lesssim \nu\pr{\Om_{e, \al}}.$
\item[(ii)]
By taking an appropriate intersection, there is no loss in assuming that each interval $I_k$  is contained in the set $I_*$.
\item[(iii)]
By concatenating the overlapping parts, we may also assume that all intervals $I_k$ are disjoint with length at least $\max\set{r_{n_2+N^{-10/100}N}^-/r, r}$.
\item[(iv)]
Since $\Om_{e,\al} \ne I_*$, then by enlarging each set slightly, we may also assume that each $I_k$ contains a point in the complement of $\Om_{e,\al}$.
\end{enumerate}

We establish a lower bound on $\disp \nu\pr{\bigcup_{k=1}^K I_k}$.  
Before proceeding, we gather two observations.   
First, observe that since $\pr{e,\al} \in F$, then for every $\al' \in I_*$, $\pr{e, \al'} \notin H\cup D$. 
Indeed, by the definition of $F$ given in \eqref{FDef}, $\pr{e,\al} \notin \widetilde H$.
By \eqref{Sos} and the assumption that $n_2 \ge n_0 - 0.9 \pr{N^{-3/100} + N^{-7/100}} N$, we have $10^4 r_{n_2-200}^- \le \frac 1 2 r_{n_0 - N^{-3/100}N+10}^-$.
Therefore, since $\pr{e,\al} \notin \widetilde H$, see \eqref{wHDef}, then for every $\al' \in I_*$, $\pr{e,\al'} \notin H$, see \eqref{HDef}.
By analogy, since $\pr{e,\al} \notin \widetilde D$, see \eqref{wDDef}, and $10^4 r_{n_2-200}^- \le \frac 1 2 r_{n_2 - N^{-10/100}N+10}^-$, we also have that for every $\al'$ in $I_*$, $\pr{e, \al'} \notin D$, where $D$ is as defined in \eqref{DDef}.

Next, observe that 
\begin{equation}
\label{EcapCurvesandBalls}
E \cap C_{e, \al'} \cap \pr{B_r(e) \setminus B_{r_{n_2+ 100}^-}(e)} = \emptyset
\textrm{ for any } \al' \in I_*\setminus \bigcup_{k=1}^K I_k.
\end{equation}
Let $\al' \in I_*$ be a point that is not contained in $I_1, I_2, \ldots, I_K$.
By the discussion in the previous paragraph, $\pr{e, \al'} \notin H$.
Since $I_1, I_2, \ldots, I_K$ forms a cover for $\Om_{e, \al}$, then $\pr{e, \al'} \notin \Delta$.
From \eqref{DeltaDef}, we conclude that $\pr{e, \al'} \notin P_{n_1}$.
Looking at Definition~\ref{pmSC}, this means that 
$$E \cap C_{e, \al'} \cap \pr{B_{r_{n_1-N^{-7/100}N}^+}(e) \setminus B_{r_{n_1+ N^{-7/100}N}^-}(e)} = \emptyset.$$
Since $n_2 \in \brac{n_1 - 0.9 N^{-7/100} N, n_1 + 0.9 N^{-7/100} N}$ and we may assume that $N^{-7/100} N \ge 1000$, then $B_r(e) \setminus B_{r_{n_2+ 100}^-}(e) \su B_{r_{n_1-N^{-7/100}N}^+}(e) \setminus B_{r_{n_1+ N^{-7/100}N}^-}(e)$. 
In other words, \eqref{EcapCurvesandBalls} is verified. 
\smallskip

With these observations in tow, we turn to establishing upper and lower bounds on the measure of the set $\set{y \in C_{e, \al'} : \al' \in I_*} \cap \pr{B_r(e) \setminus B_{r_{n_2+ 100}^-}(e)}$.  
From \eqref{circSectors}, it is immediate that
$$\mathcal{X}_{e, \al}\pr{r, 1/r_{n_2-200}^-} \setminus \mathcal{X}_{e, \al}\pr{r_{n_2+100}^-, 1/r_{n_2-200}^-} \su \set{y \in C_{e, \al'} : \al' \in I_*} \cap \pr{B_r(e) \setminus B_{r_{n_2+ 100}^-}(e)}.$$
Since $\mu$ is supported on $E$, we deduce from \eqref{EcapCurvesandBalls} that
\begin{align*}
& \mu\pr{\mathcal{X}_{e, \al}\pr{r, 1/r_{n_2-200}^-} \setminus \mathcal{X}_{e, \al}\pr{r_{n_2+100}^-, 1/r_{n_2-200}^-}} \\
\le& \mu\pr{\set{y \in C_{e, \al'} : \al' \in \bigcup_{k=1}^K I_k} \cap \pr{B_r(e) \setminus B_{r_{n_2+ 100}^-}(e)}} \\
\le& \sum_{k=1}^K  \mu\pr{\set{y \in C_{e, \al'} : \al' \in I_k} \cap \pr{B_r(e) \setminus B_{r_{n_2+ 100}^-}(e)}}.
\end{align*}
Combining this inequality with \eqref{lowBdSectors} shows that
\begin{align}
\label{lowBdSectCons}
10^{-4} r N^{-1/100} r_{n_2-200}^-
\le \sum_{k=1}^K  \mu\pr{\set{y \in C_{e, \al'} : \al' \in I_k} \cap \pr{B_r(e) \setminus B_{r_{n_2+ 100}^-}(e)}}.
\end{align}

Now choose $k \in \set{1, \ldots, K}$.
Observe that for any $\al_k \in I_k$, since $I_k \su \brac{\al_k - \abs{I_k}, \al_k + \abs{I_k}}$, then we have
\begin{equation*}
\set{y \in C_{e, \al'} : \al' \in I_k} \cap \pr{B_r(e) \setminus B_{r_{n_2+ 100}^-}(e)} 
\su \set{y \in C_{e, \al'} : \abs{\al' - \al_k} \le \abs{I_k}} \cap B_r(e) 
= \mathcal{X}_{e, \al}\pr{r, \frac 1 {\abs{I_k}}},
\end{equation*}
where we used \eqref{circSectors} to reach the last equality. 
To apply Corollary~\ref{bowtiestripcontainment}, we need to check that $\abs{I_k} + r < \min\set{10^{-5} + 2^{-100}, \frac 1{\sqrt 2 \la}}$.
By property (ii) and Remark \ref{n2Smallness}, $\abs{I_k} \le \abs{I_*} \le 2 \cdot 10^4 r_{n_2 - 200}^- \le 2^{-185}$ while $r \le r_{n_2-100}^+ \le 2^{-300}$.
Since $\la \le 2^{35}$ by Remark \ref{lambda assumption}, then the assumptions of Corollary~\ref{bowtiestripcontainment} are satisfied and we see that
\begin{equation}
\label{WedgeInStrip}
\begin{aligned}
\set{y \in C_{e, \al'} : \al' \in I_k} \cap \pr{B_r(e) \setminus B_{r_{n_2+ 100}^-}(e)} 
\su \mathcal{X}_{e, \al}\pr{r, \frac 1 {\abs{I_k}}}
\su \Phi_{\al_k,+}^{-1}\pr{J_k},
\end{aligned}
\end{equation}
where $J_k \su \R$ is an interval centered about $\Phi_{\al_k}(e)$ with $\abs{J_k} = 4 \sqrt 2 \la\pr{\abs{I_k} + r} r$.
In fact, by the size condition on $I_k$ described in property (iii), $r_{n_2 + N^{-10/100}N}^- \le \abs{J_k} \lesssim \abs{I_k} r$.
See Figure~\ref{BowtieInCurveStrip} for a depiction of this second set inclusion.

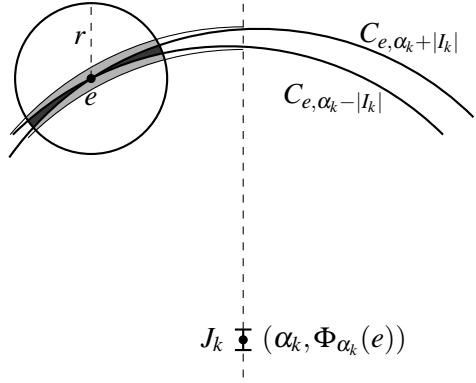
\begin{figure}
\begin{tikzpicture}
\fill[black, opacity = 0.3]
(-2.72644, 2.77687) arc (133: 104.9:4cm) -- 
(-1.03189, 3.71461) arc (14.51: 31.23:1cm) -- 
(-1.14494, 3.98264) arc (106.6: 135.6: 4cm) -- 
   cycle;
\draw[thick] (3.0284, 2.9518) arc (45:145:4cm);
\draw (0, 4.15) arc (90:140:4cm);
\draw (0, 3.85) arc (90:135:4cm);
\draw[thick] (2.6284, 2.720386) arc (45:135:4cm);
\draw[thick] (-1, 3.4641) arc (0:360:1cm);
\draw [fill=black] ( -2,3.4641) circle (1.5pt);
\draw[color=black] (-2,3.2) node {$e$};
\draw[dashed] ( -2,3.4641) -- ( -2, 4.4641);
\draw[color=black] (-2.15, 4) node {$r$};
\fill[black, opacity = 0.7]
(-2,3.4641) arc (123.367: 109: 4cm) -- 
(-1.10262, 3.90536) arc (26.18: 19.56:1cm) -- 
(-1.05773,3.79894) arc (102.38: 116.74:4cm) -- 
   cycle;
\fill[black, opacity = 0.7]
(-2,3.4641) arc (123.367: 137.728: 4cm) -- 
(-2.75985, 2.814) arc (220.55: 213.925: 1cm) -- 
(-2.82977, 2.906) arc (131.105: 116.74: 4cm) -- 
   cycle;
\draw [fill=black] (0,0) circle (1.5pt);
\draw[color=black] (-0.4, 0) node {$J_k$};
\draw[color=black] (1.2, 0) node {$\pr{\al_k, \Phi_{\al_k}(e)}$};
\draw[dashed] (0,-.5) -- (0, 4.5);
\draw[thick, -|] (0,0) -- (0, .15);
\draw[thick, -|] (0,0) -- (0, -.15);
\draw[color=black] (1.2,3.2) node {$C_{e,\al_k - \abs{I_k}}$};
\draw[color=black] (2.2,4) node {$C_{e,\al_k + \abs{I_k}}$};
\end{tikzpicture}
\centering
\caption{The darkly-shaded curve double-sector around $e$ is defined by $\set{y \in C_{e, \al'} : \abs{\al' - \al_k} \le \abs{I_k}} \cap B_r(e)$, while the lightly-shaded region that bounds it is the intersection of the curve strip $\Phi_{\al_k,+}^{-1}\pr{J_k}$ with $B_r(e)$.
The interval $J_k$ is pictured along $x = \al_k$ and is centered at the point $\pr{\al_k, \Phi_{\al_k}(e)}$.}
\label{BowtieInCurveStrip}
\end{figure}

By property (iv), there exists $\al_k \in I_k$ such that $\pr{e, \al_k} \notin \Delta$.
Referring to \eqref{DeltaDef}, this means that $\pr{e, \al_k} \notin \Delta D$.
As discussed above, $\pr{e, \al_k} \notin D$ as well, so we deduce that $\pr{e, \al_k} \notin D_{n_2 + N^{-10/100}N}$.
From Definition~\ref{hdSS} applied with $n= n_2 + N^{-10/100}N$, we recall that whenever $J \su \R$ with $\Phi_{\al_k}(e) \in J$ and $\abs{J} \ge r_{n_2 + N^{-10/100}N}^-$, we have $\mu\pr{\Phi_{\al_k,+}^{-1}\pr{J}} \le N^{1/100} \abs{J}$.
In particular, if we take $J =J_k$
as defined in the previous paragraph, an interval centered about $\Phi_{\al_k}(e)$ with $r_{n_2+N^{-10/100}N}^- \le \abs{J_k} \lesssim r \abs{I_k}$, then it follows from \eqref{WedgeInStrip} that

\begin{equation}
\label{muUpperBd}
\begin{aligned}
\mu\pr{\set{y \in C_{e, \al'} : \al' \in I_k} \cap \pr{B_r(e) \setminus B_{r_{n_2+ 100}^-}(e)}} 
&\le \mu\pr{\Phi_{\al_k,+}^{-1}\pr{J_{k}}} \le N^{1/100} \abs{J_k} \\
&\lesssim N^{1/100} r \abs{I_k}.
\end{aligned}
\end{equation}
Combining \eqref{lowBdSectCons} with \eqref{muUpperBd} shows that
\begin{align*}
10^{-4} N^{-1/100} r_{n_2-200}^-
&\lesssim N^{1/100} \sum_{k=1}^K \abs{I_k}
\simeq
N^{1/100} \nu\pr{\bigcup_{k=1}^K I_k}
\lesssim N^{1/100} \nu\pr{\Om_{e,\al}},
\end{align*}
where we have applied property (i) in the last inequality. 
The conclusion of Lemma~\ref{DeltaClusterLemma} follows.
\end{proof}

Using the lemma, we show that the points in $F$ are close to many points in $\Delta$.
The main tool in this proof is another application of the Hardy-Littlewood maximal inequality.

\begin{prop}[$F$ has small measure]
\label{FProp}
For $F$ as given in \eqref{DeltaDef}, $\pr{\mu \times \nu}\pr{F} \lesssim N^{-1/100} L$.
\end{prop}

\begin{proof}
Given $e \in E$, define $\nu_{e, \Delta}$ to be the restriction of the measure $\nu$ to the set $\Delta_e :=\set{\al' \in A : \pr{e, \al'} \in \Delta}$.
If $\pr{e, \al} \in F$, then Lemma~\ref{DeltaClusterLemma} shows that 
\begin{align*}
\mathrm{M} \nu_{e, \Delta}\pr{\al} 
&:= \sup_{r > 0} \frac{\nu_{e, \Delta}\pr{\brac{\al - r, \al + r}}}{2r}
\ge \frac{\nu_{e, \Delta}\pr{\brac{\al - 10^4 r_{n_2-200}^-, \al + 10^4 r_{n_2-200}^-}}}{2 \cdot 10^4 r_{n_2-200}^-} \\
&= \frac{\nu\pr{\Om_{e,\al}}}{2 \cdot 10^4 r_{n_2-200}^-}
\gtrsim N^{-2/100}.
\end{align*}
It follows that 
$$F_e := \set{\al : \pr{e, \al} \in F}   \subset    \set{\al : \mathrm{M} \nu_{e, \De}\pr{\al} \gtrsim N^{-2/100}},$$
and so, by the Hardy-Littlewood maximal inequality, we see that for any $e \in E$,
\begin{align*}
\nu\pr{F_e}
\lesssim \abs{\set{\al : \mathrm{M} \nu_{e, \De}\pr{\al} \gtrsim N^{-2/100}}}
\lesssim N^{2/100} \nu_{e, \De}\pr{\Delta_e},
\end{align*}
where we have used the fact that $\nu$ is the normalized Lebesgue measure restricted to $A$ (see Section~\ref{measObs}).

Integrating in $E$ then shows that
\begin{align*}
\pr{\mu \times \nu}\pr{F} \lesssim N^{2/100} \pr{\mu \times \nu }\pr{\Delta}.
\end{align*}
An application of \eqref{DelBd} completes the proof.
\end{proof}

\smallskip

\section{Completion of the Proof}
\label{conclusion}

To prove Theorem~\ref{QCBT}, we need to show that $\FavC(E) \lesssim N^{-1/100} L$.
Referring to the decomposition of $E$ and $\mathcal{E}$ described by Figure~\ref{trees} and using properties of the Favard curve length, we have that
\begin{align}
\label{FavE}
\FavC(E)
&= \FavC\pr{\mathcal{E}}
\le \FavC\pr{\NCur_{n_2, M_2}}  + \FavC\pr{\Cur_{n_2, M_2}}.
\end{align}
Since $N_\al = \set{e : \pr{e, \al} \in \NCur_{n_2, M_2}} \su E$ and $N_\al = G_\al \cup K_\al$, then by Definition~\ref{FavC} (see also \eqref{3dFavC}),
\begin{equation*}
\begin{aligned}
\FavC\pr{\NCur_{n_2, M_2}} 
&= \int_A \abs{\Phi_\al\pr{N_\al}} d\al
\le \int_A \abs{\Phi_\al\pr{G_\al}} d\al + \int_A \abs{\Phi_\al\pr{K_\al}} d\al \\
&\lesssim \int_A \mu \pr{G_\al} d\al + \int_A \mu\pr{K_\al} d\al,
\end{aligned}
\end{equation*}
where we have applied Lemma~\ref{smallProjProp} in the second line.
In Section~\ref{goodPoints}, we showed via Propositions~\ref{EalProp} and \ref{KalProp} that the sets $G_\al$ and $K_\al$, respectively, have small $\mu$-measures.
Substituting these bounds into the previous inequality and using the boundedness of $A$ shows that
\begin{equation}
\label{NonCirc}
\FavC\pr{\NCur_{n_2, M_2}} \lesssim N^{-1/100} L.
\end{equation}
Using the decomposition of $\Cur_{n_2, M_2} \su \mathcal{E}$ from \eqref{CurDecomp}, then applying Corollary~\ref{smallFavard}, we see that
\begin{equation*}
\begin{aligned}
\FavC\pr{\Cur_{n_2, M_2}} 
&\le \FavC\pr{\widetilde H} + \FavC\pr{\widetilde D} + \FavC\pr{F} \\
&\lesssim \FavC\pr{\widetilde H} + \FavC\pr{\widetilde D} + \pr{\mu \times \nu}\pr{F} \\
&\lesssim N^{-1/100} L + N^{-1/100} L + N^{-1/100} L,
\end{aligned}
\end{equation*}
where we have invoked Propositions~\ref{FavHProp}, \ref{FavDProp}, and \ref{FProp}, respectively, in the last line.
Therefore, the inequality above reduces to
\begin{equation}
\label{circPairs}
\FavC\pr{\Cur_{n_2, M_2}} \lesssim N^{-1/100} L.
\end{equation}
Substituting \eqref{NonCirc} and \eqref{circPairs} into \eqref{FavE} leads to the conclusion of the proof of Theorem~\ref{QCBT}.

\smallskip

\begin{appendix}

\section{Technical Results}
\label{apx}

In this section, we prove that the sets of high multiplicity pairs and high density pairs are closed.  
We also show that $\Delta$ is parametrically closed.


\begin{lem}[$H_n$ is closed]
\label{HnClosed}
Let $H_n$ be as in Definition~\ref{hmSC}. 
For any $1 \le n \le N$, $H_n$ closed.
\end{lem}

\begin{proof}
Fix $n$ and let $\set{\pr{e_m, \al_m}}_{m=1}^\iny \su H_n$ be a sequence of points such that $\pr{e_m, \al_m} \to \pr{e, \al}$.
We need to show that $C_{e,\al}$ is of high multiplicity at a scale index at most $n$. 
Since $\pr{e_m, \al_m} \in H_n$, then the curve $C_m := C_{e_m, \al_m}$ contains $N^{1/100}$ points that are $r_{n}^-$-separated.
Let $e_{m,1}, e_{m, 2}, \ldots, e_{m, N^{1/100}}$ denote the points on the curve $C_m$, where they are ordered so that $e_{m, i}$ is to the left of $e_{m, j}$ whenever $i < j$.
Moreover, for any $i = 1, \ldots, N^{1/100} - 1$, $\abs{e_{m,i+1} - e_{m,i}} \ge r_{n}^-$. 
Since $\set{e_{m,1}}_{m=1}^\iny \su E$ is bounded (since $E$ is compact), then it contains a convergent subsequence, $\set{e_{m(1,k),1}}_{k=1}^\iny$.
Since $\set{e_{m(1,k), 2}}_{k=1}^\iny$ is also bounded, then it too contains a convergent subsequence, $\set{e_{m(2,k), 2}}_{k=1}^\iny$.
Continuing on with this diagonalization process, we extract a subsequence $\set{m_k}_{k=1}^\iny \su \N$, where $m_k = m(N^{1/100}, k)$, so that for every $j = 1, \ldots, N^{1/100}$, $\disp \lim_{k \to \iny}e_{m_k, j} = e_j$. 
Since $E$ is compact, then $e_j \in E$ for each $j$.
We first show that $\set{e_j}_{j=1}^{N^{1/100}}$ is $r_{n}^-$-separated.
Let $\eps > 0$.
There exists $K \in \N$ so that whenever $k \ge K$, we have $\abs{e_{m_k, j} - e_j} < \frac \eps 2$ for any $j \in \set{1, \ldots, N^{1/100}}$.
It follows that for any $j \in \set{1, \ldots, N^{1/100}-1}$,
\begin{align*}
r_{n}^-
&\le \abs{e_{m_K, j+1} - e_{m_K, j}} 
= \abs{e_{m_K,j+1} - e_{j+1} + e_{j+1} - e_{m_K, j} + e_j - e_{j} } \\
&\le \abs{e_{m_K,j+1} - e_{j+1}} + \abs{e_{j+1}  - e_{j} } + \abs{e_{m_K, j} - e_j} 
< \abs{e_{j+1}  - e_{j} } + \eps.
\end{align*}
Since $\eps > 0$ was arbitrary, we conclude that $\abs{e_{j+1}  - e_{j} } \ge r_{n}^-$, showing that $\set{e_j}_{j=1}^{N^{1/100}}$ is $r_{n}^-$-separated.
Finally, since $\pr{e_m, \al_m} \to \pr{e, \al}$, then $C_m \to C_{e, \al}$.
In particular, $C_{m_k} \to C_{e,\al}$.
Since each $e_{m_k, j} \in C_{m_k}$, we deduce that $e_j \in C_{e, \al}$.
It follows that $\set{e_j}_{j=1}^{N^{1/100}} \su C_{e,\al}$, completing the proof.
\end{proof}


\begin{lem}[$D_n$ is closed]
\label{DnClosed}
Let $D_n$ be as in Definition~\ref{hdSS}.
For any $1 \le n \le N$, $D_n$ is closed.  
\end{lem}

\begin{proof}
Fix n and let $\set{\pr{e_m, \al_m}}_{m=1}^\iny \su D_n$ so that $\pr{e_m, \al_m} \to \pr{e, \al}$.
Since $\pr{e_m, \al_m} \in D_n$, then there exists $J_m$ with $\abs{J_m} \ge r_{n}^-$ and $\mu\pr{\Phi_{\al_m,+}^{-1}\pr{J_m}} \ge N^{1/100} \abs{J_m}$.
Write $J_m = \brac{a_m, b_m}$.
Since $\set{a_m}_{m=1}^\iny \su \R$ is bounded (since $E$ is compact), then there exists a convergence subsequence $\set{a_{m(1,k)}}_{k=1}^\iny$.
Similarly, $\set{b_{m(1,k)}}_{k=1}^\iny \su \R$ is bounded, so there is a convergent subsequence $\set{b_{m(2, k)}}_{k=1}^\iny$.
With $m_k = m(2, k)$, both $\set{a_{m_k}}_{k=1}^\iny$ and $\set{b_{m_k}}_{k=1}^\iny$ are convergent sequences in $\R$, with limits $a$ and $b$, respectively.
Define $J = \brac{a, b}$.
Since $b_m - a_m = \abs{J_m} \ge r_{n}^-$ for all $m \in \N$, then taking limits shows that $b - a = \abs{J} \ge r_{n}^-$ as well. 
Taking limits and appealing to continuity also shows that $\mu\pr{\Phi_{\al,+}^{-1}\pr{J}} \ge N^{1/100} \abs{J}$.
In particular, $\Phi_{\al,+}^{-1}\pr{J}$ has high density at scale index $n$.
Since $C_{e_m, \al_m} \su \Phi_{\al_m,+}^{-1}(J_m)$ for each $m \in \N$, then another limiting argument shows that $C_{e, \al} \su \Phi_{\al,+}^{-1}(J)$, completing the proof.
\end{proof}

\begin{lem}[$\Delta$ is parametrically closed]
For $\Delta$ be as defined in \eqref{DeltaDef}, $\Delta$ is parametrically closed.
\end{lem}

\begin{proof}
Since $H$ is parametrically open by definition and $H_{n_0 + N^{-3/100} N}$ is closed (by Lemma~\ref{HnClosed}), and therefore parametrically closed, then $\Delta H := H_{n_0 + N^{-3/100} N} \setminus H$ is parametrically closed.
Similarly, since $D$ is open by definition and $D_{n_2 + N^{-10/100} N}$ is closed (by Lemma~\ref{DnClosed}), then $\Delta D := D_{n_2 + N^{-10/100} N} \setminus D$ is closed, and consequently parametrically closed.
It can be shown (following arguments similar to those used for each $H_n$) that $P_{n_1}$ is closed.
Since $H$ is parametrically open, then $P_{n_1} \setminus H$ is also parametrically closed.
It follows that $\Delta$, the union of three parametrically closed sets, is itself parametrically closed.

\end{proof}

\end{appendix}


\def\cprime{$'$}

\end{document}